\documentclass[11 pt]{amsart}
\title[Derived Isogeny Theory for K3 Surfaces in Positive Characteristic]{Twisted Derived Equivalences and Isogenies between K3 Surfaces in Positive Characteristic}
\author{Daniel Bragg and Ziquan Yang}
\usepackage{merged_macros}

\begin{document}


\maketitle
\setcounter{tocdepth}{1}

\begin{abstract}
    We study isogenies between K3 surfaces in positive characteristic. Our main result is a characterization of K3 surfaces isogenous to a given K3 surface $X$ in terms of certain integral sublattices of the second rational $\ell$-adic and crystalline cohomology groups of $X$. This is a positive characteristic analog of a result of Huybrechts \cite{Huy}, and extends results of \cite{Yang}. We give applications to the reduction types of K3 surfaces and to the surjectivity of the period morphism. To prove these results we describe a theory of B-fields and Mukai lattices in positive characteristic, which may be of independent interest. We also prove some results on lifting twisted Fourier--Mukai equivalences to characteristic 0, generalizing results of Lieblich and Olsson \cite{LO1}.
\end{abstract}

\tableofcontents

\section{Introduction}

The purpose of this paper is to study twisted Fourier-Mukai partners of K3 surfaces in positive characteristics and to develop an isogeny theory for these surfaces which is analogous to that of abelian varieties. 

Let $k$ be an algebraically closed field and $p$ be a prime number. When $\mathrm{char\,} k = p$, we simply write $W$ for the ring of Witt vectors $W(k)$. Let $\what{\bZ}^p$ denote the prime to $p$ part of $\what{\bZ}$. For a variety $Y$ over $k$, we set $\H^*(Y) := \H^*_\et(Y, \what{\bZ})$ if $\mathrm{char\,} k = 0$, and $\H^*(Y) := \H^*_\et(Y, \what{\bZ}^p) \times \H^*_\cris(Y/W)$ if $\mathrm{char\,} k = p$, and write $\H^*(Y)_\bQ:=\H^*(Y)\otimes_{\bZ}\bQ$.
\begin{definition}\emph{\upshape{(cf. \cite[Def.~1.1]{Yang})}}
Let $X, X'$ be K3 surfaces over $k$. An isogeny $f : X \isog X'$ is a correspondence, i.e., a $\bQ$-linear combination of algebraic cycles on $X \times X'$, such that the induced action $\H^2(X')_\bQ \to \H^2(X)_\bQ$ is an isomorphism which preserves the Poincar\'e pairing. Two isogenies are deemed equivalent if they induce the same map $\H^2(X')_\bQ \sto \H^2(X)_\bQ$.
\end{definition} 

Our main results concern the existence and uniqueness of isogenies with prescribed cohomological action. We begin with the former. A natural source for isogenies between K3 surfaces is provided by twisted Fourier-Mukai equivaleces: For a K3 surface $X$ and Brauer class $\alpha \in \Br(X)$, we denote by $D^b(X, \alpha)$ the bounded derived category of $\alpha$-twisted sheaves. Given another K3 surface $X'$ and Brauer class $\alpha'$, an equivalence $D^b(X, \alpha) \iso D^b(X', \alpha')$ induces, up to some choices, an isogeny $f : X \isog X'$. We call isogenies which arise this way \textbf{primitive derived isogenies}, and compositions of such isogenies \textbf{derived isogenies}. The precise definitions are given in \S\ref{ssec:rational chow motives}. There we also give a motivic reformulation of the above definition, which will be used for the rest of the paper. 

To state our theorems, we denote the K3 lattice $\mathbf{U}^{\oplus 3} \oplus \mathbf{E}_8^{\oplus 2}$ by $\Lambda$ and recall Ogus' notion of K3 crystals \cite[Def.~3.1]{Ogus}. Here $\mathbf{U}$ denotes the standard hyperbolic plane and $\mathbf{E}_8$ denotes the unique unimodular even negative definite lattice of rank $8$. Our first theorem is an existence result on derived isogenies: 

\begin{theorem}
\label{thm: existence}
Assume $\mathrm{char\,} k = p \ge 5$. Let $X$ be a K3 surface over $k$. Endow $\Lambda \tensor W$ with a K3 crystal structure and denote it by $H_p$ and let $H^p$ denote $\Lambda \tensor \what{\bZ}^p$. 

Let $\iota : H^p \times H_p \into \H^2(X)_\bQ$ be an isometric embedding which respects the Frobenius actions on $H_p$ and $\H^2_\cris(X/W)[1/p]$. There exists a derived isogeny $f : X \isog X'$ to another K3 surface $X'$ such that $f^*(\H^2(X')) = \mathrm{im}(\iota)$ if and only if $\iota$ sends the slope $< 1$ part of $H_p$ isomorphically onto that of $\H^2_\cris(X/W)$. 
\end{theorem}
We refer the reader to Rmk~\ref{rmk: relax assumption} for the reason to restrict to $p \ge 5$. The above result is inspired by a theorem of Huybrechts \cite[Thm~0.1]{Huy}, which can be stated as follows in our terminology: 
\begin{theorem}\label{thm:Huybrechts theorem}
\emph{(Huybrechts)}
Let $X, X'$ be two K3 surfaces over $\bC$. Every isomorphism of Hodge structures $\H^2(X', \bQ) \sto \H^2(X, \bQ)$ which preserves the Poicar\'e pairings is induced by a derived isogeny $f : X \isog X'$.
\end{theorem}
This refines an earlier theorem of Buskin \cite[Thm~1.1]{Buskin}, which affirms a conjecture of Shafarevich's. Using the global Torelli theorem and surjectivity of the period map, one checks that Huybrechts' theorem is equivalent to an existence theorem for isogenies: For every K3 surface $X$ over $\bC$ and every isometric embedding $\iota: \Lambda \into \H^2(X, \bQ)$, there exists another K3 surface $X'$ over $\bC$ and a derived isogeny $f : X \isog X'$ such that $f^*(\H^2(X', \bZ)) = \mathrm{im}(\iota)$ (see \S\ref{ssec: Huybrecht's theorem algebraic version in char 0}). Note that this statement does not involve Hodge structures. Our Thm~\ref{thm: existence} is a positive characteristic analog for this version of Huybrechts' theorem. 

Huybrechts' refinement shows in particular that every isogeny between K3 surfaces over $\bC$ is equivalent to a derived isogeny. In contrast, the ``only if'' part of Thm~\ref{thm: existence} implies that the cohomological actions of derived isogenies in characteristic $p$ obey a certain nontrivial constraint at $p$. In particular, not every isogeny is equivalent to a derived isogeny. Given this, it is of interest to characterize also the possible cohomological actions of all (not necessarily derived) isogenies. The following result shows that, under some technical assumptions, the ``if'' part of Thm~\ref{thm: existence} can be removed for $k = \bar{\bF}_p$ if one is willing to consider all isogenies: 

\begin{theorem}
\label{thm: crys-isog}
Let $X, H_p, H^p$ and $\iota$ be as in Thm~\ref{thm: existence}. If $k = \bar{\bF}_p$ and 
\begin{enumerate}[label=\upshape{(\alph*)}]
    \item $\Pic(X)$ has rank $\ge 12$ or contains a standard hyperbolic plane, or
    \item $\Pic(X)$ contains an ample line bundle $L$ of degree $L^2 < p - 4$,
\end{enumerate}
then there exists another K3 surface $X'$ over $k$ and an isogeny $f : X \isog X'$ such that $f^*(\H^2(X')) = \mathrm{im}(\iota)$. 
\end{theorem}
This is a strengthening of \cite[Thm~1.4]{Yang}. We mention that a byproduct in the course of proving the above is a generalization (Thm~\ref{thm: char NO lifting}) of Taelman's characterization of the canonical liftings of ordinary K3 surfaces \cite[Thm~C]{TaelmanOrd}. In \cite{NO}, Nygaard and Ogus constructed for every non-supersingular K3 surface $X$ a ``section'' to the natural morphism $\mathrm{Def}(X) \to \mathrm{Def}({\what{\Br}_X})$ from the deformation space of $X$ to that of its formal Brauer group, such that a lifting of $\what{\Br}_X$ induces a lifting of $X$. We call liftings of $X$ which arise this way ``Nygaard-Ogus liftings''. When $X$ is ordinary, a Nygaard-Ogus lifting is the same as a canonical lifting. Thm~\ref{thm: char NO lifting} gives an integral $p$-adic Hodge theoretic characterization of Nygaard-Ogus liftings. See \S\ref{sec: NO revisited} for details.

We now describe our uniqueness results. We recall some terminology from \cite[\S6]{Yang}: An isogeny $f : X \isog X'$ between K3 surfaces is said to be \textbf{polarizable} if the induced map $\Pic(X')_\bQ \sto \Pic(X)_\bQ$ sends an ample class to another ample class, and \textbf{$\bZ$-integral} if the induced isomorphism $\H^2(X')_\bQ \sto \H^2(X)_\bQ$ restricts to an isomorphism $\H^2(X') \sto \H^2(X)$. We prove the following Torelli theorem for derived isogenies: 

\begin{theorem}
\label{thm: Torelli}
Assume $\mathrm{char\,} k \ge 5$. Let $X, X'$ be K3 surfaces over $k$. A derived isogeny $f : X \isog X'$ is equivalent to the graph of an isomorphism $X' \sto X$, if and only if $f$ is polarizable and $\bZ$-integral. 
\end{theorem}

Finally, we remark that Li and Zou considered derived isogenies and Torelli type theorems for abelian surfaces in a recent preprint \cite{LiZou}. 

\subsection{Applications to good reductions of K3 surfaces}
We apply our results to study the good reduction conjecture for K3 surfaces: 

\begin{conjecture}
\label{conj: GR}
Let $k$ be an algebraically closed field of characteristic $p > 0$ and let $F$ be a finite extension of $W[1/p]$. Let $X_F$ be a K3 surface over $F$ such that $\H^2_\et(X_{\bar{F}}, \bQ_\ell)$ is unramified for some prime $\ell \neq p$. Then, $X_F$ has potentially good reduction. 
\end{conjecture}

This conjecture is a K3 analogue of the N\'eron-Ogg-Shafarevich criterion for abelian varieties. It admits many variants (e.g., ones that concern semistable reductions) and is verified in cases when $X_F$ admits a polarization of low degree (see \cite{Matsumoto} and \cite{LMGR}). We prove the following: 

\begin{theorem}
\label{thm: intro GR}
Let $X_F$ be as in Conj.~\ref{conj: GR}. Assume that $p > 2$ and that $X_F$ admits a line bundle of degree prime to $p$. Then the $\Gal_F$-representation $\H^2_\et(X_{\bar{F}}, \bQ_p)$ is potentially crystalline. If $p \ge 5$ (resp. $p > 2$) and $\H^2_\et(X_{\bar{F}}, \bQ_p)$ has potentially good ordinary or (resp. supersingular) reduction, then so does $X_F$. 
\end{theorem}

Roughly speaking, the theorem is saying that if the cohomology of $X_F$ predicts that $X_F$ should have potential ordinary or supersingular reduction, then it does. We derive this as a consequence of a more general result (Thm~\ref{thm: GR}), which essentially reduces Conj. \ref{conj: GR} to the Hecke orbit conjecture (see Conj. \ref{conj: HO for i}), which is a purely Shimura--theoretic statement. In particular, we prove the following.

\begin{theorem}
    Let $X_F$ be an in Conj.~\ref{conj: GR}. Suppose that $p>2$ and that $X_F$ admits a line bundle of degree prime to $p$. Assume the Hecke orbit conjecture (Conj. \ref{conj: HO for i}) holds all $i$. Then, $X_F$ has potentially good reduction.
\end{theorem}

Our unconditional Thm \ref{thm: intro GR}, in the ordinary case, is then a consequence of recent work of Maulik--Shankar--Tang  \cite[Thm~1.4]{maulik2020picard} proving the Hecke orbit conjecture in certain special cases. The supersingular case will be treated by a slightly different argument. Moreover, it seems very likely that a slight generalization of the conjecture can remove the condition on the existence of a prime-to-$p$ line bundle as well, and hence completely affirms Conj.~\ref{conj: GR}.

We remark that nowhere in the proofs of the above results do we directly analyze a degeneration of K3 surfaces, unlike in \cite{Matsumoto} and \cite{LMGR}. In particular, we avoid the use of any techniques from the minimal model program. As far as the authors are aware, our method of proving good reduction results by marrying moduli theory of sheaves with density arguments is new in the literature.

\subsection{Ideas of proof}
(1) The ``only if'' part of Thm~\ref{thm: existence} follows from the general theory of twisted derived equivalences in positive characteristics. The idea for the ``if'' part is to construct the desired $X'$ together with the isogeny $f : X \isog X'$ by iteratively taking moduli spaces of twisted sheaves on $X$. This approach is inspired by that of \cite[Thm 1.1]{Huy}. A key technical tool is the theory of B-fields in $\ell$-adic and crystalline cohomology, described in \S\ref{sec:twisted Mukai lattices}. This allows us to relate classes in $\H^2(X)_{\bQ}$ to the Brauer group, and provides a replacement for the Hodge-theoretic B-fields in Huybrecht's proof, although there are some additional complications at $p$. There are some further technical difficulties caused by the fact that in positive characteristic the cohomology $\H^2(X)_\bQ$ can only take on adelic coefficients (i.e., $\bA^p_f \times W[1/p]$) instead of $\bQ$-coefficients.
For instance, the Mukai vector which one must specify in order to form a moduli of sheaves is not an adelic object. That is, unlike Brauer classes, one cannot specify a Mukai vector by prescribing its local factors in $\H^2(X)_\bQ$. We solve these problems by using local-global type results on quadratic forms (e.g., the strong approximation theorem), and the theory of quadratic forms over local rings.\\\\
(2) Thm~\ref{thm: crys-isog} is obtained by the realizing $X'$ as the reduction of a suitable K3 surface in characteristic zero. This strategy is a simultaneous simplification and strengthening of that of \cite{Yang}, with the additional input of Thm~\ref{thm: existence}. The characterization of Nygaard-Ogus liftings (Thm~\ref{thm: char NO lifting}) is obtained by applying recent advance on integral $p$-adic Hodge theory from \cite{BMS} and \cite{CaisLiu} to study deformations of K3 crystals. These techniques for handling crystalline cohomology were unnecessary in Taelman's case \cite{TaelmanOrd}, as the deformation of the formal Brauer group of an ordinary K3 is rigid, which is not true for a general finite height K3. We remark that here the restriction $p \ge 5$ is mainly due to our usage of the deformation theory of K3 crystals.   \\\\
(3) Thm~\ref{thm: Torelli} is a twisted generalization of the derived Torelli theorem of Lieblich-Olsson \cite[Thm 6.1]{LO1}. Just as in \textit{loc. cit.}, we prove this result by using a lifting argument to reduce to the global Torelli theorem over $\bC$. The main difficulty which arises in our generalization is that instead of considering isogenies which arise directly from a (twisted or untwisted) derived equivalence, we are allowing any finite compositions of such. The derived equivalences involved may not be simultaneously liftable to characteristic zero. To overcome this difficulty, we combine the lifting results on derived equivalences with the Kuga-Satake method. This helps us reduce composing isogenies of K3's to composing isogenies abelian varieties, which is much better understood. There is a technical problem which arises from the usage of Kuga-Satake. Namely, we need to put the relevant K3 surfaces into the same moduli space. However, the K3 surfaces themselves may not have a quasi-polarization of a common degree. To overcome this problem, we pass from K3 surfaces to their Hilbert squares, which are treated in \cite{Yang2}. The restriction to $p \ge 5$ is imposed because in \textit{loc. cit.} the second author only treated $\mathrm{K3}^{[n]}$-type varieties when $p > n + 1$ for certain technical reasons.  \\\\
(4) For Thm~\ref{thm: intro GR}, we first show that the derived prime-to-$p$ isogeny classes of K3's match up with the notion of prime-to-$p$ Hecke orbit on the period domains of Kuga-Satake morphisms, which are some orthogonal Shimura varieties. It follows from some intermediate steps in the proof of Thm~\ref{thm: existence} that the property of satisfying Conj.~\ref{conj: GR} is invariant in a prime-to-$p$ derived isogeny class. On the other hand, any $X_F$ which satisfies the hypothesis of Thm~\ref{thm: intro GR} produces a mod $p$ point $x(X_F)$ on the period domain, and the set $\sL_{\mathrm{bad}} := \{ x(X_F) : X_F \textit{ violates Conj.~\ref{conj: GR}} \}$ is closed. 

If we combine the above observations with the Hecke orbit (HO) conjecture (see Conj.~\ref{conj: HO for i}), we see that if $\sL_{\mathrm{bad}}$ intersects any of the height stratum of the period domains, then it must contain the entirety of that stratum, which is false by a deformation argument. Hence the HO conjecture forces $\sL_{\mathrm{bad}}$ to be empty. The HO conjecture is now known for the ordinary locus by the recent work of Maulik-Shankar-Tang \cite{maulik2020picard} and we will verify it in the superspecial locus for cases relevant to us (Thm~\ref{thm: Hecke orbit for superspecial}). This gives Thm~\ref{thm: intro GR}.

\subsection{Plan of Paper:} In \S\ref{sec:twisted Mukai lattices}, we develop the formalism of B-fields and twisted Mukai lattices in positive characteristic. \S\ref{sec:twisted Chern characters and action on cohomology} concerns the construction of twisted Chern characters, the twisted N\'{e}ron--Severi lattice, and the action of a twisted derived equivalence on cohomology. In \S\ref{ssec:rational chow motives} we discuss rational Chow motives and isogenies. In \S\ref{sec:lifting isogenies} we prove some lifting results for twisted derived isogenies. In \S\ref{sec: Existence}, we first prove Thm~\ref{thm: existence}. We then revisit Nygaard-Ogus' theory for the point of view of integral $p$-adic Hodge theory and prove Thm~\ref{thm: crys-isog}. In \S\ref{sec: uniqueness}, we review the basics of Hilbert squares and the Kuga-Satake period morphism, and then prove Thm~\ref{thm: Torelli}. Finally, in \S\ref{sec: Hecke}, we explain the relationship between our isogeny theory and Hecke orbits, and prove Thm~\ref{thm: intro GR}.

\subsection{Notation:}\label{ssec:notation}
\begin{itemize}
    \item Let $p$ denote a prime. The letter $k$ denotes a perfect base field of characteristic either $0$ or $p$ and $\ell$ denotes a prime not equal to $\mathrm{char\,}k$. When $\mathrm{char\,} k = p$, we write $W$ for $W(k)$ and $K$ for $W[1/p]$.
    \item If $Z$ is a scheme, we write $\H^i(Z,\mu_n)$ for the flat (fppf) cohomology of the sheaf of $n$th roots of unity on $Z$. If $n$ is coprime to the characteristics of all residue fields of $Z$, this is equal to the \'{e}tale cohomology of $\mu_n$.
    \item We normalize our Chern characters so that the mod $m$ Chern character of a line bundle $L$ is equal to the image of the class of $L$ under the boundary map $\H^1(Z,\mathbf{G}_m)\to \H^2(Z,\mu_m)$ of the Kummer sequence. 
    \item Suppose $k$ is a perfect field of characteristic $p$ and $S$ is an $k$-scheme. If $f : X \to S$ is a scheme, we denote by $\H^j_\cris(X)$ the sheaf on $\Cris(S/W)$ given by $R^j f_{\cris *} \mathscr{O}_{X/W}$ when $S$ is understood. 
    \item For any integral domain $R$, and $R$-modules $M, N$, an isomorphism $f : M_\bQ \sto N_\bQ$ is said to be $R$-integral if $f(M) = N$. 
    \item In this paper we only make use of singular, de Rham, \'etale, flat, and crystalline cohomology.  We may omit the subscripts ``$\cris$'', ``$\fl$'', or ``$\dR$'', when the choice of the relevant Grothendieck topology is clear from the coefficients.
    \item For a smooth proper variety $Y$ over $k$, we let $\H^j(Y)$ denote either $\H^j_\et(Y, \what{\bZ})$ if $\mathrm{char\,} k = 0$ or $\H^j_\et(Y, \what{\bZ}^p) \times \H^j_\cris(Y/W)$ if $\mathrm{char\,} k = p$.  
    \item Let $R$ be a commutative ring. A quadratic lattice $M$ over $R$ is a free $R$-module of finite rank equipped with a bilinear symmetric pairing $M \times M \to R$. The pairing is said to be non-degenerate (resp. unimodular or perfect) if the induced map $M \to M^\vee$ is an injection (resp. an isomorphism). 
\end{itemize}

\subsection{Acknowledgements} The first author was supported by NSF Postdoctoral Research Fellowship DMS-1902875. The second author would like to thank Kai Xu and Yuchen Fu for their mental support and helpful discussions and Zhiyuan Li for his interest in the work. 

\section{B--fields and the twisted Mukai lattice in positive characteristic}\label{sec:twisted Mukai lattices}

Let $X$ be a K3 surface over the complex numbers. Associated to $X$ is the \textit{Mukai lattice} $\widetilde{\H}(X,\mathbf{Z})$, which is the direct sum of the singular cohomology groups of $X$ equipped with a certain pairing and Hodge structure. Consider a class $\alpha\in\Br(X)$. Huybrechts and Stellari generalized Mukai's construction to the twisted K3 surface $(X,\alpha)$ by defining the \textit{twisted Mukai lattice} $\widetilde{\H}(X,B,\mathbf{Z})$ \cite[Rmk 1.3]{HS}. This construction modifies the Hodge structure on the Mukai lattice in a certain way using an auxillary choice of a \textit{B-field lift} of $\alpha$, which is a class $B\in\H^2(X,\mathbf{Q})$ whose image in $\Br(X)$ under the exponential map is equal to $\alpha$.

Suppose now that $X$ is a K3 surface defined over an algebraically closed field of characteristic $p>0$. After \cite{LO1}, we may consider the $\ell$-adic and crystalline realizations of the Mukai motive of $X$. These are respectively a $\mathbf{Z}_l$-lattice $\widetilde{\H}(X,\mathbf{Z}_l)$ and a $W$-lattice $\widetilde{\H}(X/W)$, both of rank 24. 
In the crystalline setting, $\widetilde{\H}(X/W)$ is equipped with a Frobenius action, which makes $\widetilde{\H}(X/W)$ into a K3 crystal in the sense of Ogus \cite[Def.~3.1]{Ogus}. That this construction makes sense integrally is first observed in \cite{BL}.

Consider a Brauer class $\alpha\in\Br(X)$. We wish to have an analog of Huybrechts and Stellari's construction of the twisted Mukai lattice in both the $\ell$-adic and crystalline settings. The main task is to find the appropriate analog of a B-field lift of a Brauer class in $\ell$-adic or crystalline cohomology. The $\ell$-adic case is considered in \cite{LMS} (we remark that the authors also deal with some additional complications coming from working over a non algebraically closed field, which we ignore here).
The crystalline case is considered in \cite[\S 3]{Bragg-Derived-Equiv} and \cite[\S 3.4]{BL}, with the restriction that the Brauer class $\alpha$ is killed by $p$ (rather than a power of $p$).

In this section we make two contributions. First, we complete the crystalline realization by defining crystalline B-field lifts of classes killed by an arbitrary power of $p$. We then treat the mixed case, considering all primes simultaneously, and define mixed B-field lifts of Brauer classes whose order is divisible by more than one prime. To assist the reader in connecting these constructions in the Hodge, $\ell$-adic, and crystalline settings, we have included a brief summary of the Hodge and $\ell$-adic realizations. We have tried to present a perspective which emphasizes the unifying features present in the different settings.

\subsection{Hodge realization}

Let $X$ be a K3 surface over the complex numbers. We have the exponential exact sequence
\[
    0\to\mathbf{Z}\to\mathscr{O}_X\xrightarrow{\exp}\mathscr{O}_X^{\times}\to 1
\]
Consider the induced map $\H^2(X,\mathscr{O}_X)\xrightarrow{\exp}\H^2(X,\mathscr{O}_X^{\times})$, which because $\H^3(X,\mathbf{Z})=0$ is a surjection. Given a class $v\in\H^2(X,\mathscr{O}_X)$, we note that $\exp(v)$ is contained in the torsion subgroup $\H^2(X,\mathscr{O}_X^{\times})_{\tors}=\H^2(X,\mathbf{G}_m)=\Br(X)$ if and only if $v$ is contained in the subgroup $\H^2(X,\mathbf{Q})\subset\H^2(X,\mathscr{O}_X)$. Thus, this map restricts to a surjection
\begin{equation}\label{eq:B field map Hodge case}
    \exp:\H^2(X,\mathbf{Q})\to\Br(X)
\end{equation}
which we denote by $B\mapsto \alpha_B=\exp(B)$. According to \cite{HS}, a \textit{B-field lift} of a class $\alpha\in\Br(X)$ is a class $B\in\H^2(X,\mathbf{Q})$ such that $\alpha_B=\alpha$.



The relationship between B-fields and the Brauer group is expressed in the diagram
\begin{equation}\label{eq:big ol diagram Hodge}
    \begin{tikzcd}
        &&0\arrow{d}&0\arrow{d}&\\
        0\arrow{r}&\H^2(X,\mathbf{Z})\arrow[equals]{d}\arrow{r}&\H^2(X,\mathbf{Z})+\Pic(X)\otimes\mathbf{Q}\arrow{d}\arrow{r}&\Pic(X)\otimes(\mathbf{Q}/\mathbf{Z})\arrow{d}\arrow{r}&0\\
        0\arrow{r}&\H^2(X,\mathbf{Z})\arrow{r}&\H^2(X,\mathbf{Q})\arrow{d}\arrow{r}&\H^2(X,\mathbf{Z})\otimes(\mathbf{Q}/\mathbf{Z})\arrow{r}\arrow{d}&0\\
        &&\dfrac{\H^2(X,\mathbf{Q})}{\H^2(X,\mathbf{Z})+\Pic(X)\otimes\mathbf{Q}}\arrow{r}{\sim}\arrow{d}&\Br(X)\arrow{d}&\\
        &&0&0&
    \end{tikzcd}
\end{equation}
with exact rows and columns. In particular, we see that there are two sources of ambiguity in choosing a B-field lift of a Brauer class, namely, integral classes in $\H^2(X,\mathbf{Z})$ and rational classes in $\H^{1,1}(X,\mathbf{Q})=\Pic(X)\otimes\mathbf{Q}\subset\H^2(X,\mathbf{Q})$.
 
\subsection{$\ell$--adic realization}\label{sec:etale}

Let $k$ be an algebraically closed field of arbitrary characteristic. Fix a prime number $\ell$, not equal to the characteristic of $k$. We review the $\ell$-adic B-fields and the $\ell$-adic realization of the twisted Mukai motive introduced in \cite{LMS}.

Let $X$ be a K3 surface over $k$. By duality in \'{e}tale cohomology, we have $\H^3(X,\mu_{\ell^n})=0$ for all $n\geq 1$. It follows that natural map
\begin{equation}\label{eq:etale natural quotient map}
  \H^2(X,\bZ_\ell(1))\to\H^2(X,\mu_{\ell^n})
\end{equation}
is surjective, and hence we have an identification
\[
  \H^2(X,\bZ_\ell(1))\otimes\bZ/\ell^n\bZ\cong\H^2(X,\mu_{\ell^n})
\]
We consider the composition
\begin{equation}\label{eq:etalediagram}
  \begin{tikzcd}
    \H^2(X,\bZ_\ell(1))\arrow[two heads]{r}&\H^2(X,\mu_{\ell^n})\arrow[two heads]{r}&\Br(X)[\ell^n]
  \end{tikzcd}
\end{equation}
where the second map is induced by the inclusion $\mu_{\ell^n}\subset\mathbf{G}_m$.
\begin{definition}\label{def:ladicBfield}
  Let $\alpha\in\Br(X)$ be a Brauer class which is killed by a power of $\ell$. An $\ell$-\textit{adic B-field lift} of $\alpha$ is an element 
    \[
      B\in\H^2(X,\bQ_\ell(1))\defeq\H^2(X,\bZ_\ell(1))\otimes_{\bZ_\ell}\bQ_\ell
  \]
  such that if we write $B=a/ \ell^n$ for some $a\in\H^2(X,\bZ_\ell(1))$, then $a$ maps to $\alpha$ under the composition~\eqref{eq:etalediagram}.
\end{definition}

We give the following alternative description. Define $\mu_{\ell^{\infty}}=\bigcup_n\mu_{\ell^n}\subset\mathbf{G}_m$. The Picard group of $X$ is torsion free, which implies the vanishing $\H^1(X,\mu_{\ell})=0$. It follows that the inclusions $\mu_{\ell^n}\subset\mu_{{\ell}^{n+1}}$ induce injections on $\H^2$, and we have a natural identification $\H^2(X,\mu_{\ell^{\infty}})=\bigcup_n\H^2(X,\mu_{\ell^n})$. Moreover, for every $n$ we have a commutative diagram
\begin{equation}\label{eq:compatible square}
    \begin{tikzcd}
        \H^2(X,\mathbf{Z}_\ell(1))\arrow{r}{\cdot \ell^m}\arrow[two heads]{d}&\H^2(X,\mathbf{Z}_\ell(1))\arrow[two heads]{d}\\
        \H^2(X,\mu_{\ell^n})\arrow[hook]{r}&\H^2(X,\mu_{\ell^{n+m}})
    \end{tikzcd}
\end{equation}
Taking the direct limit of the maps~\eqref{eq:etale natural quotient map}, we get a map
\begin{equation}\label{eq:B field map etale case}
    \H^2(X,\mathbf{Q}_\ell(1))\to\H^2(X,\mu_{\ell^{\infty}})
\end{equation}
This map may be explicitly described as follows: given $B\in\H^2(X,\mathbf{Q}_\ell(1))$, choose $n\geq 0$ such that $\ell^nB\in\H^2(X,\mathbf{Z}_\ell(1))$, and map $B$ to the image of $\ell^nB$ under the left map of~\eqref{eq:etalediagram}. Note that by the commutativity of~\eqref{eq:compatible square}, this association is well defined, independent of our choice of $n$. Composing~\eqref{eq:B field map etale case} with the natural map $\H^2(X,\mu_{\ell^{\infty}})\to\Br(X)$, we get a map
\begin{equation}\label{eq:B field map etale case, part 2}
    \H^2(X,\mathbf{Q}_\ell(1))\to \Br(X)
\end{equation}
This is the $\ell$-adic analog of the exponential map~\eqref{eq:B field map Hodge case}. The image of this map is exactly the subgroup $\Br(X)[\ell^{\infty}]\subset\Br(X)$ consisting of classes killed by some power of $\ell$. Furthermore, an $\ell$-adic B-field lift of a class $\alpha\in\Br(X)[\ell^{\infty}]$ (in the sense of Definition \ref{def:ladicBfield}) is exactly a preimage of $\alpha$ under~\eqref{eq:B field map etale case, part 2}. We denote~\eqref{eq:B field map etale case, part 2} by $B\mapsto \alpha_B$.

The relationship between $\ell$-adic B-fields and the Brauer group is expressed by the diagram
\begin{equation}\label{eq:big ol diagram}
    \begin{tikzcd}
        &&0\arrow{d}&0\arrow{d}&\\
        0\arrow{r}&\H^2(X,\mathbf{Z}_\ell(1))\arrow[equals]{d}\arrow{r}&\H^2(X,\mathbf{Z}_\ell(1))+\Pic(X)\otimes\mathbf{Q}_\ell\arrow{d}\arrow{r}&\Pic(X)\otimes(\mathbf{Q}_\ell/\mathbf{Z}_\ell)\arrow{d}\arrow{r}&0\\
        0\arrow{r}&\H^2(X,\mathbf{Z}_\ell(1))\arrow{r}&\H^2(X,\mathbf{Q}_\ell(1))\arrow{d}\arrow{r}&\H^2(X,\mu_{\ell^\infty})\arrow{r}\arrow{d}&0\\
        &&\dfrac{\H^2(X,\mathbf{Q}_\ell(1))}{\H^2(X,\mathbf{Z}_\ell(1))+\Pic(X)\otimes\mathbf{Q}_\ell}\arrow{r}{\sim}\arrow{d}&\Br(X)[\ell^{\infty}]\arrow{d}&\\
        &&0&0&
    \end{tikzcd}
\end{equation}
with exact rows and columns, where the right hand column is given by taking the direct limit of the exact sequence induced by the Kummer sequence.

In particular, we have an isomorphism
\begin{equation}\label{eq:Brauer group crystalline case finite height}
    \Br(X)[\ell^\infty]\cong\left(\mathbf{Q}_\ell/\mathbf{Z}_\ell\right)^{\oplus 22-\rho}
\end{equation}
where $\rho$ is the Picard rank of $X$.

\subsection{The twisted $\ell$-adic Mukai lattice}
The $\ell$\textit{-adic Mukai lattice} associated to $X$ \cite[Definition 3.3.1]{LMS} is
\[
    \tH(X,\mathbf{Z}_\ell)=\H^0(X,\mathbf{Z}_\ell)(-1)\oplus\H^2(X,\mathbf{Z}_\ell)\oplus\H^4(X,\mathbf{Z}_\ell)(1)
\]
which we equip with the Mukai pairing. Given a class $B\in\H^2(X,\mathbf{Q}_\ell)$, we define the associated \textit{twisted} $\ell$\textit{-adic Mukai lattice} to be the submodule
\[
    \tH(X,\mathbf{Z}_\ell,B)=\exp(B)\tH(X,\mathbf{Z}_\ell)\subset\tH(X,\mathbf{Q}_\ell)
\]
Here, $\exp(B)$ denotes the isometry $\tH(X,\mathbf{Q}_\ell)\to\tH(X,\mathbf{Q}_\ell)$ given by
\begin{equation}\label{eq:exp B}
    (a,b,c)\mapsto (a,b+aB,c+b.B+a\frac{B^2}{2})
\end{equation}

\subsection{Crystalline realization}\label{sec:twistedk3crystals}
Let $k$ be an algebraically closed field of characteristic $p>0$ and let $X$ be a K3 surface over $k$. We will define crystalline B-fields associated to Brauer classes on $X$ whose order is a power of $p$. There are some new phenomena which present themselves in the crystalline setting that are not present in the Hodge and $\ell$--adic theories. In particular, there is a nontrivial interaction between crystalline B-fields and the Frobenius operator on the crystalline cohomology. A related feature is that not every class in rational crystalline cohomology is a crystalline B-field. We give a characterization of which classes are B-fields using only the $F$--crystal structure on crystalline cohomology in Prop. \ref{prop:description1}. We then construct the crystalline version of the twisted Mukai lattice, and show that this object has a natural structure of a K3 crystal in the sense of Ogus \cite[Def.~3.1]{Ogus}. We conclude with some calculations with the twisted Mukai crystals. In the special case when the Brauer class is killed by $p$, the results of this section have appeared in \cite{Bragg-Derived-Equiv,BL}.

Set $W_n=W/p^nW$, so in particular $W_1=k$. Let $\sigma\colon k\to k$ be the Frobenius $\lambda\mapsto \lambda^p$. We denote the induced map $\sigma\colon W\to W$ (abusively) by the same symbol.

\subsection{Crystalline B-fields}
We begin by relating the flat cohomology of $\mu_{p^n}$ to certain \'{e}tale cohomology groups. Consider the Kummer sequence
\[
  1\to\mu_{p^n}\to\mathbf{G}_{m}\xrightarrow{x\mapsto x^{p^n}}\mathbf{G}_m\to 1
\]
which is exact in the fppf topology. Let $\varepsilon\colon X_{\fl}\to X_{\etale}$ be the natural map from the big fppf site of $X$ to the small \'{e}tale site of $X$. By a theorem of Grothendieck, the cohomology of the complex $R\varepsilon_*\mathbf{G}_m$ vanishes in all positive degrees. Applying $\varepsilon_*$ to the Kummer sequence, we obtain an exact sequence
\[
  1\to\mathbf{G}_m\xrightarrow{x\mapsto x^{p^n}}\mathbf{G}_m\to R^1\varepsilon_*\mu_{p^n}\to 1
\]
of sheaves on the small \'{e}tale site of $X$ (because $X$ is reduced, the restriction of $\mu_{p^n}$ to the small \'{e}tale site of $X$ is trivial). It follows that
\[
  R^1\varepsilon_*\mu_{p^n}=\mathbf{G}_m/\mathbf{G}_m^{\times p^n}
\]
where the quotient is taken in the \'{e}tale topology. We therefore obtain isomorphisms
\begin{equation}\label{eq:flattoetale}
  \H^i(X_{\fl},\mu_{p^n})\iso\H^{i-1}(X_{\etale},\mathbf{G}_m/\mathbf{G}_m^{\times p^n})
\end{equation}
We next relate the \'{e}tale cohomology groups on the right to crystalline cohomology. We consider the map of \'{e}tale sheaves
\[
  d\log\colon\mathbf{G}_m\to \W_n\Omega^1_X
\]
given by $x\mapsto d\underline{x}/\underline{x}$, where $\underline{x}=(x,0,0,\dots)$ is the multiplicative representative of $x$ in $\W_n\cO_X$. By \cite[Prop. I.3.23.2, page 580]{Illusie} the kernel of $d\log$ is equal to the subsheaf $\mathbf{G}_m^{\times p^n}\subset\mathbf{G}_m$, so there is an induced injection
\begin{equation}\label{eq:dlog2}
  d\log\colon\mathbf{G}_m/\mathbf{G}_m^{\times p^n}\hto \W_n\Omega^1_X
\end{equation}
As the image of $d\log$ is contained in the kernel of $d$, we have a commutative diagram
\[
    \begin{tikzcd}
        0\arrow{d}\arrow{r}&\W_n\ms O_X\arrow{d}{d}\\
        \mathbf{G}_m/\mathbf{G}_m^{\times p^n}\arrow{r}{d\log}\arrow{d}&\W_n\Omega^1_X\arrow{d}{d}\\
        0\arrow{r}&\W_n\Omega^2_X
    \end{tikzcd}
\]
which we interpret as a map of complexes
\begin{equation}\label{eq:dlog complexes}
  d\log\colon\mathbf{G}_m/\mathbf{G}_m^{\times p^n}[-1]\hto \W_n\Omega^{\bullet}_X. 
\end{equation}

An important fact is that the de Rham-Witt complex computes crystalline cohomology, in the sense that there is a canonical isomorphism
\begin{align}\label{eq:identification}
  \H^{\ast}(X,\W_n\Omega^{\bullet}_X)&\iso\H^{\ast}(X/W_n)
\end{align}
in each degree \cite[Th\'{e}or\'{e}me II.1.4, page 606]{Illusie}. Taking cohomology of~\eqref{eq:dlog complexes} and using the identifications \eqref{eq:flattoetale} and \eqref{eq:identification}, we find a map
\begin{equation}\label{eq:crystallinemap1}
  d\log\colon\H^2(X,\mu_{p^n})\to\H^2(X/W_n).
\end{equation}
\begin{lemma}
  For each $n\geq 1$, the map~\eqref{eq:crystallinemap1} is injective.
\end{lemma}
\begin{proof}
    We induct on $n$. By \cite[Corollaire 0.2.1.18, pg. 517]{Illusie}, there is a short exact sequence
    \[
        1\to\mathbf{G}_m/\mathbf{G}_m^{\times p}\xrightarrow{d\log}Z\Omega^1_{X}\xrightarrow{W-C}\Omega^1_{X'}\to 0
    \]
    of \'{e}tale sheaves, where $X'$ denotes the Frobenius twist of $X$ over $k$. In particular, from the vanishing of $\H^0(X,\Omega^1_X)$ and the injectivity of $\H^1(X,Z\Omega^1_X)\to\H^2_{\dR}(X/k)=\H^2(X/W_1)$ (a consequence of the degeneration of the Hodge--de Rham spectral sequence) we obtain injectivity of~\eqref{eq:crystallinemap1} for $n=1$.
    
    We recall that the crystalline cohomology groups $\H^\ast(X/W)$ of a K3 surface are torsion free. This implies in particular that the maps
    \[
        \H^2(X/W)\otimes_{\mathbf{Z}}\mathbf{Z}/p^n\mathbf{Z}\to\H^2(X/W_n)
    \]
    are isomorphisms. Hence, multiplication by $p^n$ on $\H^2(X/W)$ induces a short exact sequence
    \[
        0\to\H^2(X/k)\xrightarrow{\cdot p^n}\H^2(X/W_{n+1})\to\H^2(X/W_n)\to 0
    \]
    We also have a short exact sequence
    \begin{equation}\label{eq:multiplication by p on mu}
        1\to\mu_p\to\mu_{p^{n+1}}\xrightarrow{\cdot p}\mu_{p^n}\to 1
    \end{equation}
    of fppf groups. We claim that the diagram
    \begin{equation}\label{eq:induction diagram}
        \begin{tikzcd}
            0\arrow{r}&\H^2(X,\mu_{p})\arrow{r}\arrow[hook]{d}&\H^2(X,\mu_{p^{n+1}})\arrow{r}{\cdot p}\arrow{d}&\H^2(X,\mu_{p^n})\arrow{d}&\\
            0\arrow{r}&\H^2(X/k)\arrow{r}{\cdot p^n}&\H^2(X/W_{n+1})\arrow{r}&\H^2(X/W_n)\arrow{r}&0
        \end{tikzcd}
    \end{equation}
    commutes and has exact rows, where the top horizontal row is given by the second cohomology of~\eqref{eq:multiplication by p on mu}, and the vertical arrows are~\eqref{eq:crystallinemap1}. The exactness of the top row follows from the vanishing of $\H^1(X,\mu_{p^n})$ (we remark that the top right horizontal arrow is surjective if and only if $X$ has finite height). To see the commutativity, note that applying $R^1\varepsilon_*$ to~\eqref{eq:multiplication by p on mu} results in the short exact sequence
    \[
        1\to\mathbf{G}_m/\mathbf{G}_m^{\times p}\xrightarrow{\cdot p^n} \mathbf{G}_m/\mathbf{G}_m^{\times p^{n+1}}\to\mathbf{G}_m/\mathbf{G}_m^{\times p^n}\to 1
    \]
    of \'{e}tale sheaves. Using diagram~\eqref{eq:induction diagram}, the result follows immediately by induction.
\end{proof}

We arrive at a diagram
\begin{equation}\label{eq:crystallinediagram}
  \begin{tikzcd}[row sep=small]
    \H^2(X/W)\arrow[two heads]{r}{\pi_n}&\H^2(X/W_n)&\\
    &\H^2(X,\mu_{p^n})\arrow[phantom]{u}{\rotatebox{90}{$\subset$}}\arrow[two heads]{r}&\Br(X)[p^n]
  \end{tikzcd}
\end{equation}
where $\pi_n$ denotes reduction modulo $p^n$. This is the crystalline analog of~\eqref{eq:etalediagram}.
\begin{definition}\label{def:crystallineBfield}
  Let $\alpha\in\Br(X)$ be a Brauer class which is killed by a power of $p$. A \textit{crystalline B-field lift of} $\alpha$ is an element 
    \[
      B\in\H^2(X/K)\defeq\H^2(X/W)\otimes_{W}K
    \]
  such that if we write $B=\frac{a}{p^n}$ for some $a\in\H^2(X/W)$, then $\pi_n(a)$ is equal to $d\log(\alpha')$ for some $\alpha'\in\H^2(X,\mu_{p^n})$ whose image in $\Br(X)$ is equal to $\alpha$.
\end{definition}

From the surjectivity of the horizontal maps in~\eqref{eq:crystallinediagram}, we see that any $p$-power torsion Brauer class admits a crystalline B-field lift. However, in contrast to the Hodge and $\ell$--adic cases, not every element of $\H^2(X/K)$ is a crystalline B-field lift of a Brauer class, because $\H^2(X,\mu_{p^n})$ is only a subgroup of $\H^2(X/W)\otimes\bZ/p^n\bZ$.

\begin{definition}
  A class $B\in\H^2(X/K)$ is a \textit{crystalline B-field} if it is a B-field lift of some Brauer class. Let $\ms B(X)\subset\H^2(X/K)$ denote the subgroup of crystalline B-fields. Let $\ms B_n(X)\subset\ms B(X)$ denote the subgroup of crystalline B-fields $B$ such that $p^nB\in\H^2(X/W)$.
\end{definition}

We take the direct limit of the maps $\ms B_n(X)\to\H^2(X,\mu_{p^n})$ to obtain a map
\begin{equation}\label{eq:B field map crystalline case}
    \begin{tikzcd}
        \ms B(X)\arrow[two heads]{r}&\H^2(X,\mu_{p^{\infty}})
    \end{tikzcd}
\end{equation}
which may be explicitly described exactly as in the \'{e}tale case~\eqref{eq:B field map etale case}: given a class $B\in\ms B(X)$, we choose $n\geq 0$ such that $p^nB\in\H^2(X/W)$, and then reduce modulo $p^n$. We compose~\eqref{eq:B field map crystalline case} with the map to the Brauer group to obtain a map
\begin{equation}\label{eq:B field map crystalline case part 2}
    \ms B(X)\to\Br(X)
\end{equation}
which we denote by $B\mapsto\alpha_B$. This is the crystalline analog of the exponential map~\eqref{eq:B field map Hodge case}. As in the $\ell$-adic case, the image of this map is $\Br(X)[p^{\infty}]\subset\Br(X)$, and a crystalline B-field lift of a class $\alpha\in\Br(X)[p^\infty]$ is exactly a preimage of $\alpha$ under~\eqref{eq:B field map crystalline case part 2}. We have a diagram
\begin{equation}\label{eq:big ol diagram crystalline}
    \begin{tikzcd}
        &&0\arrow{d}&0\arrow{d}&\\
        0\arrow{r}&\H^2(X/W)\arrow[equals]{d}\arrow{r}&\H^2(X/W)+\Pic(X)\otimes\mathbf{Q}_p\arrow{d}\arrow{r}&\Pic(X)\otimes(\mathbf{Q}_p/\mathbf{Z}_p)\arrow{d}\arrow{r}&0\\
        0\arrow{r}&\H^2(X/W)\arrow{r}&\ms B(X)\arrow{d}\arrow{r}&\H^2(X,\mu_{p^\infty})\arrow{r}\arrow{d}&0\\
        &&\dfrac{\ms B(X)}{\H^2(X/W)+\Pic(X)\otimes\mathbf{Q}_p}\arrow{r}{\sim}\arrow{d}&\Br(X)[p^{\infty}]\arrow{d}&\\
        &&0&0&
    \end{tikzcd}
\end{equation}
with exact rows and columns.
\subsection{Description of the group of crystalline B-fields} We will now give some results describing the subgroup $\ms B(X)\subset\H^2(X/K)$ more explicitly.

We recall that the \textit{Tate module} of a K3 crystal $H$ (in the sense of Ogus \cite[Def. 3.1]{Ogus}) is the $\mathbf{Z}_p$-module $H^{\phi=1}\subset H$ consisting of those elements $h\in H$ satisfying $\phi(h)=h$, where $\phi:=p^{-1}\Phi$ and $\Phi$ is the Frobenius endomorphism of $H$. By a result of Illusie \cite[Th\'{e}or\`{e}me 5.14, pg. 631]{Illusie}, if $X$ is a K3 surface then we have an exact sequence
  \[
        0\to\H^2(X,\mathbf{Z}_p(1))\to\H^2(X/W)\xrightarrow{p-\Phi}\H^2(X/W)
  \]
identifying $\H^2(X,\mathbf{Z}_p(1))$ with the Tate module $\H^2(X/W)^{\phi=1}$ of the K3 crystal $\H^2(X/W)$, where the left inclusion is given by the inverse limit of the inclusions~\eqref{eq:crystallinemap1}. We have inclusions
\[
    \Pic(X)\otimes\mathbf{Q}_p\subset\H^2(X,\mathbf{Q}_p(1))\subset\ms B(X)
\]
where as usual $\H^2(X,\mathbf{Q}_p(1))=\H^2(X,\mathbf{Z}_p(1))\otimes\mathbf{Q}_p$.
\begin{remark}
    By analogy with the Lefschetz (1,1) theorem, one might imagine that the inclusion $\Pic(X)\otimes\mathbf{Q}_p\subset\H^2(X,\mathbf{Q}_p(1))$ is an equality. However, this is frequently false, eg. for a very general ordinary K3 surface. It is true if $X$ is supersingular, being a consequence of the Tate conjecture for supersingular K3 surfaces (of course, the Tate conjecture is known for all K3 surfaces, but it is only in the supersingular case that there is such a consequence for K3 surfaces over general algebraically closed fields).
\end{remark}
\begin{proposition}\label{prop:description}
Let $X$ be a K3 surface.
\begin{enumerate}
  \item  If $X$ has finite height, then $\ms B(X)=\H^2(X/W)+\H^2(X,\mathbf{Q}_p(1))$.
  \item  If $X$ is supersingular, then $\ms B(X)=\ms B_1(X)+\H^2(X,\mathbf{Q}_p(1))$.
\end{enumerate}
\end{proposition}
\begin{proof}
    In either case, we have $\H^2(X/W)\subset\ms B_1(X)\subset\ms B(X)$ and $\H^2(X,\mathbf{Q}_p(1))\subset\ms B(X)$. It follows that in both cases the right hand side is contained in $\ms B(X)$. We prove the reverse containments. Consider the commutative diagram
  \begin{equation}\label{eq:surjective map in a diagram}
    \begin{tikzcd}
        \H^2(X,\mathbf{Z}_p(1))\arrow[hook]{r}\arrow{d}&\H^2(X/W)\arrow[two heads]{d}{\mod p^n}\\
        \H^2(X,\mu_{p^n})\arrow[hook]{r}&\H^2(X/W_n)
    \end{tikzcd}
  \end{equation}
  Suppose that $X$ has finite height. Flat duality implies that $\H^3(X,\mu_{p^n})=0$ for all $n\geq 1$. Hence, the maps $\H^2(X,\mathbf{Z}_p(1))\to\H^2(X,\mu_{p^n})$ are surjective. It follows that the restriction $\H^2(X,\mathbf{Q}_p(1))\to\Br(X)[p^{\infty}]$ of the exponential map~\eqref{eq:B field map crystalline case part 2} is surjective. This proves $(1)$. We next prove $(2)$. Suppose that $X$ is supersingular. For each $n$ and $i$ we consider the short exact sequence
  \[
    0\to \U^i(X,\mu_{p^n})\to\H^i(X,\mu_{p^n})\to \D^i(X,\mu_{p^n})\to 0
  \]
  As $\H^1(X,\mu_{p^n})=0$, flat duality shows that $\D^3(X,\mu_{p^n})=0$. Hence, the maps $\D^2(X,\mu_{p^{n+1}})\to\D^2(X,\mu_{p^n})$ induced by multiplication $p:\mu_{p^{n+1}}\to\mu_{p^n}$ are surjective. Furthermore, the formal group associated to $\U^2(X,\mu_{p^n})$ is isomorphic to $\widehat{\Br}(X)\cong\widehat{\mathbf{G}}_a$, so $\U^2(X,\mu_{p^n})\cong\mathbf{G}_a(k)$. In particular, the groups $\U^2(X,\mu_{p^n})$ are $p$-torsion, and the maps $\U^2(X,\mu_{p^{n}})\to\U^2(X,\mu_{p^{n+1}})$ induced by the inclusion $\mu_{p^n}\subset\mu_{p^{n+1}}$ are isomorphisms. Write $\U^2(X,\mu_{p^{\infty}})$ for the union of the $\U^2(X,\mu_{p^n})$ and $\D^2(X,\mu_{p^\infty})$ for the union of the $\D^2(X,\mu_{p^n})$. It follows that the composition
  \[
    \H^2(X,\mathbf{Q}_p(1))\to\H^2(X,\mu_{p^{\infty}})\to\D^2(X,\mu_{p^{\infty}})
  \]
  is surjective, and that $\U^2(X,\mu_p)=\U^2(X,\mu_{p^{\infty}})$. Hence, the exponential map~\eqref{eq:B field map crystalline case part 2} restricts to a surjection $\ms B_1(X)+\H^2(X,\mathbf{Q}_p(1))\to\Br(X)[p^{\infty}]$, which proves (2).
\end{proof}

The following describes the subgroup $\ms B(X)\subset \H^2(X/K)$ in terms of the $F$-crystal structure on $\H^2(X/W)$, without explicit mention of flat cohomology or the Brauer group. The special case of classes $B\in p^{-1}\H^2(X/W)$ is Lem. 3.4.11 of \cite{BL}.

\begin{proposition}\label{prop:description1}
  A class $B\in\H^2(X/K)$ is a crystalline B-field if and only if
  \begin{equation}\label{eq:B field equation}
    B-\phi(B)\in \H^2(X/W)+\phi(\H^2(X/W))
  \end{equation}
  where $\phi=p^{-1}\Phi$.
\end{proposition}
\begin{proof}
  Write $H=\H^2(X/W)$. Suppose that $X$ has finite height. It is immediate from Prop. \ref{prop:description} $(1)$ that any B-field satisfies the claimed relation. Conversely, suppose that $B=\frac{a}{p^n}$ is an element satisfying~\eqref{eq:B field equation}. Consider the Newton--Hodge decomposition
  \[
    \H^2(X/W)=H_{<1}\oplus H_1\oplus H_{>1}
  \]
  of $\H^2(X/W)$ into subcrystals with the indicated slopes (see \S\ref{ssec:Newton Hodge decomp} below). Write $a=(a_{<1},a_1,a_{>1})$. We have $pH_{<1}\subset\Phi(H_{<1})$, or equivalently $H_{<1}\subset\phi(H_{<1})$ (see for instance \cite[\S 1.2]{MR563463}). Consider the map
  \[
    1-\phi:H_{<1}\to \phi(H_{<1})
  \]
  All slopes of $H_{<1}$ are less than one, so this map is injective. By \cite[Lem. II.5.3]{Illusie}, it is surjective, and hence an isomorphism. We have $(1-\phi)(a_{<1})\in p^n\phi(H_{<1})$, so in fact $a_{<1}\in p^n H_{<1}$. We have $\phi(H_{>1})\subset H_{>1}$. Thus, we have a map
  \[
    1-\phi:H_{>1}\to H_{>1}
  \]
  which as before is both injective and surjective, and hence an isomorphism. We have $(1-\phi)(a_{>1})\in p^nH_{>1}$, so in fact $a_{>1}\in p^n H_{>1}$. Finally, note that $H_1$ is a unit root crystal. It follows quickly that $a_1=p^nh+t$ for some $h\in H_1$ and some $t$ which is fixed by $\phi$. We conclude that $B\in\ms B(X)$. This completes the proof of Prop. \ref{prop:description1} in the case when $X$ has finite height.
  
  Suppose that $X$ is supersingular. By Lem. 3.4.11 of \cite{BL}, we have that $\ms B_1(X)$ consists exactly of those classes $B=\frac{a}{p}$ with $a\in H$ that satisfy~\eqref{eq:B field equation}. By Prop. \ref{prop:description} (2), any B-field satisfies the claimed relation. We prove the converse. The inclusion of the Tate module is an isogeny, meaning that the map $T\otimes K\to H\otimes K$ is an isomorphism. Thus, the natural map $H\xrightarrow{\sim}H^{\vee}\to T^{\vee}\otimes W$ is injective, and we may regard $H$ as a subgroup of the dual lattice $T^{\vee}\otimes W$. Note that if $h\in H$ and $t\in T$, then $\phi(h).t=\phi(h).\phi(t)=\sigma(h.t)$. It follows that $H+\phi(H)\subset T^{\vee}\otimes W$. Now, if $B\in\H^2(X/K)$ satisfies the claimed relation, then $B$ is in the kernel of the map $1-\phi:T^{\vee}\otimes W\to T^{\vee}\otimes (K/W)$, which is equal to $T^{\vee}\otimes W+T\otimes\mathbf{Q}_p$. We may therefore write $B=B'+\frac{t}{p^n}$ for some $B'\in T^{\vee}\otimes W$ and some $t\in T$. As $t$ is killed by $1-\phi$, $B'$ also satisfies the relation~\eqref{eq:B field equation}. But by \cite[Lem. 3.10]{Ogus}, we have $T^{\vee}\otimes W\subset p^{-1}H$, so $B'\in p^{-1}H$. By Lem. 3.4.11 of \cite{BL} we have $B'\in\ms B(X)$. We also have $t/p^n\in\ms B(X)$, and we conclude that $B\in\ms B(X)$, as desired.
\end{proof}

\begin{remark}
    One can alternatively prove Prop. \ref{prop:description1} by generalizing the method of \cite[Lem. 3.4.11]{BL}, which we sketch. This proof has the advantage of avoiding flat duality and being uniform in the height of $X$. The first step is to understand the cokernel of the map~\eqref{eq:dlog2}. This is described by the short exact sequence \cite[Lem. 2, page 779]{MR714830}
  \begin{equation}\label{eq:dlog1}
    0\to\mathbf{G}_m/\mathbf{G}_m^{\times p^n}\to\W_n\Omega^1_X\xrightarrow{1-\F}\W_n\Omega^1_X/d(\W_n\cO_X)\to 0
  \end{equation}
  where $1$ denotes the projection and $\F$ is the map defined in \cite[Prop. II.3.3]{Illusie}. One then proceeds by analyzing the $p$-adic filtrations on crystalline and de Rham--Witt cohomology.
\end{remark}

\subsection{$p$-primary torsion in the Brauer group} 
We make some observations connecting the group $\ms B(X)$ of crystalline B-fields to the $p$-primary torsion in the Brauer group of $X$. Suppose that $X$ has finite height $h$. By Prop. \ref{prop:description}, we have
\[
    \ms B(X)=\H^2(X,\mathbf{Q}_p(1))+\H^2(X/W)
\]
In particular,~\eqref{eq:big ol diagram crystalline} induces an isomorphism
\begin{equation}\label{eq:B field map crystalline case finite height2}
    \dfrac{\H^2(X,\mathbf{Q}_p(1))}{\H^2(X,\mathbf{Z}_p(1))+\Pic(X)\otimes\mathbf{Q}_p}\iso\Br(X)[p^{\infty}]
\end{equation}
The slope 1 part of $\H^2(X/W)$ has rank $22-2h$, so we have $\H^2(X,\mathbf{Z}_p(1))\cong\mathbf{Z}_p^{\oplus 22-2h}$. Thus,~\eqref{eq:B field map crystalline case finite height2} gives an isomorphism
\begin{equation}\label{eq:Brauer group crystalline case finite height1}
    \Br(X)[p^\infty]\cong\left(\mathbf{Q}_p/\mathbf{Z}_p\right)^{\oplus 22-\rho-2h}
\end{equation}
where $\rho$ is the Picard rank of $X$. 
This could also be seen from the fact that, in the finite height case, the diagram~\eqref{eq:big ol diagram} with $\ell$ replaced by $p$ (and \'{e}tale cohomology with flat cohomology) still has exact rows and columns.
\begin{remark}\label{rem:missing Brauer classes}
    The exponent appearing in the formula~\eqref{eq:Brauer group crystalline case finite height1} for the $p$-primary torsion of the Brauer group is smaller than that for the $l$-primary torsion~\eqref{eq:Brauer group crystalline case finite height} by a factor of $2h$. These ``missing'' $p$-primary torsion Brauer classes are the cause of the restriction at $p$ in Thm \ref{thm: existence}.
\end{remark}


We now suppose $X$ is supersingular. By Prop. \ref{prop:description}, we have
\[
    \ms B(X)=\ms B_1(X)+\H^2(X,\mathbf{Q}_p(1))
\]
By the Tate conjecture for supersingular K3 surfaces, the first crystalline Chern character induces an isomorphism $\Pic(X)\otimes\mathbf{Z}_p\xrightarrow{\sim}T\otimes\mathbf{Z}_p=\H^2(X,\mathbf{Z}_p(1))$, and so $\H^2(X,\mathbf{Q}_p(1))$ is in the kernel of the crystalline exponential map~\eqref{eq:B field map crystalline case part 2}. Write $N=\Pic(X)$. We have $\rho=22$, so $\Br(X)$ has no prime to $p$ torsion (see~\eqref{eq:Brauer group crystalline case finite height}). We conclude that~\eqref{eq:B field map crystalline case part 2} restricts to a surjection $\ms B_1(X)\to \Br(X)$. We have a short exact sequence
\[
    0\to p^{-1}N/N\to \ms B_1(X)/H\to\Br(X)\to 0
\]
In particular, $\Br(X)$ is $p$-torsion. As shown in the proof of Prop. \ref{prop:description1}, we have that $\ms B_1(X)\subset N^{\vee}\otimes W+p^{-1}N\subset p^{-1}N\otimes W$, where the latter inclusion holds because discriminant group of $N$ is $p$-torsion. Let $\ms B_1(X)^\circ=\ms B_1(X)\cap (N^{\vee}\otimes W)$. We have a short exact sequence
\begin{equation}\label{eq:a SES for B}
    0\to N^{\vee}/N\to\ms B_1(X)^\circ/H\to\Br(X)\to 0
\end{equation}
The subgroup $\ms B_1(X)^{\circ}$ can be understood using Ogus's results on the classification of supersingular K3 crystals \cite{Ogus}. Write $K=H/N\otimes W$ and $V=N^{\vee}/N\cong\mathbf{F}_p^{2\sigma_0}$ (here, $\sigma_0$ is the Artin invariant of $X$). The subspace $K\subset V\otimes k$ is Ogus's \textit{characteristic subspace}, and has dimension $\sigma_0$. Let $\phi:V\otimes k\to V\otimes k$ be the map $\phi(v\otimes\lambda)=v\otimes\lambda^p$. Ogus showed that $K$ is totally isotropic and is in a special position with respect to $\phi$. Namely, $K+\phi(K)$ has dimension $\sigma_0+1$, and $V\otimes k=\sum_{i}\phi^i(K)$ has dimension $2\sigma_0$. This implies that there exists a \textit{characteristic vector} for $K$, which is an element $e_1\in V\otimes k$ such that writing $e_i=\phi^{i-1}(e_1)$, we have that $\left\{e_{0},\dots,e_{\sigma_0-1}\right\}$ is a basis for $K$ and $\left\{e_{0},\dots,e_{2\sigma_0-1}\right\}$ is a basis for $V\otimes k$. We let $f_i$ denote the functional given by pairing with $e_i$, so that $\left\{f_0,\dots,f_{\sigma_0-1}\right\}$ is a basis for $K^{\vee}=V\otimes k/K$. By Prop. \ref{prop:description1}, the subgroup $\ms B_1(X)^{\circ}/H\subset V\otimes k/K$ is the kernel of the map $1-\phi:V\otimes k/K\to V\otimes k/(K+\phi(K))$. It follows that we have
\[
    \ms B_1(X)^{\circ}/H=\left\{\lambda f_1+\lambda^pf_2+\dots+\lambda^{p^{\sigma_0-1}}f_{\sigma_0-1}|\lambda\in k\right\}
\]
We conclude that $\ms B_1(X)^{\circ}/H$ is isomorphic to the underlying additive group $\mathbf{G}_a(k)$ of the group field $k$. The left term of~\eqref{eq:a SES for B} is discrete, and hence there is an isomorphism
\[
    \Br(X)\cong\mathbf{G}_a(k)
\]
\begin{remark}
    Multiplying by $p$ and then reducing modulo $p$, the characteristic subspace $K$ is identified with the kernel of the $k$-linearized first de Rham Chern character $c_1^{\dR}\otimes k:\Pic(X)\otimes k\to\H^2_{\dR}(X/k)$, and the vector space $V\otimes k/K$ is identified with its image. Furthermore, $\ms B_1(X)/H$ is identified with $\H^2(X,\mu_p)$ (regarded as a subgroup of $\H^2_{\dR}(X/k)$ via the map $d\log$) and $\ms B_1(X)^{\circ}/H$ is identified with $\U^2(X,\mu_p)$.
\end{remark}

\subsection{The twisted Mukai crystal}

We recall the Mukai crystal introduced in \cite{LO1}. We set
\[
    \widetilde{\H}(X/W)=\H^0(X/W)(-1)\oplus\H^2(X/W)\oplus\H^4(X/W)(1)
\]
As a result of the Tate twists on the first and third factors on the right hand side, the Frobenius operator $\widetilde{\Phi}$ on $\widetilde{\H}(X/W)$ is given by the formula
\[
    \widetilde{\Phi}(a,b,c)=(p\sigma(a),\Phi(b),p\sigma(c))
\]
where we have identified $\H^0$ and $\H^4$ with $W$, and where $\Phi$ is the Frobenius operator on $\H^2(X/W)$. We equip $\widetilde{\H}(X/W)$ with the Mukai pairing. It is immediate from the definitions that $\tH(X/W)$ is a K3 crystal of rank 24. 

\begin{definition}\label{def:twistedK3crystal}
Let $B$ be a crystalline B-field. The \textit{twisted Mukai crystal} associated to $(X,B)$ is
\[
  \tH(X/W,B)=\exp(B)\tH(X/W)\subset \tH(X/K)
\]
Here, $\exp(B)$ is the isometry of $\tH(X/K)$ defined by the formula~\eqref{eq:exp B}.
\end{definition}

The twisted Mukai crystal has a natural structure of a K3 crystal by the following result.
\begin{theorem}
     Let $B\in\mathscr{B}(X)$ be a crystalline B-field. The endomorphism $\widetilde{\Phi}$ of $\tH(X/K)$ restricts to an endomorphism of $\tH(X/W,B)$. When equipped with the restriction of the Mukai pairing, the twisted Mukai crystal $\tH(X/W,B)$ is a K3 crystal of rank 24.
\end{theorem}
\begin{proof}
    When $B\in\ms B_1(X)$, this is Prop. 3.4.15 of \cite{BL}. Using Prop. \ref{prop:description1}, the proof of loc. cit. applies verbatim to give the result for general B-fields as well.
\end{proof}
 
Note that if $h\in \H^2(X/W)$ then $\exp(h)=(1,h,h^2/2)\in \H^*(X/W)$. Thus, as a submodule of $\tH(X/K)$, $\tH(X/W,B)$ depends only on the image of $B$ in $\H^2(X,\mu_{p^n})$. Furthermore, up to isomorphism (of K3 crystals), $\tH(X/W,B)$ only depends on the Brauer class $\alpha_B$ (see \cite[Lem. 3.2.4]{Bragg-Derived-Equiv}).

\begin{remark}
    For a K3 surface over the complex numbers, Huybrechts and Stellari \cite{HS} define the twisted Mukai lattice $\tH(X,B,\mathbf{Z})$ to be equal to the untwisted lattice $\tH(X,\mathbf{Z})$ with a modified Hodge structure. This differs from our definition of the twisted Mukai crystal (as well as the twisted $\ell$-adic Mukai lattice), as we have defined $\tH(X/W,B)$ by equipping the rational Mukai lattice $\tH(X/K)$ with a nonstandard integral structure, but the same crystal structure. The convention analogous to that of loc. cit. would be to define $\tH(X/W,B)$ to be equal to $\tH(X/W)$ as a $W$-module, but equipped with the twisted Frobenius operator $\tPhi_B=\exp(-B)\circ\tPhi\circ\exp(B)=\exp(\phi(B)-B)$.
\end{remark}
 We record the following observation.
\begin{proposition}
    Let $X$ be a K3 surface and $B$ be a crystalline B-field. If $X$ has finite height $h$, then $\tH(X/W,B)$ is a K3 crystal of height $h$, and in particular is abstractly isomorphic to $\tH(X/W)$. If $X$ is supersingular of Artin invariant $\sigma_0$, then $\tH(X/W,B)$ is a supersingular K3 crystal whose Artin invariant is equal to either $\sigma_0$ if $\alpha_B=0$ or $\sigma_0+1$ if $\alpha_B\neq 0$.
\end{proposition}
\begin{proof}
    Suppose that $X$ has finite height. The defining inclusion $\tH(X/W,B)\subset\tH(X/K)$ is compatible with the pairing and Frobenius. Thus, $\tH(X/W,B)$ and $\tH(X/K)$ are isogenous, and so $\tH(X/W,B)$ has height $h$. If $h$ is finite, this implies the crystals are isomorphic integrally. Alternatively, we may reason as follows. Because $X$ has finite height, by Prop. \ref{prop:description} we may assume $B$ satisfies $B=\varphi(B)$. The map $\exp(-B)$ then defines an isomorphism $\tH(X/W,B)\cong\tH(X/W)$ of K3 crystals. If $X$ is supersingular, then the Brauer group of $X$ is $p$-torsion. As the twisted Mukai crystal depends up to isomorphism only on the class $\alpha_B$, we may assume that $B\in\ms B(X)_1$. The result then follows from \cite[Cor. 3.4.23]{BL}.
\end{proof}

\subsection{The Newton--Hodge decomposition of the twisted Mukai crystal}\label{ssec:Newton Hodge decomp}

Let $H$ be a K3 crystal. The \textit{Newton--Hodge decomposition} of $H$ is a canonical direct sum decomposition
\[
    H=H_{<1}\oplus H_1\oplus H_{>1}
\]
with the following properties. If $H$ has finite height $h$ and rank $r$, then $H_{<1}$ has slope $1-1/h$ and rank $h$, $H_1$ has slope $1$ and rank $r-2h$, and $H_{>1}$ has slope $1+1/h$ and rank $h$. Furthermore, $H_1$ is orthogonal to $H_{<1}\oplus H_{>1}$, and under the pairing $H_{<1}$ and $H_{>1}$ are dual. If $H$ is supersingular, then $H_1=H$ and $H_{<1}=H_{>1}=0$.

Let $X$ be a K3 surface and let $B$ be a crystalline B-field. We will relate the Newton--Hodge decompositions of $\tH(X/W,B)$ and $\H^2(X/W)$.

\begin{proposition}\label{prop:Newton Hodge decomp of twisted Mukai crystal}
    As submodules of $\tH(X/K)$, we have equalities
        \begin{align*}
            \tH(X/W,B)_{<1}&=\H^2(X/W)_{<1}\\
            \tH(X/W,B)_{1}&=\exp(B)\left(\H^0(X/W)\oplus \H^2(X/W)_{1}\oplus\H^4(X/W)\right)\\
            \tH(X/W,B)_{>1}&=\H^2(X/W)_{>1}
        \end{align*}
\end{proposition}
\begin{proof}
    If $X$ is supersingular, then $\tH(X/W,B)$ is also supersingular, and the result is trivial. Suppose $X$ has finite height.  Write $\H^i=\H^i(X/W)$. By Prop. \ref{prop:description}, $B$ is congruent modulo $\H^2$ to a B-field $B'$ satisfying $\varphi(B')=B'$. As $\tH(X/W,B)=\tH(X/W,B')$, we may assume without loss of generality that $B$ is fixed by $\varphi$, and in particular $B\in(\H^2)_1\otimes K$. We then have that $\tH(X/W,B)_{<1}=\exp(B)\tH(X/W)_{<1}$, and similarly for the slope $1$ and $>1$ parts. It is immediate that the Newton--Hodge decomposition of $\tH(X/W)$ is given by  $\tH(X/W)_{<1}=(\H^2)_{<1}$, $\tH(X/W)_{1}=\H^0\oplus(\H^2)_1\oplus \H^4$, and $\tH(X/W)_{>1}=(\H^2)_{>1}$. The result follows upon noting that $\H_{1}\otimes K$ is orthogonal to $(\H^2)_{<1}$ and $(\H^2)_{>1}$, and so $\exp (B)(\H^2)_{<1}=(\H^2)_{<1}$ and $\exp(B)(\H^2)_{>1}=(\H^2)_{>1}$.
\end{proof}

Note that if $B$ is a general B-field, the direct sum decomposition of $\tH(X/W,B)_1$ described in the statement of Prop. \ref{prop:Newton Hodge decomp of twisted Mukai crystal} may \textit{not} be preserved by $\tPhi$.

\subsection{Mixed realization}\label{ssec:mixed}

We define B-fields for Brauer classes whose order is not necessarily a prime power. For simplicity we give the definitions only when $\Char k=p>0$.

\begin{definition}\label{def:mixed B field}
    Let $\alpha\in\Br(X)$ be a class of exact order $m$.
    Fix a prime $q$. Let $q^n$ be the largest power of $q$ dividing $m$, and set $t=m/q^n$. If $q=\ell\neq p$, then an $\ell$\textit{-adic B-field lift} of $\alpha$ is an $\ell$-adic B-field lift (in the sense of Def. \ref{def:ladicBfield}) of $t\alpha$. Similarly, if $q=p$, then a \textit{crystalline B-field lift} of $\alpha$ is a crystalline B-field lift (in the sense of Def. \ref{def:crystallineBfield}) of $t\alpha$.
\end{definition}

\begin{definition}\label{def:mixed B field 2}
  Let $\alpha\in\Br(X)$ be a Brauer class. A \textit{mixed B-field lift of} $\alpha$ is a set $\mathbf{B}=\left\{B_{\ell}\right\}_{\ell\neq p}\cup\left\{B_p\right\}$ consisting of a choice of an $\ell$-adic B-field lift $B_{\ell}$ of $\alpha$ for each prime $\ell\neq p$ and a crystalline B-field lift $B_{p}$ of $\alpha$ (in both cases in the sense of Def. \ref{def:mixed B field}).
  
  Given a mixed B-field $\mathbf{B}$, we write $\mathbf{B}^p$ for the component in $\H^2(X,\mathbf{A}_f^p)$, and $\mathbf{B}_p=B_p$ for the component in $\ms B(X)\subset\H^2(X/K)$.
\end{definition}

We say a few words to explain this definition. Let $\mu_{\ast}=\bigcup_{m}\mu_m$ be the subsheaf of torsion sections of $\mathbf{G}_m$. Let $p_0=p$ and let $p_1,p_2,\dots$ be an enumeration of the remaining primes. We have a canonical isomorphism
\[
    \mu_{p_0^{\infty}}\oplus\mu_{p_1^{\infty}}\oplus\mu_{p_2^{\infty}}\dots\cong\mu_{\ast}
\]
given by multiplication. As described in the introduction, we have
\[
    \H^2(X,\mathbf{A}_f^p)={\prod_{i\geq 1}}'\H^2(X,\mathbf{Q}_{p_i}(1))
\]
where the restricted product on the right hand side consists of tuples $\left\{B_i\right\}$ such that for all but finitely many $i$ we have $B_i\in\H^2(X,\mathbf{Z}_{p_i}(1))$. A mixed B-field lift of a class $\alpha$ is a preimage of $\alpha$ under the composition.
\begin{equation}\label{eq:b field diagram, mixed case}
    \begin{tikzcd}
        \ms B(X)\times \H^2(X,\mathbf{A}_f^p)\arrow[two heads]{r}&\bigoplus_i\H^2(X,\mu_{p_i^{\infty}})\arrow{r}{\sim}&\H^2(X,\mu_{\ast})\arrow[two heads]{r}&\Br(X)
    \end{tikzcd}
\end{equation}
which we denote by $\mathbf{B}\mapsto\alpha_{\mathbf{B}}$. Here, the right horizontal map is induced by the inclusion $\mu_{\ast}\subset\mathbf{G}_m$.

\section{Twisted Chern characters and action on cohomology}\label{sec:twisted Chern characters and action on cohomology}

Let $X$ be a smooth projective variety over a field $k$ and let $\alpha\in\Br(X)$ be a torsion Brauer class. In this section we will define a certain twisted Chern character for $\alpha$-twisted sheaves on $X$. This will be a map from the Grothendieck group of coherent $\alpha$-twisted sheaves on $X$ to the rational Chow group $A^*(X)_{\mathbf{Q}}$ of $X$. There are multiple inequivalent definitions of twisted Chern characters appearing in the literature, several of which are reviewed and compared in \cite[\S3]{HSApp}. These all seem to be essentially equivalent in practice. We will the notion appearing in \cite{LMS,Bragg-Derived-Equiv,BL}, which is also used in \cite[\S2]{Huy}. This formulation seems to us to be the most flexible, and has a uniform interaction with B-fields in each of the contexts we have considered. We remark that our definition below is described in terms of cocycles in \cite[App. A.1]{Bragg-Derived-Equiv} and is compared to the twisted Chern characters of Huybrechts and Stellari \cite{HS} in \cite[App. A.2]{Bragg-Derived-Equiv}.



Suppose that $n\alpha=0$ for some positive integer $n$. To define our twisted Chern character we will make an auxillary choice of a preimage $\alpha'\in\H^2(X,\mu_n)$ of $\alpha$ under the surjection
\[
    \H^2(X,\mu_n)\twoheadrightarrow\Br(X)[n]
\]
induced by the inclusion $\mu_n\subset\mathbf{G}_m$.

We choose a $\mathbf{G}_m$--gerbe $\pi:\mathscr{X}\to X$ with cohomology class $\alpha$, and identify the category of $\alpha$-twisted sheaves on $X$ with the category of coherent sheaves on $\mathscr{X}$ of weight 1. We also choose a $\mu_n$-gerbe $\mathscr{X}'\to X$ with cohomology class $\alpha'$, and an isomorphism $\mathscr{X}'\wedge_{\mu_n}\mathbf{G}_m\cong\mathscr{X}$ (see \cite[Chapter 12.3]{MR3495343}). There is then a canonical $n$-fold twisted invertible sheaf $\mathscr{L}$ on $\mathscr{X}$. Given a locally free $\alpha$-twisted sheaf $\ms E$ of finite rank, we note that $\ms E^{\otimes n}\otimes\mathscr{L}^{\vee}$ is a 0-twisted sheaf on $\mathscr{X}$. We define
\[
    \ch^{\alpha'}(\ms E)=\sqrt[n]{\ch(\pi_*(\ms E^{\otimes n}\otimes\mathscr{L}^{\vee}))}
\]
where the $n$-th root is chosen so that $\rk$ is positive. One can check that $\ch^{\alpha'}$ depends only on $\alpha'$, and not on the choice of gerbes or on $\ms L$. We note that $\ch_0$ and $\ch_1$ are given by
\begin{equation}\label{eq:formula for twisted Chern characters}
    \ch^{\alpha'}(\ms E)=(\rk(\ms E),\pi_*(\det(\ms E)\otimes\ms L^{\vee}),\dots)
\end{equation}
Assume that $\mathscr{X}$ has the resolution property, so that every $\alpha$-twisted sheaf admits a finite resolution by locally free $\alpha$-twisted sheaves. We then obtain by additivity a map
\[
    \ch^{\alpha'}:K(X,\alpha)\to A^*(X)_{\mathbf{Q}}
\]
where $K(X,\alpha)$ denotes the Grothendieck group of the category of $\alpha$-twisted sheaves. We note that this definition is purely algebraic, and hence makes sense in any characteristic. Furthermore, we did not need $\alpha$ to be topologically trivial, only torsion.

Suppose that $X$ is a K3 surface. We explain the relationship between the choice of $\alpha'$ and the choice of a B-field lift of $\alpha$. We first observe that, in any of the contexts we have considered, a choice of B-field lift for $\alpha$ determines in particular a choice of preimage of $\alpha$ in $\H^2(X,\mu_n)$. More precisely, a choice of singular B-field lift (if the ground field is the complex numbers) or of a mixed B-field lift determines a preimage in $\H^2(X,\mu_n)$. If $\alpha$ is killed by $\ell^n$, then a choice of $\ell$-adic B-field lift determines a preimage in $\H^2(X,\mu_{\ell^n})$, and if $\alpha$ is killed by $p^n$ a choice of crystalline B-field lift determines a preimage in $\H^2(X,\mu_{p^n})$. In any of these situations, we write
\[
    \ch^{B}(\ms E)=\ch^{\alpha'}(\ms E)
\]
where $\alpha'$ is the induced preimage. We also set
\[
    v^B(\ms E)=\ch^B(\ms E).\sqrt{\td(X)}
\]


\subsection{Twisted Chern characters on twisted K3 surfaces}
\label{sec: twisted Chern character}
\begin{definition}
  We assume now that $k$ is an algebraically closed field of characteristic $p>0$. If $X$ is a K3 surface over $k$, we define the \textit{extended N\'{e}ron-Severi group} of $X$ by
  \[
    \tN(X)=\langle (1,0,0)\rangle\oplus N(X)\oplus\langle (0,0,1)\rangle =A^*(X)\subset A^*(X)_{\mathbf{Q}}
  \]
\end{definition}
As the Chern characters of a coherent sheaf on a K3 surface are integral, the extended N\'{e}ron-Severi group is a natural recipient for the Chern class map, and in fact the Chern class map
\[
  \ch\colon K(X)\to\tN(X)
\]
is an isomorphism. Let $\alpha\in\Br(X)$ be a Brauer class.
For two $\alpha$-twisted sheaves $\mathscr{E},\mathscr{F}$, we have the Riemann--Roch formula
\[
    \chi(\mathscr{E},\mathscr{F})=-\langle v^B(\mathscr{E}),v^B(\mathscr{F})\rangle
\]
We will identify a subgroup of $\tN(X)\otimes\bQ$ which contains the image of the twisted Chern class map
\[
  \ch^B\colon K(X,\alpha)\to\tN(X)\otimes\bQ
\]
\begin{definition}\label{def:twisted NS lattice}
  If $B$ is an $\ell$-adic B-field, we define the $\ell$\textit{-adic twisted N\'{e}ron--Severi group} by
  \[
   \tN(X,B_\ell)=\left(\tN(X)\otimes\mathbf{Z}\left[\ell^{-1}\right]\right)\cap \tH(X,\bZ_{\ell},B_{\ell})
  \]
  If $\mathrm{char\,} k = p$ and $B$ is a crystalline B-field, we define the \textit{crystalline twisted N\'{e}ron--Severi group} by
  \[
    \tN(X,B_p)=\left(\tN(X)\otimes\mathbf{Z}\left[p^{-1}\right]\right)\cap \tH(X/W,B_p)
  \]
  and if $\mathbf{B}=\left\{B_\ell\right\}_{\ell\neq p}\cup\left\{B_p\right\}$ is a mixed B-field, we define the \textit{mixed twisted N\'{e}ron-Severi group} by
  \[
    \tN(X,\mathbf{B})=\left(\bigcap_{\ell \neq p}\tN(X,B_{\ell})\right)\cap \tN(X,B_p)
  \]
  where the intersection is taken inside of $\tN(X)\otimes\bQ$. Note that for all but finitely many primes $q$ the B-field $B_q$ is integral. Hence, the intersection defining $\tN(X,\mathbf{B})$ is finite.
\end{definition}

The restriction of the Mukai pairing on $\tN(X)\otimes\mathbf{Q}$ to the $\ell$-adic twisted N\'{e}ron-Severi group takes values in $\mathbf{Z}[\ell^{-1}]\cap\mathbf{Z}_\ell=\mathbf{Z}$. Similarly, the Mukai pairing restricts to an integral pairing on the crystalline twisted N\'{e}ron--Severi group.
The following is the crucial integrality result for twisted Chern characters, generalizing the fact that the Chern characters of usual sheaves on K3 surfaces are integral.

\begin{proposition}\label{prop:twistedcherncharacter}
  Let $X$ be a K3 surface and $\mathbf{B}$ a mixed (resp. $\ell$-adic, resp. crystalline) B-field lift of a Brauer class $\alpha\in\Br(X)$. For any twisted sheaf $\sE\in\Coh^{(1)}(X,\alpha)$, the twisted Chern character $\ch^{\mathbf{B}}(\sE)$ lies in the mixed (resp. $\ell$-adic, resp. crystalline) twisted N\'{e}ron-Severi group $\tN(X,\mathbf{B})$.
\end{proposition}
\begin{proof}
  This is proved in Appendix A of \cite{Bragg-Derived-Equiv} (the quoted statement is written for a crystalline B-field of the form $B=\frac{a}{p}$, but the proof applies essentially unchanged).
\end{proof}

\begin{remark}
    The analog of Prop. \ref{prop:twistedcherncharacter} for the Hodge realization follows immediately from the existence of an invertible twisted sheaf in the differentiable category (in fact, this existence is used to define twisted Chern characters in \cite{HS}). The $\ell$-adic case is proved in \cite[Lem. 3.3.7]{LMS} by lifting to characteristic 0.
\end{remark}

\begin{proposition}\label{prop:twisted Chern character is surjective}
    For any mixed B-field lift of $\alpha$, the twisted Chern character
    \[
        \ch^{\mathbf{B}}:K(X,\alpha)\to\tN(X,\mathbf{B})
    \]
    is surjective.
\end{proposition}
\begin{proof}
    The analogous result over the complex numbers is \cite[Prop. 1.4]{HS}. The proof in our case is identical, up to our differences in convention.
\end{proof}

\subsection{Action on cohomology}
\label{sec: action on cohomology}
Let $(X,\alpha)$ and $(Y,\beta)$ be twisted K3 surfaces over $k$. Choose mixed B-field lifts $\mathbf{B}$ of $\alpha$ and $\mathbf{B}'$ of $\beta$. As above, we define the twisted Chern character map
\[
    \ch^{-\mathbf{B}\boxplus\mathbf{B}'}\colon K(X\times Y,-\alpha\boxplus\beta)\to\tN(X\times Y)\otimes\bQ
\]
and set
\[
   v^{-\mathbf{B}\boxplus\mathbf{B}'}(\_)=\ch^{-\mathbf{B}\boxplus\mathbf{B}'}(\_).\sqrt{\td(X\times Y)}
\]
Let $\Phi_P:D^b(X,\alpha)\to D^b(Y,\beta)$ be a Fourier--Mukai equivalence. We consider the map
\begin{equation}\label{eq:map induced by FM transform}
    \Phi_{v^{-\mathbf{B}\boxplus\mathbf{B}'}(P)}:=\pi_{2*}(\pi_1^*(\_)\cup v^{-\mathbf{B}\boxplus\mathbf{B}'}(P)):\H^*(X)_{\mathbf{Q}}\to\H^*(Y)_{\mathbf{Q}}
\end{equation}
where $\pi_1:X\times Y\to X$ and $\pi_2:X\times Y\to Y$ are the respective projections. Using the same formula, we define maps $\Phi^{\ell}_{v^{-B_\ell\boxplus B'_\ell}(P)}$ on the rational $\ell$--adic cohomologies and $\Phi^{\cris}_{v^{-B_p\boxplus B'_p}(P)}$ on rational crystalline cohomology. By definition, these maps are equal to the maps given by restricting~\eqref{eq:map induced by FM transform} to the $\ell$--adic and crystalline components of $\H^*_{\mathbf{Q}}$.
\begin{theorem}
\label{thm:action on cohomology}
    Let $\Phi_P:D^b(X,\alpha)\to D^b(Y,\beta)$ be a Fourier--Mukai equivalence. The map~\eqref{eq:map induced by FM transform} restricts to an isomorphism
    \begin{equation}\label{eq:map induced by FM transform at ell}
            \Phi^{\ell}_{v^{-B_\ell\boxplus B'_\ell}(P)}:\tH(X,\mathbf{Z}_\ell,B_\ell)\to\tH(Y,\mathbf{Z}_\ell,B'_\ell)
    \end{equation}
    which is compatible with the Mukai pairings for each $\ell\neq p$, to an isomorphism
    \begin{equation}\label{eq:map induced by FM transform at p}
            \Phi^{\cris}_{v^{-B_p\boxplus B'_p}(P)}:\tH(X/W,B_p)\to\tH(Y/W,B'_p)
    \end{equation}
    of K3 crystals (that is, an isomorphism of $W$-modules which is compatible with the pairing and Frobenius operators), and to an isometry
    \begin{equation}\label{eq:map induced by FM transform on NS}
        \Phi_{v^{-\mathbf{B}\boxplus\mathbf{B}'}(P)}:\tN(X,\mathbf{B})\xrightarrow{\sim}\tN(Y,\mathbf{B}')
    \end{equation}
\end{theorem}
\begin{proof}
    By definition, the map~\eqref{eq:map induced by FM transform at ell} is equal to the correspondence induced by the cycle $v^{-B_\ell\boxplus B'_\ell}(P)$, and the map~\eqref{eq:map induced by FM transform at p} is equal to the correspondence induced by the cycle $v^{-B_p\boxplus B'_p}(P)$. The compatibility with the pairing, and in the crystalline case, with the Frobenius, is proved exactly as in \cite[\S 3.4]{Bragg-Derived-Equiv}. It remains to show that the correspondences preserve the integral structures. Under the assumption that $p\geq 5$, this is shown in \cite[App. A]{Bragg-Derived-Equiv}. The result in general can be shown by lifting to characteristic 0, using the techniques of the following section. We omit further details. This proves the claims regarding~\eqref{eq:map induced by FM transform at ell} and~\eqref{eq:map induced by FM transform at p}. To prove the claimed properties of~\eqref{eq:map induced by FM transform on NS}, note that the indicated correspondence preserves the subgroups of algebraic cycles, and so restricts to an isomorphism $\tN(X)\otimes\mathbf{Q}\xrightarrow{\sim}\tN(Y)\otimes\mathbf{Q}$. The result then follows from the previous claims.
\end{proof}

\section{Rational Chow motives and isogenies}\label{ssec:rational chow motives}

Given a smooth proper variety $X$ over and algebraically closed field $k$, we let $\fh(X)$ denote its rational Chow motive.

\begin{definition}\label{def:isogeny}\emph{\upshape{(cf. \cite[Def.~1.1]{Yang})}}
    Let $X, X'$ be K3 surfaces over $k$. An \textit{isogeny} from $X$ to $X'$ is an isomorphism of motives $f : \fh^2(X') \sto \fh^2(X)$ whose cohomological realization $\H^2(X')_\bQ \to \H^2(X)_\bQ$ preserves the Poincar\'e pairing. Two isogenies are said to be \textit{equivalent} if they induce the same map $\H^2(X')_\bQ \sto \H^2(X)_\bQ$ (see \ref{ssec:notation}).
\end{definition} 

Recall \cite[14.2.2]{MR2187153} that if $X$ is a K3 surface over an algebraically closed field $k$ then there is a canonical decomposition
\[
    \fh^2(X) = \fh^2_\algebr(X) \oplus \fh^2_{\tr}(X)
\]
of the Chow motive in degree two into an algebraic part and a transcendental part. The algebraic part $\fh^2_\algebr(X)$ is isomorphic to $\bL \tensor \NS(X)$, where $\bL$ stands for the Lefschetz motive. Similarly, $\fh(X)$ decomposes as $\fh_\algebr(X) \oplus \fh_{\tr}^2(X)$, where $\fh_\algebr = \bL^0 \oplus \fh^2_\algebr \oplus \bL^2$.

Now suppose $(X, \alpha)$ and $(Y, \beta)$ are twisted K3 surfaces with mixed B-field lifts $\bB$ of $\alpha$ and $\bB'$ of $\beta$. Let $\Phi_P : D^b(X, \alpha) \to D^b(Y, \beta)$ be a Fourier--Mukai equivalence. Following Huybrechts \cite[Theorem 2.1]{Huy}, we have that the correspondence $v^{-\bB \boxplus \bB'}(P)$ induces an isomorphism $\fh(X)\sto\fh(Y)$, which restricts to isomorphisms $\fh^2_{\tr}(X) \sto \fh^2_{\tr}(Y)$ and $\fh_\algebr(X) \sto \fh_\algebr(Y)$. By Witt's cancellation theorem, we can always find some isomorphism $\fh^2_\algebr(X) \sto \fh^2_\algebr(Y)$ which preserves the Poincar\'e pairing. Adding this and the isomorphism $\fh^2_{\tr}(X) \sto \fh^2_{\tr}(Y)$ induced by $v^{-\bB \boxplus \bB'}(P)$, we obtain an isogeny $\fh^2(X) \sto \fh^2(Y)$. We make the following definition: 
\begin{definition}
\label{def: primitive derived isogeny}
Let $X, Y$ be K3 surfaces. An isogeny $f : \fh^2(X) \sto \fh^2(Y)$ is a \textit{primitive derived isogeny} if its restriction $\fh^2_{\tr}(X) \sto \fh^2_{\tr}(Y)$ agrees with the one induced by $v^{-\bB \boxplus \bB'}(P)$ for some choices of $\alpha, \beta, \bB, \bB'$ and $\Phi_P$ as above. A \textit{derived isogeny} is a composition of finitely many primitive derived isogenies.
\end{definition}
In particular, note that if there exists a primitive derived isogeny between $X$ and $Y$, then $X$ and $Y$ are twisted derived equivalent. Twisted derived equivalent K3 surfaces clearly have the same rational Chow motive. In a recent paper, Fu and Vials proved that their motives are moreover isomorphic as Frobenius algebra objects, and over $\bC$ they also give a motivic characterization of twisted derived equivalent K3's (\cite[Thm~1, Cor.~2]{FV}).

\section{Lifting derived isogenies to characteristic 0}\label{sec:lifting isogenies}

The goal of this section is to give some lifting results for primitive derived isogenies. This requires understanding deformations of twisted K3 surfaces and of twisted Fourier--Mukai equivalences in mixed characteristic. Deformations of a twisted K3 surface $(X,\alpha)$ over the complex numbers can be profitably understood in terms of deformations of a pair $(X,B)$, where $B\in\H^2(X,\mathbf{Q})$ is a Hodge B-field lift of $\alpha$ (see for instance \cite{MR3946279}). Considering deformations of $(X,B)$ serves two purposes: first, the B-field $B$ allows one to algebraize formal deformations of the Brauer class, and second, $B$ gives a notion of twisted Chern characters in the deformation family.
Suppose now that $(X,\alpha)$ is a twisted K3 surface in positive characteristic. To similarly understand deformations of $(X,\alpha)$ over a base of mixed characteristic, we would need a notion of mixed characteristic B-field lift. The $\ell$-adic theory works essentially unchanged in this setting, but the analog of the crystalline theory seems more complicated. We will avoid this issue by using instead of a B-field a simpler object, namely a preimage $\alpha'\in\H^2(X,\mu_n)$ of $\alpha$ under the map $\H^2(X,\mu_n)\to\Br(X)$. The deformation theory of such pairs $(X,\alpha')$ has been considered by the first author \cite{Bragg}: the flat cohomology groups $\H^2(X,\mu_n)$ can be defined relatively in families, and their tangent spaces can be understood in terms of de Rham cohomology. Moreover, it turns out that formal projective deformations of such pairs $(X,\alpha')$ algebraize, and furthermore the class $\alpha'$ can be used to define twisted Chern characters in families. Our approach to the deformation theory of twisted Fourier--Mukai equivalences is based on the techniques of Lieblich and Olsson \cite{LO1}, which we in particular extend to the twisted setting. 


Let $(X,\alpha)$ and $(Y,\beta)$ be twisted K3 surfaces over an algebraically closed field $k$ of characteristic $p>0$. Let $\Phi_P:D^b(X,\alpha)\cong D^b(Y,\beta)$ be a Fourier--Mukai equivalence induced by a complex $P\in D^b(X\times Y,-\alpha\boxplus\beta)$.

\begin{definition}
    The equivalence $\Phi_P:D^b(X,\alpha)\xrightarrow{\sim}D^b(Y,\beta)$ is \textit{filtered} if there exist preimages $\alpha'\in \H^2(X,\mu_n)$ of $\alpha$ and $\beta'\in \H^2(Y,\mu_m)$ of $\beta$ such that the cohomological transform
    \[
        \Phi_{v^{-\alpha'\boxplus\beta'}(P)}:\tN(X)_{\mathbf{Q}}\xrightarrow{\sim}\tN(Y)_{\mathbf{Q}}
    \]
    sends $(0,0,1)$ to $(0,0,1)$.
    
\end{definition}
Note that the condition to be filtered does not depend on the choices of $\alpha'$ and $\beta'$, and thus is an intrinsic property of $\Phi_P$. We consider the deformation functor $\Def_{(X,\alpha')}$, whose objects over an Artinian local $W$-algebra $A$ are isomorphism classes of pairs $(X_A,\alpha'_A)$ where $X_A$ is a flat scheme over $\Spec A$ such that $X_A\otimes k\cong X$, and $\alpha'_A\in\H^2(X_A,\mu_n)$ is a cohomology class such that $\alpha'_A|_X=\alpha'$ (see \cite[Definition 1.1]{Bragg}).


\begin{proposition}\label{prop:definition of delta}
    Suppose that $\Phi_P$ is filtered. Given a preimage $\alpha'\in \H^2(X,\mu_n)$ of $\alpha$, there is a canonically induced preimage $\beta'\in \H^2(Y,\mu_n)$ of $\beta$ and a morphism
    \[
        \delta_P:\Def_{(Y,\beta')}\to\Def_{(X,\alpha')}
    \]
    of deformation functors over $W$ (depending on $P$ and $\alpha'$).
\end{proposition}
\begin{proof}
    Let $\mathscr{X}\to X$ and $\mathscr{Y}\to Y$ be $\mathbf{G}_m$-gerbes representing $\alpha$ and $\beta$. The chosen preimage $\alpha'$ corresponds to an $n$-twisted invertible sheaf $\mathscr{L}$ on $\mathscr{X}$. Using Prop. \ref{prop:twisted Chern character is surjective}, we find a complex of twisted sheaves $\ms E$ on $\mathscr{X}$ with rank $n$ and $\det\ms E\cong\ms L$.
    Using the assumption that $\Phi_P$ is filtered, we see that $\Phi_{P}(\ms E)$ is a complex of twisted sheaves on $\mathscr{Y}$ of rank $n$. Thus, its determinant $\mathscr{N}=\det(\Phi_{P}(\ms E))$ is an invertible $n$-twisted sheaf on $\mathscr{Y}$. Note that this implies $n\beta=0$. We let $\beta'\in \H^2(Y,\mu_n)$ be the preimage of $\beta$ corresponding to $\ms N$. Note that the class $\beta'$ does not depend on our choice of $\ms E$.
    
    Let $\mathscr{X}'\to X$ and $\mathscr{Y}'\to Y$ be $\mu_n$-gerbes corresponding to $\alpha'$ and $\beta'$. Suppose given an Artinian local $W$-algebra $A$ and a deformation of $(Y,\beta')$ over $A$. Up to isomorphism, this is the same as giving a pair $(\ms Y'_A,\varphi)$, where $\ms Y'_A$ is a $\mu_n$-gerbe equipped with a flat proper map to $\Spec A$ and $\varphi:\ms Y'_A\otimes k\cong\ms Y'$ is an isomorphism of gerbes. We let $\mathscr{D}_{\mathscr{Y}'_A/A}$ be the stack of relatively perfect universally glueable simple $\mathscr{Y}'_A$-twisted complexes over $\Spec A$ with twisted Mukai vector $(0,0,1)$ (see \cite[Section 5]{LO1}).
    We let $P'$ be the pullback of $P$ along the product of the maps $\mathscr{X}'\subset\mathscr{X}$ and $\mathscr{Y}'\subset\mathscr{Y}$. As $\Phi_P$ is filtered, the complex $P'$ induces a map $\mathscr{X}'\to \mathscr{D}_{\mathscr{Y}'_A/A}\otimes k$. By reasoning identical to \cite[Lem. 5.5]{LO1}, this map is an open immersion. The image of $\mathscr{X}'$ is contained in the smooth locus of the morphism $\pi\otimes k$, so there is a unique open substack $\ms X'_A\subset\mathscr{D}_{\mathscr{Y}'_A/A}$ which is flat and proper over $\Spec A$ whose restriction to the closed fiber is isomorphic to $\mathscr{X}'$. Via this isomorphism, the stack $\ms X'_A$ has a canonical structure of $\mu_n$-gerbe. Thus, given a deformation of $(Y,\beta')$ over $A$, we have produced (using the complex $P$) a deformation of $(X,\alpha')$ over $A$. This defines a morphism $\Def_{(Y,\beta')}\to\Def_{(X,\alpha')}$.
\end{proof}

We now assume that $\Phi_P$ is a filtered Fourier--Mukai equivalence. We fix a preimage $\alpha'\in\H^2(X,\mu_n)$ of $\alpha$. Let $\beta'\in\H^2(Y,\mu_n)$ and
\begin{equation}\label{eq:delta}
        \delta_P:\Def_{(Y,\beta')}\to\Def_{(X,\alpha')}
\end{equation}
be the preimage and morphism produced by Prop. \ref{prop:definition of delta}. We continue the notation introduced above, so that $\pi_{\ms X}:\ms X\to X$ and $\pi_{\ms Y}:\ms Y\to Y$ are $\mathbf{G}_m$-gerbes corresponding to $\alpha$ and $\beta$, $\ms X'$ and $\ms Y'$ are $\mu_n$-gerbes corresponding to $\alpha'$ and $\beta'$, $\ms L$ and $\ms N$ are the corresponding $n$-fold twisted invertible sheaves on $\ms X'$ and $\ms Y'$, and $P'$ is the restriction of $P$ to $\ms X'\times\ms Y'$. Let $\mathbf{B}_{\alpha}$ and $\mathbf{B}_{\beta}$ be mixed B-field lifts of $\alpha$ and $\beta$ such that $n\mathbf{B}_{\alpha}$ and $n\mathbf{B}_{\beta}$ are integral and such that $n\mathbf{B}_{\alpha}\pmod{n}$ equals $\alpha'$ and $n\mathbf{B}_{\beta}\pmod{n}$ equals $\beta'$.
Write $\Phi$ for the cohomological transform
\[
    \Phi=\Phi_{v^{-\alpha'\boxplus\beta'}(P)}:\tN(X)_{\mathbf{Q}}\to\tN(Y)_{\mathbf{Q}}
\]
\begin{lemma}\label{lem:strongly filtered}
    The transform $\Phi$ satisfies $\Phi(0,0,1)=(0,0,1)$ and $\Phi(1,0,0)=(1,0,0)$, and restricts to an isometry $N(X)\xrightarrow{\sim}N(Y)$ of integral N\'{e}ron--Severi lattices.
\end{lemma}
\begin{proof}
    We are assuming that $\Phi_P$ is filtered, so we have $\Phi(0,0,1)=(0,0,1)$. Consider a complex $\ms E$ of twisted sheaves on $\ms X$ with rank $n$ and $\det\ms E\cong\ms L$. It follows immediately from the definition of the twisted Chern character that $v^{\alpha'}(\ms E)=(n,0,s)$ for some integer $s$. Moreover, we see that the vector
    \[
        \Phi(n,0,s)=\Phi(v^{\alpha'}(\ms E))=v^{\beta'}(\Phi_P(\mathscr{E}))
    \]
    has trivial second component. As $\Phi$ is an isometry, we conclude that $\Phi(n,0,s)=(n,0,s)$. It follows that $\Phi(1,0,0)=(1,0,0)$, and that $\Phi$ restricts to an isometry on the rational N\'{e}ron--Severi lattices. 
    
    We now prove that $\Phi$ in fact restricts to an isometry between the integral N\'{e}ron--Severi lattices. Consider an invertible sheaf $L$ on $X$. The complex $\Phi_P(\ms E\otimes\pi^*L)$ of twisted sheaves on $\ms Y$ has rank $n$. Let $\ms M$ be its determinant. Using the formula~\eqref{eq:formula for twisted Chern characters}, we see that the pushforward of the (0-twisted) invertible sheaf $\ms M\otimes\ms N^{\vee}$ to $Y$ has cohomology class $\Phi(L)$. In particular, $\Phi(L)$ is in $N(Y)$.
\end{proof}

The following result is our twisted analog of \cite[Prop. 6.3]{LO1}.

\begin{proposition}\label{prop:delta is an iso}
The morphism $\delta_P$~\eqref{eq:delta} is an isomorphism, and furthermore has the following properties.
\begin{enumerate}
    \item For any class $L\in \Pic(X)$, the map $\delta$ restricts to an isomorphism 
    \[
        \Def_{(Y,\beta',\Phi(L))}\cong\Def_{(X,\alpha',L)}
    \]
    \item For any augmented Artinian $W$-algebra $A$ and any lift $(X_A,\alpha_A')$ of $(X,\alpha')$ over $A$, there exists a perfect complex $P_A\in D^b(X_A\times_A Y_A,-\alpha_A\boxplus\beta_A)$ lifting $P$, where $(Y_A,\beta_A')=\delta^{-1}(X_A,\alpha_A')$ and $\alpha_A$ and $\beta_A$ are the Brauer classes associated to $\alpha'_A$ and $\beta'_A$.
\end{enumerate}
\end{proposition}
\begin{proof}
    To see that $\mu_P$ is an isomorphism, consider the same construction applied to the kernel $Q=P^{\vee}$ of the inverse Fourier--Mukai transform and the preimage $\beta'$ of $\beta$, which yields a map
    \[
        \mu_Q:\Def_{(X,\alpha')}\xrightarrow{\sim}\Def_{(Y,\beta')}
    \]
    We claim that $\mu_P$ and $\mu_Q$ are inverses. This may be verified exactly as in \cite[Prop. 6.3]{LO1}. To see claim (2), note that the restriction along the open immersion
    \[
        \ms X'_A\times \ms Y'_A\subset\ms D_{\ms Y'_A/A}\times\ms Y'_A
    \]
    of the universal complex lifts $P'$. To see (1), suppose that the deformation $(X_A,\alpha'_A)$ is contained in the subfunctor $\Def_{(X,\alpha',L)}$. There is then an invertible sheaf $L_A$ on $X_A$ deforming $L$. Let $\ms E_A$ be a relatively perfect complex of $\alpha_A$-twisted sheaves on $X_A$ with rank $n$ and trivial determinant. Let $\pi_A:\ms X_A\to X_A$ be the coarse space map. The determinant of the complex $\Phi_{P_A}(\ms E_A\otimes \pi_A^{*}L_A)$ is a 0-fold twisted sheaf on $\ms Y_{A}$, hence its pushforward to $Y_A$ is an invertible sheaf. Moreover, this sheaf has class lifting $\Phi(L)$.
\end{proof}

\begin{definition}
    We say that a filtered Fourier--Mukai equivalence $\Phi_P$ is \textit{polarized} if there exists B-field lifts $\mathbf{B},\mathbf{B}'$ of $\alpha$ and $\beta$ such that the isometry $\Phi:\Pic(X)\to\Pic(Y)$ (see Lem. \ref{lem:strongly filtered}) sends the ample cone $C_X$ of $X$ to the ample cone $C_Y$ of $Y$.
    
\end{definition}

One checks that the condition to be polarized is independent of the choice of B-field lifts, in the sense that it is verified for one choice of lifts if and only if it is verified for all choices of lifts.

We now prove our main lifting results. By results of the first author \cite{Bragg}, the twisted K3 surface $(X,\alpha)$ can be lifted to characteristic 0. Moreover, we can also compatibly lift the preimage $\alpha'$ of $\alpha$. As a consequence, Prop. \ref{prop:delta is an iso} shows that given such a lift there is an induced formal lift of $(Y,\beta)$, together with a lift of $\beta'$ and of the complex $P$ inducing the equivalence. Under the assumption that $\Phi_P$ is polarized, we can even produce a (non-formal) lift. We make this precise in the following result.

\begin{theorem}\label{thm:lifting in filtered situation}
    Suppose that $\Phi_P$ is a filtered polarized Fourier--Mukai equivalence. Let $L$ be an ample line bundle on $X$. Suppose we are given a complete DVR $V$ with residue field $k$ and a lift $(X_V,\alpha'_V,L_V)$ of $(X,\alpha',L)$ over $V$. There exists an ample line bundle $M$ on $Y$, a lift $(Y_V,\beta'_V,M_V)$ of $(Y,\beta',M)$ over $V$, and a perfect complex $P_V\in D^b(X_V\times_VY_V,-\alpha_V\boxplus\beta_V)$ (where $\beta_V$ is the image of $\beta'_V$ in the Brauer group) which induces a Fourier--Mukai equivalence and whose restriction to $D^b(X\times Y,-\alpha\boxplus\beta)$ is quasiisomorphic to $P$.
\end{theorem}
\begin{proof}
    Let $M$ be the line bundle on $Y$ corresponding to $\Phi(L)$. By Prop. \ref{prop:delta is an iso}, we find compatible deformations $(Y_{V_n},\beta'_{V_n},M_{V_n})$ of $(Y,\beta',M)$ over $V_n=V/\mathfrak{m}^{n+1}$ for each $n\geq 0$, together with compatible perfect complexes $P_{V_n}\in D^b(X_{V_n}\times_{V_n} Y_{V_n},-\alpha_{V_n}\boxplus\beta_{V_n})$ deforming $P$, where $\beta_{V_n}$ is the image of $\beta'_{V_n}$ in the Brauer group. As $\Phi_P$ is polarized, $M$ is ample, so by the Grothendieck existence theorem, there exists a scheme $(Y_V,M_V)$ over $V$ restricting to the $(Y_{V_n},\beta'_{V_n})$. By \cite[Prop. 1.4]{Bragg} there exists a class $\beta'_V\in\H^2(Y_{V},\mu_n)$ restricting to the $\beta'_{V_n}$. Finally, by the Grothendieck existence theorem for perfect complexes \cite[Prop. 3.6.1]{MR2177199}, there is a perfect complex $P_V\in D^b(X_V\times_VY_V,-\alpha_V\boxplus\beta_V)$ whose restriction to $V_n$ is quasi-isomorphic to $P_{V_n}$ for each $n$. Moreover, arguing as in the proof of Theorem 6.1 of \cite{LO1}, we see that the complex $P_V$ induces a Fourier--Mukai equivalence.
\end{proof}

\begin{definition}
\label{def: perfect lifting}
Let $\cX$ be a K3 surface over a local ring and $X$ be its special fiber. We say that $\cX$ is a \textit{perfect lifting} of $X$ if the restriction map $\Pic(\cX) \to \Pic(X)$ is an isomorphism. 
\end{definition}
We remark that if $\cX$ as above is over a DVR, the ample and the big and nef cones of the generic fiber are canonically identified with those of the special fiber.

\begin{theorem}\label{thm:lifting isogenies}
    Let $(X,\alpha)$ and $(Y,\beta)$ be twisted K3 surfaces over $k$. Let $\Phi_P:D^b(X,\alpha)\xrightarrow{\sim}D^b(Y,\beta)$ be a Fourier--Mukai equivalence. There exists
    \begin{enumerate}[label=\upshape{(\alph*)}]
        \item an autoequivalence $\Phi'$ of $D^b(Y,\beta)$ which is a composition of spherical twists about $(-2)$-curves,
        \item a DVR $V$ whose fraction field has characteristic 0 and with residue field $k$,
        \item projective lifts $(X_V,\alpha_V)$ and $(Y_V,\beta_V)$ of $(X,\alpha)$ and $(Y,\beta)$ over $V$, and
        \item a perfect complex $R_V\in D^b(X_V \times_V Y_V,-\alpha_V\boxplus\beta_V)$ which induces a Fourier--Mukai equivalence and whose restriction to $X\times Y$ is quasiisomorphic to the kernel $R$ of the equivalence $\Phi'\circ\Phi_P$.
    \end{enumerate}
    Moreover, if $X$ and $Y$ have finite height, we may choose the above data so that $\Phi'$ is the identity and $X_V$ and $Y_V$ are perfect liftings.
\end{theorem}
\begin{proof}
    Choose a preimage $\alpha'\in\H^2(X,\mu_n)$ of $\alpha$. Given a choice of preimage of $\beta$, we obtain an isometry
    \[
        \Phi:\tN(X)_{\mathbf{Q}}\xrightarrow{\sim}\tN(Y)_{\mathbf{Q}}
    \]
    Consider the class $v=(\Phi)^{-1}(0,0,1)\in \widetilde{\NS}(X)_{\mathbf{Q}}$ 
    Note that this class does not depend on the choice of preimage of $\beta$. By \cite[Thm 7.3]{Bragg}, we may find a DVR $V$ of characteristic 0 and residue field $k$ and a polarized lift $(X_V,\alpha'_V)$ of $(X,\alpha')$ over $V$ over which the class $v$ extends. Let $\alpha_V$ be the image of $\alpha'_V$ in the Brauer group of $X_V$. Let $\mathscr{M}_V=\mathscr{M}_{(X_V,\alpha_V)}(v)$ be the relative moduli space of $H$-stable $\alpha_V$-twisted sheaves on $X_V\to\Spec V$ with twisted Mukai vector $v^{\alpha_V'}=v$, where $H$ is a $v$-generic polarization. Let $M_V$ be the coarse space of $\mathscr{M}_V$. The morphism $M_V\to\Spec V$ is a projective family of K3 surfaces, and there is a class $\gamma_V\in\Br(M_V)$ such that the universal complex $Q_V$ induces an equivalence
    \[
        \Phi_{Q_V}:D^b(M_V,\gamma_V)\xrightarrow{\sim} D^b(X_V,\alpha_V)
    \]
    Let $\gamma\in\Br(M)$ be the restriction of $\gamma_V$ to $M$, and let $Q$ be the restriction of $Q_V$. The Fourier--Mukai equivalence
    \[
        \Phi_P\circ\Phi_Q:D^b(M,\gamma)\xrightarrow{\sim} D^b(Y,\beta)
    \]
    is filtered. As in \cite[Lem. 6.2]{LO1}, we may find an autoequivalence $\Phi'$ as in the statement of the theorem so that $\Phi'\circ\Phi_P\circ\Phi_Q$ is both filtered and polarized. Let $R$ denote its kernel. Choose a preimage $\gamma'_V\in\H^2(M_V,\mu_m)$ of $\gamma_V$, and write $\gamma'$ for the restriction of $\gamma'_V$ to $M$. Let $\beta'$ be the corresponding lift of $\beta$ produced by Prop. \ref{prop:definition of delta}. By Theorem \ref{thm:lifting in filtered situation}, there is a lift $(Y_V,\beta'_V)$ of $(Y,\beta')$ and $R_V$ of $R$ over $V$, corresponding to the lift $(M_V,\gamma_V')$ of $(M,\gamma')$. Consider the Fourier--Mukai equivalence
    \[
        \Phi_{R_V}\circ\Phi_{Q_V}^{-1}:D^b(X_V,\alpha_V)\to D^b(Y_V,\beta_V)
    \]
    This equivalence restricts over $k$ to $\Phi'\circ\Phi_P$. By the uniqueness of the kernel, we conclude that $R_V$ restricts to the kernel of the equivalence $\Phi'\circ\Phi_P$, as claimed.
    
    Suppose that $X$ and $Y$ have finite height. We modify the above as follows. Choose $\alpha'$ so that $p$ does not divide $n/\ord(\alpha)$. By \cite[Thm 7.3]{Bragg}, we may choose the lift $X_V$ so that the restriction map $\Pic(X_V)\to\Pic(X)$ is an isomorphism. It follows that $\Pic(Y_V)\to\Pic(Y)$ is also an isomorphism. In particular, every $(-2)$-class in $\Pic(Y)$ extends to $Y_V$. We now compose $\Phi_{R_V}$ with an autoequivalence of $D^b(Y,\beta)$ lifting the inverse of $\Phi'$. The kernel of the resulting equivalence then restricts to $P$, as desired.
\end{proof}



\section{Existence Theorems}
\label{sec: Existence}

The goal of this section is to construct isogenies with prescribed action on cohomology. In particular, we will prove Thm~\ref{thm: existence} and Thm~\ref{thm: crys-isog}.

\subsection{Construction of Derived Isogenies}

We begin with Thm~\ref{thm: existence}.

Let $R$ be an integral domain whose field of fractions is of characteristic 0 (we have in mind $R=\mathbf{Z}_{\ell}$ or $R=W$). Set $R_\bQ := R \tensor_\bZ \bQ$. Let $M$ be a quadratic lattice such that $2^{-1} m^2 \in R$ for every $m \in M$.

Given an element $b\in M$ such that $\langle b,b\rangle\neq 0$, the reflection in $b$ is the isometry $s_b:M_{\mathbf{Q}}\to M_{\mathbf{Q}}$ defined by
\[
    s_b(x)=x-\frac{2\langle x,b\rangle}{\langle b,b\rangle}b
\]

Let $\wH$ be a lattice of the form $R \oplus M \oplus R$ equipped with the Mukai pairing, i.e., 
$$\< ( r, m, s ), (r', m', s')\> = \<m, m'\> - rs' - r's, $$
and a multiplicative structure given by 
$$ ( r, m, s ) \cdot (r', m', s') = (rr', rm' + r'm, rs' + r's + \< m, m'\>). $$

\begin{lemma}\label{lem:a lattice lemma}
Let $b \in M$ be a primitive element such that $\langle b,b\rangle\neq 0$. Set $n := b^2/2$ and $B: = b/n \in M_\bQ := M \tensor_\bZ \bQ$. Let $B' \in M_\bQ$ be another element. If $\Phi : \wH_\bQ \sto \wH_\bQ$ satisfies
\begin{enumerate}[label=\upshape{(\alph*)}]
    \item $\Phi(1, 0, 0) = (0, 0, 1/n)$ and $\Phi(0, 0, 1) = (n, 0, 0)$, and
    \item $e^{B} \Phi e^{-B'}$ is $R$-integral (i.e., restricts to an isometry $\wH \sto \wH$),
\end{enumerate}
then $\varphi(M) = s_b(M)$, where $s_b \in \Aut(M_\bQ)$ is the reflection in $b$ and $\varphi$ is the restriction of $\Phi$ to $M_\bQ$. 
\end{lemma}
\begin{proof}
We extend $r_b := -s_b$ to an isometry $\Psi : \wH_\bQ \sto \wH_\bQ$ by requiring that $\Psi$ satisfies (a). It is straightforward to verify that $e^B \Psi e^{-B}$ is $R$-integral: 
\begin{align*}
    e^B \Psi e^{-B} (0, 0, 1) &= e^B \Psi(0, 0, 1) = e^B(n, 0, 0) = (n, b, 1)\\
    e^B \Psi e^{-B} (1, 0, 0) &= e^B \Psi (1, -B, 1/n) = e^B(1, -B, 1/n) = (1, 0, 0) \\
    e^B \Psi e^{-B} (0, m, 0) &= e^B \Psi (0, m, - \< B, m \>) = e^B(n\< -B, m\>, r_b(m), 0) = ( \< -b, m \>, -m, 0 ) 
\end{align*}

We now consider the composition 
$ (e^B \Psi e^{-B})^{-1} \circ (e^B \Phi e^{-B'}) = e^B (\Psi^{-1} \circ \Phi) e^{-B'}$, which has to be $R$-integral. Direct computation shows
\begin{align*}
    e^B (\Psi^{-1} \circ \Phi) e^{-B'} (0, m, 0) &= e^B  (\Psi^{-1} \circ \Phi) (0, m, -\< B', m\>) \\
    &= e^B (0, r_b^{-1} (\varphi(m)), -\<B', m\>) \\
    &= (0, r_b^{-1} (\varphi(m)), \<B, \varphi(m)\> - \< B', m\>).
\end{align*}

As $e^B (\Psi^{-1} \circ \Phi) e^{-B'}$ is $R$-integral, we deduce that $r_b^{-1} \circ \varphi$ is $R$-integral, and so $r_b(M) = \varphi(M)$. 
\end{proof}


The following result is the key geometric input for the proof of Thm \ref{thm: existence}. Let $k$ be an algebraically closed field of characteristic $p$. Given a K3 surface $X$ over $k$ and a class $b\in\H^2(X)$, we let $s_b:\H^2(X)_{\mathbf{Q}}\to\H^2(X)_{\mathbf{Q}}$ denote the isometry $s_{b_{p_1}}\times s_{b_{p_2}}\times\dots$. We say that a class $b\in\H^2(X)$ is \textit{primitive} if $nb'=b$ for an integer $n$ and $b'\in\H^2(X)$ implies $n=\pm 1$.

\begin{proposition}\label{prop: existence of primitive derived isogenies for reflections}
Let $X$ be a K3 surface over $k$. Let $b\in \H^2(X)$ be a primitive class such that $n:=b^2/2$ is an integer\footnote{that is, $n$ is in the image of the diagonal embedding $\mathbf{Z}\hookrightarrow W\times\widehat{\mathbf{Z}}^p$} and such that $b/n$ is a mixed B-field. There exists a K3 surface $X'$ together with a primitive derived isogeny $f: \fh^2(X') \to \fh^2(X)$ such that $f_*(\H^2(X')) = s_b(\H^2(X))$ in $\H^2(X)_{\mathbf{Q}}$.
\end{proposition}
\begin{proof}
Set $\mathbf{B} := b/n$ and let $\alpha=\alpha_{\mathbf{B}}$ be the Brauer class defined by $\mathbf{B}$. Let $X'$ be the moduli space of stable $\alpha$-twisted sheaves with Mukai vector $v^{\mathbf{B}} = (n, 0, 0)$ (where stability is taken with respect to a sufficiently generic polarization). As $b$ is primitive, the class $(n,0,0)$ is primitive in $\widetilde{N}(X,\mathbf{B})$. Thus, $X'$ is a K3 surface, and there exists a Brauer class $\alpha' \in \Br(X')$ together with an equivalence $\Phi_\sE: D^b(X', \alpha') \sto D^b(X, \alpha)$. Choose a mixed B-field lift $\mathbf{B}'$ of $\alpha'$. Then the cohomological action $\Phi : \H(X')_{\mathbf{Q}} \sto \H(X)_{\mathbf{Q}}$ of the algebraic cycle $v^{-\mathbf{B}' \boxplus \mathbf{B}}(\sE)$ sends $(0, 0, 1)$ to $(n, 0, 0)$.

Since $\Phi$ is an isometry, the vector $u = (\Phi)^{-1}(0, 0, 1/n)$ satisfies $u^2 = 0$ and $\< u, (0, 0, 1) \> = -1$. Therefore, $u$ is necessarily of the form $e^\delta = (1, \delta, \delta^2 /2)$ for some $\delta \in \H^2_\et(X')_{\mathbf{Q}}$. 
As $(0, 0, 1)$ is an algebraic class, and $\Phi$ is induced by an algebraic cycle, we have $\delta \in \NS(X')_\bQ$. After replacing $\mathbf{B}'$ by $\mathbf{B}' + \delta$, we may assume that $\Phi$ sends $(1, 0, 0)$ to $(0, 0, 1/n)$.
Now we may apply Lem. \ref{lem:a lattice lemma} to the $\ell$-adic part for each $\ell\neq p$ and to the crystalline part. We conclude that the degree $0$ part of the correspondence $v^{-\mathbf{B}' \boxplus \mathbf{B}}(\sE)$ sends $\H^2(X')$ to $s_b(\H^2(X))$.
\end{proof}



\subsection{Cartan--Dieudonn\'{e} theorems and strong approximation}

To apply Prop. \ref{prop: existence of primitive derived isogenies for reflections} towards the proof of Thm \ref{thm: existence}, we need to show that the reflections $s_b$ about classes $b\in\H^2(X)$ satisfying the conditions of Prop. \ref{prop: existence of primitive derived isogenies for reflections} generate a sufficiently large subgroup of isometries of $\H^2(X)_{\mathbf{Q}}$. We need two lattice theoretic inputs. The first is the following generalized Cartan--Dieudonn\'{e} theorem (\cite[Thm~2]{Kl}).

\begin{theorem}\label{thm:generalized Cartan--Dieudonne}
    Let $R$ be a local ring with residue characteristic $\neq 2$ and let $L$ be a unimodular quadratic lattice over $R$. The group $\O(L)$ is generated by the set of reflections $s_b$, where $b$ ranges over the elements of $L$ such that $b^2\in R^{\times}$.
\end{theorem}
We also will use the following consequence of the strong approximation theorem. Recall that $\mathbf{U}$ denotes the hyperbolic plane, which is a $\mathbf{Z}$--lattice of rank 2.

\begin{lemma}\label{lem:strong approx}
    Let $L$ be a non-degenerate indefinite quadratic lattice over $\mathbf{Z}$ of rank $\geq 3$. If $q$ is a prime such that $L\otimes\mathbf{Z}_q$ contains a copy of $\mathbf{U}\otimes\mathbf{Z}_q$ as an orthogonal direct summand, then the double quotient 
    \[
            \O(L\otimes\mathbf{Q})\backslash \O(L\otimes\mathbf{Q}_q)/\O(L\otimes\mathbf{Z}_q)
    \]
    is a singleton.
\end{lemma}
\begin{proof}
This is a slight variant of \cite[Lem.~2.1.11]{Yang2}, whose proof follows from that of \cite[Lem.~7.7]{Ogus}. We briefly summarize the argument: Let $\mathsf{K} \subseteq \Spin(L \tensor \bQ_q)$ be the preimage of the $\SO(L \tensor \bZ_q)$ under the natural map $\mathrm{ad} : \Spin \to \SO$. Using the fact that $L \tensor \bZ_q$ contains $\bU \tensor \bZ_q$ as an orthogonal direct summand, we show that the following two maps are both surjections
$$ \Spin(L \tensor \bQ) \backslash \Spin(L \tensor \bQ_q) / \mathsf{K} \to \SO(L \tensor \bQ) \backslash \SO(L \tensor \bQ_q) / \SO(L \tensor \bZ_q) \to \O(L \tensor \bQ) \backslash \O(L \tensor \bQ_q) / \O(L \tensor \bZ_q). $$
Now we conclude using the fact that the first double quotient is a singleton by the strong approximation theorem. 
\end{proof}

We now return to the setting of a K3 surface $X$ over an algebraically closed field $k$ of characteristic $p$.

\begin{lemma}\label{lem:lattices in the Tate module}
    Let $X$ be a K3 surface over $k$, and assume that $p\geq 5$. There exists a $\mathbf{Z}$--lattice $L$ of rank 22 and a primitive indefinite sublattice $L'\subset L$ such that
    \begin{enumerate}[label=\upshape{(\alph*)}]
        \item for each $\ell\neq p$, there exists an isometry $L\otimes\mathbf{Z}_{\ell}\cong\H^2(X,\mathbf{Z}_{\ell})$,
        \item there exists an isometry $L'\otimes\mathbf{Z}_p\cong T(X):=\H^2(X/W)^{\varphi=1}$, and
        \item The double quotients
        \[
            \O(L\otimes\mathbf{Z}_{(p)})\backslash \O(L\otimes\mathbf{A}_f^p)/\O(L\otimes\widehat{\mathbf{Z}}^p)\hspace{.5cm}\mbox{and}\hspace{.5cm} \O(L'\otimes\mathbf{Q})\backslash \O(L'\otimes\mathbf{Q}_p)/\O(L'\otimes\mathbf{Z}_p)
        \]
        are both singletons.
    \end{enumerate}
\end{lemma}
\begin{proof}
    Suppose that $X$ has finite height $h$. We take $L=\Lambda$ to be the K3 lattice. As $L$ contains a copy of $\mathbf{U}$ as an orthogonal direct summand, we may apply \cite[Lem.~2.1.11]{Yang2} to conclude that the indicated double quotient is a singleton. We will now produce $L'$. Suppose that $h\leq 9$. By \cite[Thm 6.4]{Ito-LFunction} (which requires $p\geq 5$), there exists a K3 surface $Y$ over $\overline{\mathbf{F}}_p$ such that $h(Y)=h$ and $\rho(Y)=22-2h$. Set $L'=\Pic(Y)$. The existence of a perfect lifting of $Y$ to characteristic zero shows that $L'$ admits a primitive embedding into $L=\Lambda$. Condition (a) is immediate. The embedding $L'\to\H^2(Y/W)$ induces an isomorphism $L'\otimes\mathbf{Z}_p\cong T(Y)=T(X)$, giving (b). It remains to check that the double quotient involving $L'$ is a singleton. The pairing on $\H_1$ is perfect, and the inclusion $T(X)\subset\H_1$ induces an isomorphism $T(X)\otimes_{\mathbf{Z}_p}W\cong\H_1$, so the discriminant of the pairing on $T(X)\cong L'\otimes\mathbf{Z}_p$ is a $p$-adic unit. As $L'$ has rank $\geq 4$, the classification of $p$--adic lattices \cite[Lem. 7.5]{Ogus} implies that $L'\otimes\mathbf{Z}_p$ contains a copy of $\mathbf{U}\otimes\mathbf{Z}_p$ as an orthogonal direct summand. By the Hodge index theorem, $L'$ is indefinite. We conclude using Lem. \ref{lem:strong approx}. Suppose $h=10$. We take $L'=\mathbf{U}$. This is certainly a primitive sublattice of $L=\Lambda$, and the double quotient involving $L'$ is a singleton. It remains to check that $\mathbf{U}\otimes\mathbf{Z}_p\cong T(X)$. As explained by Ogus \cite[Rem. 1.5]{Ogus2}, the discriminant of the pairing on $\H_1$ is $-1$. The same is then true for $T(X)$, because $T(X)\otimes_{\mathbf{Z}_p}W\cong\H_1$. By the classification of quadratic lattices over $\bZ_p$, we conclude that $\mathbf{U}\otimes\mathbf{Z}_p\cong T(X)$.
    
    Suppose that $X$ is supersingular. Let $L'=L=\Lambda_{\sigma_0}$ be the supersingular K3 lattice of Artin invariant $\sigma_0=\sigma_0(X)$. The discriminant of the pairing on $\Lambda_{\sigma_0}$ is equal to $-p^{2\sigma_0}$, which is an $\ell$--adic unit for all $\ell\neq p$, and so (a) holds. Condition (b) is immediate. Finally, by \cite[Lem. 7.7]{Ogus}, condition (c) holds.
\end{proof}

The following results could be phrased purely in terms of (semi)linear algebra, but for clarity we will maintain the geometric notation.

We recall that $\O(\H^2(X,\mathbf{A}^p_f))$ is the subgroup of $\prod_{\ell\neq p}\O(\H^2(X,\mathbf{Z}_{\ell}))$ consisting of those tuples $\Theta$ such that $\Theta_{\ell}$ is $\ell$--integral for all but finitely many $\ell$ (here, we say that $\Theta_{\ell}$ is $\ell$-integral if $\Theta_{\ell}(\H^2(X,\mathbf{Z}_{\ell}))=\H^2(X,\mathbf{Z}_{\ell})$). We let $\O_{\Phi}(\H^2(X/K))$ be the group of automorphisms of $\H^2(X/K)$ which are isometries with respect to the pairing and which commute with $\Phi$. We set $\O_{\Phi}(\H^2(X))=\O_{\Phi}(\H^2(X/K))\times\O(\H^2(X,\mathbf{A}^p_f))$.

\begin{remark}\label{rem:isometries are...}
    Giving an isometric embedding $\iota$ as in the statement of Thm \ref{thm: existence} is equivalent to giving an isometry $\iota_p:\Lambda\otimes W\hookrightarrow\H^2(X/K)$ of $W$-modules and for each prime $\ell\neq p$ an isometry $\iota_{\ell}:\Lambda\otimes\mathbf{Z}_{\ell}\hookrightarrow\H^2(X,\mathbf{Q}_{\ell})$ of $\mathbf{Q}_{\ell}$-modules such that for all but finitely many $\ell$ we have $\mathrm{im}(\iota_{\ell})=\H^2(X,\mathbf{Z}_{\ell})$. A similar description holds for the isometric embedding in the statement of Thm \ref{thm: existence, char 0}.
\end{remark}

\begin{lemma}\label{lem:Cartan--Dieudonne, prime to p part}
    Suppose that $p\geq 5$. If $\Theta^p\in \O(\H^2(X,\mathbf{A}_f^p))$ is an isometry, then there exists a sequence $b_1,\dots,b_r$ of primitive elements of $\H^2(X)$ such that
    \begin{enumerate}
        \item\label{item:CD prime to p, 1} for each $i$, $n_i:=b_i^2/2$ is an integer which is not divisible by $p$, and
        \item\label{item:CD prime to p, 2} the isometry $s:=s_{b_1}\circ\dots\circ s_{b_r}$ satisfies $s(\H^2(X,\widehat{\mathbf{Z}}^p))=\Theta^p(\H^2(X,\widehat{\mathbf{Z}}^p))$.
    \end{enumerate}
\end{lemma}
\begin{proof}
    Let $L$ be a lattice as in Lem. \ref{lem:lattices in the Tate module}, and choose an identification $L\otimes\widehat{\mathbf{Z}}^p= \H^2(X,\widehat{\mathbf{Z}}^p)$. By Lemma \ref{lem:strong approx},
    \[
            \O(L\otimes\mathbf{Z}_{(p)})\backslash \O(L\otimes\mathbf{A}^p_f)/\O(L\otimes\widehat{\mathbf{Z}}^p)
    \]
    is a singleton. Hence, there exists an isometry $\Psi\in \O(L\otimes\mathbf{Z}_{(p)})$ such that $\Psi(L)\otimes \widehat{\mathbf{Z}}^p=\Theta^p(L\otimes\widehat{\mathbf{Z}}^p)$. We apply Thm \ref{thm:generalized Cartan--Dieudonne} with $R=\mathbf{Z}_{(p)}$ to produce a sequence $b_1,\dots,b_r$ of elements of $L\otimes\mathbf{Z}_{(p)}$ such that $b_i^2\in\mathbf{Z}_{(p)}^{\times}$ for each $i$ and $\Psi=s_{b_1}\circ\dots\circ s_{b_r}$. For each $i$, we may write $b_i=v/m$ for some primitive $v\in L$ and an integer $m$ which is coprime to $p$. Note that the integer $v^2/2=m^2b_i^2/2$ is in $\mathbf{Z}_{(p)}^{\times}$, and hence is not divisible by $p$. Moreover, we have $s_{b_i}=s_{v}$. So, by replacing each $b_i$ with the corresponding $v$, we may arrange so that the $b_i$ satisfy~\eqref{item:CD prime to p, 1}. Condition~\eqref{item:CD prime to p, 2} holds by construction.
\end{proof}

\begin{lemma}\label{lem:Cartan--Dieudonne, p part}
    Suppose that $p\geq 5$. Let $\Theta_p\in\O_{\Phi}(\H^2(X/K))$ be an isometry which restricts to the identity on $\H^2(X/W)_{<1}$. There exists a sequence $b_1,\dots,b_r$ of primitive elements of $\H^2(X)$ such that
    \begin{enumerate}
        \item for each $i$, $n_i:=b_i^2/2$ is an integer and $\varphi(b_i)=b_i$, and
        \item the isometry $s:=s_{b_1}\circ\dots\circ s_{b_r}$ satisfies $s(\H^2(X/W))=\Theta_p(\H^2(X/W))$.
    \end{enumerate}
\end{lemma}
\begin{proof}
    Write $\H=\H^2(X/W)$, and consider the Newton--Hodge decomposition $\H=\H_{<1}\oplus \H_1\oplus \H_{>1}$ of $\H$. The first and third factors are dual, and orthogonal to $\H_1$. As $\Theta_p$ restricts to the identity on $\H_{<1}$, it must also restrict to the identity on $\H_{>1}$, and hence $\Theta_p$ restricts to an element of $\O_{\Phi}(\H_1)=\O(T(X))$. We fix lattices $L,L'$ as in Lem. \ref{lem:lattices in the Tate module} and an identification $L'\otimes\mathbf{Z}_p=T(X)$. By Lemma \ref{lem:strong approx}, we may find $\Psi\in\O(L'\otimes\mathbf{Q})$ such that $\Psi(L')\otimes\mathbf{Z}_p=\Theta_p|_{T(X)}(L'\otimes\mathbf{Z}_p)$. 
    By the classical Cartan--Dieudonn\'{e} theorem, we may find a sequence $b_1,\dots,b_r$ of elements of $L'\otimes\mathbf{Q}$ such that $\Psi=s_p=s_{b_1}\circ\dots\circ s_{b_r}$. By scaling, we may assume that each $b_i$ is in $L'$ and is primitive. Note that, as $\H_{<1}$ and $\H_{>1}$ are orthogonal to $\H_1$, the reflections $s_{b_i}$ are the identity on $\H_{<1}\oplus\H_{>1}$. If follows that $s$ satisfies condition (2).
\end{proof}

\begin{lemma}\label{lem:Cartan--Dieudonne}
    Suppose that $p\geq 5$. Let $\Theta\in \O_{\Phi}(\H^2(X)_{\mathbf{Q}})$ be an isometry such that $\Theta_p$ restricts to the identity on $\H^2(X/W)_{<1}$. There exists a sequence $b_1,\dots,b_m$ of primitive elements of $\H^2(X)$ such that
    \begin{enumerate}
        \item for each $i$, $n_i:=b_i^2/2$ is an integer and $b_i/n_i$ is a mixed B--field, and
        \item the isometry $s:=s_{b_1}\circ\dots\circ s_{b_m}$ satisfies $s(\H^2(X))=\Theta(\H^2(X))$.
    \end{enumerate}
\end{lemma}
\begin{proof}
We first choose elements $b_1,\dots,b_r\in\H^2(X)$ by applying Lem. \ref{lem:Cartan--Dieudonne, p part} to $\Theta_p$. We set $s=s_{b_1}\circ\dots\circ s_{b_r}$. We apply Lem. \ref{lem:Cartan--Dieudonne, prime to p part} to $(s^{-1}\circ\Theta)^p$ to obtain elements $b_1',\dots,b_t'\in\H^2(X)$. Set $s'=s_{b_1'}\circ\dots\circ s_{b_{t}'}$. We claim that the sequence $b_1,\dots,b_r,b'_1,\dots,b'_t\in\H^2(X)$ satisfies the desired conditions. We check (1). We have that $n_i:=b_i^2/2$ and $n'_i:=(b'_i)^{2}/2$ are integers. We have that $\varphi((b_i)_p)=(b_i)_p$, so by Prop. \ref{prop:description1} each $(b_i)_p/n_i$ is a crystalline B-field. It follows that $b_i/n_i$ is a mixed B-field. As $n_i'$ is not divisible by $p$, $(b'_i)_p/n'_i$ is in $\H^2(X/W)$, so $(b'_i)_p/n'_i$ is a crystalline B-field, and $b'_i/n'_i$ is a mixed B-field. We have shown that (1) holds. To check (2), note that by construction, we have
\[
    (s\circ s')^p(\H^2(X,\widehat{\mathbf{Z}}^p))=\Theta^p(\H^2(X,\widehat{\mathbf{Z}}^p))
\]
Furthermore, as $p$ does not divide $(b_i')^2/2$, we have $s'_p(\H^2(X/W))=\H^2(X/W)$, and so
\[
    (s\circ s')_p(\H^2(X/W))=s_p(\H^2(X/W))=\Theta_p(\H^2(X/W))
\]
\end{proof}

\noindent \textit{Proof of Theorem~\ref{thm: existence}.}
We prove the ``only if'' direction first. Suppose that $f:\fh^2(X')\xrightarrow{\sim}\fh^2(X)$ is a primitive derived isogeny. We may choose Brauer classes $\alpha\in\Br(X)$ and $\alpha'\in\Br(X')$, a Fourier--Mukai equivalence $\Phi_P:D^b(X',\alpha')\xrightarrow{\sim}D^b(X,\alpha)$, and crystalline B-field lifts $B,B'$ of $\alpha$ and $\alpha'$ such that the cohomological transform $\Phi_{v^{-B'\boxplus B}}(P):\tH(X'/K)\to\tH(X/K)$ and the cohomological realization $\H^2(X'/K)\xrightarrow{\sim}\H^2(X/K)$ of $f$ restrict to the same map $T^2(X'/K)\xrightarrow{\sim}T^2(X/K)$, where $T^2(X/K)$ denotes the orthogonal complement to $\NS(X)\otimes K$ in $\H^2(X/K)$ (not to be confused with the Tate module of $\H^2(X/W)$). By Thm \ref{thm:action on cohomology}, $\Phi_{v^{-B'\boxplus B}}(P)$ restricts to an isomorphism $\tH(X'/W,B')\xrightarrow{\sim}\tH(X/W,B)$ of crystals. Thus, by Prop. \ref{prop:Newton Hodge decomp of twisted Mukai crystal}, it induces an isomorphism
\[
    \H^2(X'/W)_{<1}=\tH(X'/W,B')_{<1}\xrightarrow{\sim}\tH(X/W,B)_{<1}=\H^2(X/W)_{<1}
\]
The transcendental part $T^2(X/K)$ contains $\H^2(X/W)_{<1}$, so the cohomological realization of $f$ also maps the slope $<1$ part to the slope $<1$ part. This gives the result. 

We now prove the ``if'' direction. For each $\ell\neq p$ fix an isometry $\H^2(X,\mathbf{Z}_\ell)\cong \Lambda\otimes\mathbf{Z}_\ell$. Assume first that the K3 crystals $H_p$ and $\H^2(X/W)$ are abstractly isomorphic. This is the case, for instance, if $X$ has finite height. We fix an isomorphism $\H^2(X/W)\cong H_p$ of K3 crystals. Composing with the given embedding $\iota$ and tensoring with $\mathbf{Q}$, we find an isometry $\Theta \in \O_{\Phi}(\H^2(X/K))\times\O(\H^2(X,\widehat{\mathbf{Z}}^p))$ which maps $\H^2_\et(X, \bZ_\ell)$ to $\iota_l(\Lambda \tensor \bZ_\ell)$ and $\H^2(X/W)$ to $\iota_p(\Lambda\otimes W)$. By Lem. \ref{lem:Cartan--Dieudonne}, we may find a sequence $b_1,\dots,b_m\in\H^2(X)$ of primitive elements such that for every $i$, $n_i:=b_i^2/2$ is an integer and $b_i/n_i$ is a mixed B-field, and furthermore the isometry $s:=s_{b_1}\circ\dots\circ s_{b_m}$ satisfies $s(\H^2(X))=\Theta(\H^2(X))$. The result follows by repeatedly applying Prop. \ref{prop: existence of primitive derived isogenies for reflections}.

We now consider the case when $X$ is supersingular and $H_p$ and $\H^2(X/W)$ are not isomorphic. This can certainly occur: any two supersingular K3 crystals over $k$ of the same rank and discriminant are isogenous, but by results of Ogus \cite{Ogus} supersingular K3 crystals themselves have nontrivial moduli. We argue as follows. By the global crystalline Torelli theorem \cite{Ogus2}, there exists a supersingular K3 surface $X'$ such that $\H^2(X'/W)$ is isomorphic as a K3 crystal to $H_p$. By Thm \ref{thm:supersingular K3 surfaces are derived isogenous} below, there exists a derived isogeny $\fh^2(X')\sto\fh^2(X)$, which induces an isometry $\H^2(X'/K)\cong\H^2(X/K)$. We are now reduced to the previous case, and we conclude the result. \qed

\begin{remark}
\label{rmk: relax assumption}
The only place where the assumption $p \ge 5$ is used in the above proof is the usage of Ito's result \cite[Thm~6.4]{Ito-LFunction}. If in Theorem~\ref{thm: existence}, $H_p = \H^2_\cris(X/W)$, i.e., $\Theta_p$ as above can be taken to be the identity, then the assumption $p > 2$ suffices. In this case, in producing $X'$ we only need to iteratively take moduli of sheaves twisted by Brauer classes of prime-to-$p$ order.
\end{remark}

\subsection{Existence in the supersingular case}
We make a few remarks specific to the supersingular case. Here, very strong cohomological results are available: there is a global Torelli theorem \cite{Ogus,Ogus2,BL}, as well as a derived Torelli theorem \cite{Bragg-Derived-Equiv}. Together, these give a picture which closely parallels the case of complex K3 surfaces. We will show that any two supersingular K3 surfaces are derived isogenous. More refined results (along the lines of \cite[Thm 0.1]{Huy}) are possible, but we will omit this discussion here.

\begin{theorem}\label{thm:supersingular K3 surfaces are derived isogenous} 
    Suppose that $p\geq 3$. Let $X$ and $Y$ be two supersingular K3 surfaces over $k$. There exists a derived isogeny $\fh^2(X) \sto \fh^2(Y)$.
\end{theorem}
\begin{proof}
    We use a result of Lieblich and the first author \cite[Prop. 5.2.5]{BL}: if $X$ is a supersingular K3 surface, then there exists a sequence $X_0,X_1,\dots, X_n$ of supersingular K3 surfaces together with Brauer classes $\alpha_i\in\Br(X_i)$ such that $X_0=X$, $D^b(X_i,\alpha_i)\cong D^b(X_{i+1},\alpha_{i+1})$ for each $0\leq i\leq n-1$, and $X_n=Z$ is the unique supersingular K3 surface with Artin invariant 1. Applying this to both $X$ and $Y$, we find derived isogenies
    \[
        \fh^2(X)\iso\fh^2(Z)\xleftarrow{\sim}\fh^2(Y).
    \]
\end{proof}

\begin{remark}
    Shioda \cite[Theorem 1.1]{MR572983} showed that supersingular Kummer surfaces are unirational. By a result of Ogus \cite{Ogus} and the crystalline Torelli theorem these are exactly the supersingular K3 surfaces with Artin invariant $\sigma_0\leq 2$. The Chow motive of a unirational surface is of Tate type. Combining this with Theorem \ref{thm:supersingular K3 surfaces are derived isogenous} we deduce that for any supersingular K3 surface $X$ we have $\fh(X)=\fh_{\algebr}(X)=\bL^0\oplus\bL^{\oplus 22}\oplus\bL^2$ and $\fh^2_{\tr}(X)=0$. In particular, we have $\CH^2(X)=\mathbf{Z}$. This result was first proved by Fakhruddin \cite{MR1904084}, using a related method.
\end{remark}


\subsection{Existence in characteristic 0}\label{ssec: Huybrecht's theorem algebraic version in char 0}

It is possible to formulate a purely algebraic analog of Huybrecht's theorem~\ref{thm:Huybrechts theorem} along the lines of Thm \ref{thm: existence}, valid over any algebraically closed field of characteristic $0$.

\begin{theorem}
    \label{thm: existence, char 0}
        Let $X$ be a K3 surface over an algebraically closed field of characteristic $0$. Let $H=\Lambda \tensor \widehat{\mathbf{Z}}$. Let $\iota : H\into \H^2(X)_\bQ$ be an isometric embedding. There exists a K3 surface $X'$ and a derived isogeny $f : \fh^2(X') \sto \fh^2(X)$ such that $f_*(\H^2(X')) = \mathrm{im}(\iota)$. 
\end{theorem}
\begin{proof}
    This can be proved purely algebraically along the same lines as our proof of Thm \ref{thm: existence} (but avoiding the extra complications at $p$). Alternatively, it can be deduced directly from Thm \ref{thm:Huybrechts theorem}. We omit further details.
\end{proof}

\subsection{Nygaard-Ogus Theory Revisited}
\label{sec: NO revisited}
We briefly recap Nygaard-Ogus' deformation theory of K3 crystals and K3 surfaces established in \cite[\S 5]{NO}, in preparation for the proof of Thm~\ref{thm: crys-isog}. For the rest of \S\ref{sec: Existence}, assume that $k$ is a perfect field with $\mathrm{char\,}k = p \ge 5$. We refer the reader to the paragraph below of the proof of Lem.~4.6 in \textit{loc. cit.} for this restriction on $p$. Let $R: =k[\varepsilon]/(\varepsilon^e)$ for some $e$. Recall that a K3 crystal over $R$ is an F-crystal $\bH$ on $\Cris(R/W)$ equipped with a pairing $\bH \times \bH \to \cO_{R/W}$ and an isotropic line $\Fil \subset \bH_R$ which satisfy some properties (see Def.~5.1 in \textit{loc. cit.} for details\footnote{In fact, \cite[Def.~5.1]{NO} defined K3 crystals over a more general base which satisfies a technical assumption \cite[(4.4.1)]{NO}. For our purposes it suffices to consider bases of the form $k[\varepsilon]/(\varepsilon^e)$.}).

\begin{definition}
\label{def: defK3crystal to V}
Suppose $V$ is a finite flat extension of $W$ such that $V/(p) = R$. A deformation of $\bH$ to $V$ is a pair $(\bH, \wt{\Fil})$ where $\wt{\Fil} \subset \bH_V$ is an isotropic direct summand which lifts $\Fil \subset \bH_R$. 
\end{definition}

\begin{theorem}
\label{thm: NO}
\emph{(Nygaard-Ogus)} Let $X$ be a K3 surface over $k$ and $R$ be as above. 
\begin{enumerate}[label=\upshape{(\alph*)}]
    \item The natural map $X_R \mapsto \H^2_\cris(X_R)$ defines a bijection between deformations $X_R$ of $X$ to $R$ to deformations of the K3 crystal $\H^2_\cris(X/W)$ to $R$, i.e., K3 crystals $\bH$ over $R$ with $\bH |_k = \H^2_\cris(X/W)$. 
    \item If $X_R$ is a deformation of $X$ to $R$, then the map $X_V \mapsto (\H^2_\cris(X_R), \Fil^2 \H^2_\dR(X_V/V))$ defines a bijection between deformations $X_V$ of $X_R$ to $V$ and deformations of the K3 crystal $\H^2_\cris(X_R)$ to $V$, in the sense of Def.~\ref{def: defK3crystal to V}. 
\end{enumerate}
\end{theorem}
\begin{proof}
This follows from \cite[Thm~5.3]{NO} and its proof.
\end{proof}

For the rest of \S\ref{sec: NO revisited}, $X$ denote a K3 surface of finite height over $k$. Recall that there is a canonical slope decomposition (cf. \cite[Prop.~5.4]{NO})
\begin{align}
\label{eqn: candecomp}
    \delta_\can : \H^2_\cris(X/W) = \ID(\BR_X^*) \oplus \ID(D^*) \oplus \ID(\BR_X)(-1).
\end{align}

We define a map $\IK$ which sends a deformation of $\BR_X$ to $R$ to a deformation of the K3 crystal $\H^2_\cris(X/W)$ to $R$ by setting $\IK(G_R) := \ID(G_R^*) \oplus \ID(D_R^*) \oplus \ID(G_R)(-1)$, where $D_R$ denote the canonical lift of $D$ to $R$. The K3 crystal structure on $\IK(G_R)$ is given as follows: Let $\sP_{G_R} : \ID(G_R^*) \times \ID(G_R) \to \cO_{R/W}(-1)$ be the canonical pairing and let $\sP_{D_R} : \ID(D_R^*) \times \ID(D^*_R) \to \cO_{R/W}(-2)$ be the pairing inherited from that on $\ID(D^*)$. The pairing on $\IK(G_R)$ is $\sP_{G_R}(-1) \oplus \sP_{D_R}$. Finally, the isotropic direct summand $\Fil$ in $\IK(G_R)_R$ is given by $[\Fil^1 \ID(G_R)_R](-1)$. We define a decreasing filtration on $\IK(G_R)_R$ by setting 
\begin{align}
\label{def: fil on IK}
    0 = \Fil^3 \subset \Fil^2 := \Fil \subset \Fil^1 := (\Fil^2)^\perp \subset \Fil^0 = \IK(G_R)_R.
\end{align}

If we further lift $G_R$ to a $p$-divisible group $G_V$ for a finite flat extension $V$ of $W$ with $V/(p) = R$, the we can attach a deformation of $\IK(G_R)$ to $V$ by setting $\wt{\Fil} = [\Fil^1 \ID(G_V)_V](-1)$, which we denote by $\IK(G_V)$. We define a filtration on $\IK(G_V)_V$ using (\ref{def: fil on IK}) with $\IK(G_R)_R$ replaced by $\IK(G_V)_V$.   

\begin{definition}
If $X_V$ be a formal scheme over $\mathrm{Spf\,} V$ which deforms $X$, we say $X_V$ is a Nygaard-Ogus lifting if it comes from $\IK(G_V)$ for some $p$-divisible group $G_V$ lifting $\BR_X$ to $V$ via Theorem~\ref{thm: NO}. That is, setting $R := V/(p)$, $G_R := (G_V) \tensor R$ and $X_R := (X_V) \tensor R$, we have an isomorphism
$$  (\H^2_\cris(X_R), \Fil^2 \H^2_\dR(X_V/V)) \iso \IK(G_V)$$
lifting $\delta_\can$ in the obvious sense. If $X_V$ is an algebraic space over $\mathrm{Spec\,} V$ which deforms $X$, then we say $X_V$ is a Nygaard-Ogus lifting if its formal completion at the special fiber is a Nygaard-Ogus lifting. 
\end{definition}

\begin{proposition}
If a formal scheme $X_V$ is a Nygaard-Ogus lifting of $X$, then the natural map $\Pic(X_V) \to \Pic(X)$ is an isomorphism. In particular, $X_V$ is algebraizable. 
\end{proposition}
\begin{proof}
See Prop.~4.5 and Rmk~4.6 of \cite{Yang}. 
\end{proof}

Using integral $p$-adic Hodge theory, we can characterize Nygaard-Ogus liftings: 

\begin{theorem}
\label{thm: char NO lifting}
Let $F$ be a finite extension of $K$ with $V := \cO_F$. Let $X_V$ be a formal scheme over $\mathrm{\Spf\,} V$ which lifts $X$ and let $X_F$ denote its rigid-analytic generic fiber. Then $X_V$ is a Nygaard-Ogus lifting if and only if there are ${\Gal_F}$-stable $\bZ_p$-sublattices $T^0, T^1, T^2$ in $\H^2_\et(X_{\bar{F}}, \bZ_p)$ of ranks $h, 22-2h, h$ respectively, such that as crystalline $\Gal_F$-representations 
\begin{enumerate}[label=\upshape{(\alph*)}]
    \item $T^1(1)$ is unramified,
    \item $T^0$ has Hodge-Tate weight $1$ with multiplicity $h - 1$ and $0$ with multiplicity $1$,
    \item $T^2(1)$ has Hodge-Tate weight $1$ with multiplicity $1$ and $0$ with multiplicity $h - 1$. 
\end{enumerate}
\end{theorem}
\begin{proof}
We recap in the appendix the results from the integral $p$-adic Hodge theory to be used in this proof.

Let $S$ denote Breuil's $S$-ring. Using the data $\Fil^\bullet \H^2_\dR(X_F/F)$, we equip $\H^2_\cris(X/W) \tensor_W S_K$ with the structure of an object in $\mathcal{MF}^{\varphi, N}_{S_K}$. For any $G_V$ which lifts $\BR_X$ to $V$, we set $$\IT_p (G_V) := T_p G_V \oplus T_p D_V \oplus T_p G^*_V(-1). $$

Suppose first that $X_V$ is Nygaard-Ogus, so that it comes from some $G_V$ lifting $G := \BR_X$. Combining (\ref{eq: p-adic comp}) and (\ref{eq: p-div p-comp}), we obtain isomorphisms
$$ \H^2_\et(X_{\bar{F}}, \bZ_p) \tensor_{\bZ_p} B_\cris \cong \H^2(X/W) \tensor_W B_\cris = \IK(G) \tensor_W B_\cris \cong \IT_p(G_V)(-1) \tensor_{\bZ_p} B_\cris  $$
which gives rise to a rational isomorphism $\H^2_\et(X_{\bar{F}}, \bZ_p) \tensor_{\bZ_p} \bQ_p \iso \IT_p(G_V)(-1) \tensor_{\bZ_p} \bQ_p$. 
We now show that the this restricts to an integral isomorphism 
\begin{align}
    \label{eq: integral iso}
    \H^2_\et(X_{\bar{F}}, \bZ_p) \cong \IT_p(G_V)(-1)
\end{align}
It is easy to check that the object $(\IK(G)_K, \Fil^\bullet \IK(G_V)_F)$ in $\MF^{\varphi}_{F}$ admits a decomposition into 
$$ (\ID(\BR_X^*)_K, \Fil^\bullet \ID(G_V^*)_F) \oplus (\ID(D^*)_K, \Fil^\bullet \ID(D_V^*)_F) \oplus (\ID(\BR_X)_K, \Fil^\bullet \ID(G_V)_F)(-1). $$
By the construction of Nygaard-Ogus liftings, there is an isomorphism $
    \ID(G^*_R)_S \oplus \ID(D_R^*)_S \oplus \ID(G_R)(-1)_S \cong \H^2_\cris(X_R)_S$
of strongly divisible $S$-modules which is compatible with the isomorphism $\H^2(X/W) \tensor_W S_K \cong \IK(G) \tensor_W S_K$ induced by $\delta_\can$. By applying the functor $T_\cris$, we obtain (\ref{eq: integral iso}), which readily implies the ``only if'' part of the theorem. 

Now we show the ``if'' part. The proof is essentially a reincarnation of the proof of \cite[Prop.~5.5]{NO}. The hypothesis implies that there exists an isomorphism $\H^2_\et(X_{\bar{F}}, \bZ_p) \cong \IT_p(G_V)(-1)$ for some $G_V$ which lifts $\BR_X$ to $V$. By Thm~\ref{thm: Liu}, there exists a unique isomorphism \begin{align}
\label{eq: eqn S-modules}
    \ID(G^*_R)_S \oplus \ID(D_R^*)_S \oplus \ID(G_R)(-1)_S \cong \H^2_\cris(X_R)_S.
\end{align}
which gives this isomorphism $\H^2_\et(X_{\bar{F}}, \bZ_p) \cong \IT_p(G_V)(-1)$ under $T_\cris$. 

The only thing we need to check is that this isomorphism of $S$-modules comes from an isomorphism of F-crystals on $\Cris(R/W)$
\begin{align}
\label{eq: eqn F-crystals}
    \IK(G_R) = \ID(G^*_R) \oplus \ID(D_R^*) \oplus \ID(G_R)(-1) \cong \H^2_\cris(X_R)
\end{align}
which restricts to $\delta_\can$. 

Let $e$ be the ramification degree of $V$ over $W$ and $j$ be any positive number $\le e$. Set $R_j := R/(\varepsilon^j)$. We claim that there exists a sequence of isomorphisms $\delta_j : \IK(G_{R_j}) \cong \H^2_\cris(X_R)$ of F-crystals on $\Cris(R_j/W)$ such that $\delta_j$ is the restriction of $\delta_{j + 1}$ for each $j < e$, such that $\delta_1 = \delta_\can$, and $\delta_e$ gives the desired isomorphism (\ref{eq: eqn F-crystals}). Suppose we have constructed $\delta_j$ for some $j < e$. Note that $(\varepsilon^{j})$ is a square-zero ideal in $R_{j + 1}$ and we can view $R_{j + 1}$ as an object of $\Cris(R_j/W)$ by equipping $(\varepsilon^{j})$ with the trivial PD structure. By \cite[Thm~5.2]{NO}, to construct $\delta_{j + 1}$ it suffices to show that $[\Fil^1 \ID(G_{R_{j + 1}})_{R_{j + 1}}](-1)$ is sent to $\Fil^2 \H^2_\dR(X_{R_{j + 1}}/R_{j + 1})$ via the composition 
$$ \IK(G_{R_{j + 1}})_{R_{j + 1}} \cong \IK(G_{R_j})_{R_{j + 1}} \xrightarrow[(\delta_j)_{R_{j + 1}}]{\sim} \H^2_\cris(X_{R_j})_{R_{j + 1}} \cong  \H^2_\dR(X_{R_{j + 1}}/R_{j + 1}). $$
However, this follows directly from the fact that (\ref{eq: eqn S-modules}) respects the filtrations. Indeed, viewing $R_{j + 1}$ as an $S$-algebra via $S \to \cO_F \to R \to R_{j + 1}$, we get the above isomorphism by tensoring (\ref{eq: eqn S-modules}) with $R_{j + 1}$. 
\end{proof}

\begin{remark}
When $X$ is ordinary, $X_V$ is Nygaard-Ogus if and only if it is obtained via base change from the canonical lifting, as in this case deformations of $\BR_X$ are completely rigid. Therefore, the above theorem is a generalization \cite[Thm C]{TaelmanOrd} when $p \ge 5$. It also follows from (\ref{eq: integral iso}) in the above proof that when $X_V$ is a Nygaard-Ogus lifting, for the enlarged formal Brauer group $\Psi_{X_V}$ of $X_V$, there is a natural injective map of ${\Gal_F}$-modules 
$$ T_p \Psi_{X_V} \to \H^2_\et(X_{\bar{F}}, \bZ_p(1)), $$
which generalizes \cite[Thm~2.1]{TaelmanOrd}. Indeed, we have $\Psi_{X_V} = \BR_{X_V} \oplus D_V$ for a Nygaard-Ogus lifting.   
\end{remark}

\subsection{Construction of Liftable Isogenies} We now prove Thm~\ref{thm: crys-isog}. \\
\textit{Proof of Theorem~\ref{thm: crys-isog}.}
Again write $\iota^p$ and $\iota_p$ for the prime-to-$p$ and crystalline component of $\iota$.
If a Frobenius-preserving isometric embedding $\iota_p : H_p \into \H^2(X/W)$ as in the hypothesis exists, then the K3 crystal $H_p$ has to be abstractly isomorphic to $\H^2(X/W)$, and hence to $\IK(\BR_X)$. We choose an isomorphism $H_p \iso \IK(\BR_X)$ and consider $(\iota_p)_K := \iota_p \tensor K$ as an isometric automorphism of the F-isocrystal $\IK(\BR_X)_K$. Then $(\iota_p)_K$ determines, and is conversely determined by, a pair $(h, g)$ where $h \in \End(\BR_X)[1/p] $ and $g \in \End(D)[1/p]$. Our goal is to produce an isogeny $f : \fh^2(X') \to \fh^2(X)$ for some other K3 surface $X'$ over $k = \bar{\bF}_p$ such that $f_*(\H^2_\et(X', \what{\bZ}^p)) = \iota^p(\Lambda \tensor \what{\bZ}^p)$ and $f_*(\H^2(X'/W)) = (\iota_{p})_K(\IK(\BR_X))$. By Thm~\ref{thm: existence}, we first reduce to the case when $\iota^p(\Lambda \tensor \what{\bZ}^p) = \H^2_\et(X, \what{\bZ}^p)$ and $(\iota_p)_K$ sends the slope $1$ part, i.e., $\ID(D^*)$, isomorphically onto itself. 

By Lubin-Tate theory, for some finite flat extension $V$ of $W$, there exists a lift $G_V$ of $\BR_X$ to $V$ such that $h$ lifts to $\End(G_V)[1/p]$ (\cite[Lem.~4.8]{Yang}). Note that $\Fil^\bullet \IK(G_V)_F$ equips $\IK(\BR_X)_K$ with the structure of an object in $\MF^{\varphi}_F$ and $\sM := \IK(G_V)_S$ defines a strongly divisible $S$-lattice in the corresponding object $\shD := \IK(\BR_X) \tensor S_K$ in $\mathcal{MF}^{\varphi, N}_{S_K}$. It is clear that $\iota_K$ preserves $\Fil^\bullet \IK(G_V)_F$ and extends to an automorphism $\iota_{S_K}$ of $\shD$. 

Let $X_V$ be the Nygaard-Ogus lifting of $X$ which corresponds to $G_V$. We have $T_\cris(\sM) = \H^2_\et(X_{\bar{F}}, \bZ_p)$ inside $\H^2_\et(X_{\bar{F}}, \bQ_p)$ by the proof of Theorem~\ref{thm: char NO lifting}, and that $V_\cris((\iota_p)_K)$ is an automorphism of the ${\Gal_F}$-module $\H^2_\et(X_{\bar{F}}, \bQ_p)$ which preserves the Poincar\'e pairing. The image of $\H^2_\et(X_{\bar{F}}, \bZ_p)$ under $V_\cris((\iota_p)_K)$ can also be interpreted as $T_\cris(\iota_{S_K}(\sM))$. Denote this ${\Gal_F}$-stable $\bZ_p$-lattice by $\Lambda'_p$. By Thm~\ref{thm: existence, char 0}, up to replacing $F$ by a finite extension, we can find another K3 surface $X'$ over $F$ with a derived isogeny $f : \fh^2(X') \stackrel{\sim}{\to} \fh^2(X_F)$ such that $$f_*(\H^2_\et(X'_{\bar{F}}, \what{\bZ})) =  \Lambda'_p \times \prod_{\ell \neq p}(\Lambda'_\ell := \iota^p(\Lambda \tensor \bZ_\ell) = \H^2_\et(X_{\bar{F}}, \bZ_\ell)). $$ 

We argue that $f$ induces an integral isomorphism $\Pic(X'_F) \stackrel{\sim}{\to} \Pic(X_F)$. Indeed, $f$ induces an isomorphism
$$ \Pic(X'_F) \stackrel{\sim}{\to} \Pic(X_F)_\bQ \cap \prod_\ell \Lambda'_\ell(1). $$
However, we know that the image of $\Pic(X'_F)$ lies in the unramified part of $\H^2_\et(X_{\bar{F}}, \bZ_p(1))$, and the unramified part of $\Lambda'_p(1)$ coincide with that of $\H^2_\et(X_{\bar{F}}, \bZ_p(1))$. This implies that the target of the above isomorphism is just $\Pic(X_F)$. 

It follows that $\Pic(X'_F)$ also satisfies hypothesis (a), (b) or (c) if $\Pic(X_F) \cong \Pic(X)$ does. For (a) and (c) this is clear, for (b) this follows from \cite[Lem.~2.3.2]{LMS}. In any case, by \cite[Thm 1.1]{Matsumoto}, \cite[\S 2]{Ito-LFunction}, and Thm~\ref{thm: GR} to be proved below, $X'_F$ admits potentially good reduction. Up to replacing $F$ by a further extension, we can find a smooth proper algebraic space $X'_V$ over $V$ such that $X'_F$ is the generic fiber of $X'_V$. The map induced on crystalline cohomology of special fibers is $D_\cris(f)$, which sends $\H^2(X'/W)$ onto $(\iota_p)_K(\IK(\BR_X))$. \qed

\section{Uniqueness Theorems}
\label{sec: uniqueness}

In this section we prove Thm \ref{thm: Torelli} by lifting to characteristic 0 (as outlined in the introduction).

\subsection{Shimura Varieties}
\label{sec: spin}

Let $p > 2$ be a prime and $L$ be any self-dual quadratic lattice over $\bZ_{(p)}$ of rank $m \ge 5$ and signature $(2+, (m - 2)-)$. Set $\wt{G}:= \CSpin(L_{(p)})$, $G := \SO(L_{(p)})$, $\cK_p := \CSpin(L \tensor \bZ_p)$, $\sfK_p:=\SO(L \tensor \bZ_p)$ and $\Ohm := \{ \w \in \mathbf{P}(L \tensor \bC) : \<\w, \w\> = 0, \< \w, \bar{\w} \> > 0 \}$. Let ${\wt{\shS}}_{\cK_p}(L)$ (resp. $\shS_{\sfK_p}(L)$) denote the canonical integral model of $\Sh_{\cK_p}(\wt{G}, \Ohm)$ (resp. $\Sh_{\sfK_p}(G, \Ohm)$) over $\bZ_{(p)}$ given by \cite{int} (see also \cite[\S 4]{CSpin}). We choose a compact open subgroup $\cK^p$ of $\wt{G}(\bA^p_f)$ and set $\cK = \cK_p \cK^p$. Similarly, set $\sfK^p$ to be the image of $\cK^p$ and $\sfK := \sfK_p \sfK^p$. Denote by $\wt{\Sh}_{\cK}(L), \wt{\shS}_{\cK}(L), \Sh_{\sfK}(L)$, and $ \shS_{\sfK}(L)$ the stacky quotients $\wt{\Sh}_{\cK_p}(L)/\cK^p, \wt{\shS}_{\cK_p}(L)/\cK^p, \Sh_{\sfK_p}(L)/\sfK^p$, and $\shS_{\sfK_p}(L)/\sfK^p$ respectively.

The model $\wt{\shS}_{\cK}(L)$ is equipped with a universal abelian scheme $\shA$ up to prime-to-$p$ isogeny whose cohomology gives rise to sheaves $\bH_*$ ($* = B, \cris, \ell, \dR$) on suitable fibers of ${\wt{\shS}}(L)$. The abelian scheme $\shA$ is equipped with a $\Cl(L)$-action and $\bZ/2 \bZ$-grading, and the sheaves $\bH_*$ are equipped with tensors $\bpi_* \in \bH_*^{\tensor (2, 2)}$. We call the triple of $\bZ/2\bZ$-grading, $\Cl(L)$-action and various realizations of $\pi$ the \textit{CSpin structures} on $\shA$ or $\bH_{*}$. The dual of the images of $\bpi_*$ are denoted by $\bL_*$. We refer the reader to \cite[\S 4]{CSpin} for details of these constructions or \cite[(3.1.3)]{Yang} for a quick summary. 

Here is another way to view the sheaves $\bL_*$: On the double quotient $\Sh_\sfK(L)_\bC = G(\bQ) \backslash \Ohm \times G(\bA_f)/ \sfK$, the standard representation $\SO(L) \to \GL(L)$ produces a variation of $\bZ$-Hodge structures (\cite[\S3.3]{CSpin}), which is nothing but $(\bL_B, \bL_{\dR, \bC} := \bL_\dR|_{\Sh_\sfK(L)_\bC})$. The filtered vector bundle $\bL_{\dR, \bC}$ is commonly called the automorphic vector bundle associated to this representation, and by the general theory of automorphic vector bundles, we know that it admits a canonical descent to the canonical model $\Sh_{\sfK}(L)$ over the reflex field $\bQ$. This canonical descent is nothing but $\bL_\dR$ (when restricted to $\Sh_\sfK(L)$). In fact, the pair $(\bL_B, \bL_{\dR, \bC})$ is the variation of $\bZ$-Hodge structures associated to a family of $\bZ$-motives $\bL$ over $\Sh_\sfK(L)_\bC$ in the sense of \cite[\S1.4]{Keerthi}. This family of motives descend to the canonical model $\Sh_\sfK(L)$, whose $\ell$-adic realizations give $\bL_\ell|_{\Sh_\sfK(L)}$ and whose de Rham realization gives $\bL_\dR|_{\Sh_\sfK(L)}$. Once we extend $\Sh_\sfK(L)$ to $\sS_\sfK(L)$ over $\bZ_{(p)}$, these sheaves arising from cohomological realizations of motives over $\Sh_\sfK(L)$ also extend. This motivic point of view is discussed in more detail in \cite[\S4.7]{Keerthi}.

It is explained in \cite[(3.1.3)]{Yang} that the sheaves $\bL_{*}$ are equipped with an orientation tensor $\bd_* : \underline{\det{(L)}} \stackrel{\sim}{\to} \det(\bL_*)$ ($* = B, \ell \neq p$). Here $\underline{\det{(L)}}$ denotes the constant sheaf whose stalks are $\det(L)$ on ${\wt{\shS}}(L)$ in the appropriate Grothendieck topology. In short, $\bd_*$'s come up because the adjoint representation of $\wt{G}$ on $L_{(p)}$ factors through $\SO(L_{(p)})$, i.e., it preserves a choice of orientation $\delta$ on $L_{(p)}$. It is possible to discuss de Rham or crystalline realizations of $\delta$, but for our purposes it suffices to use the $2$-adic realization $\bd_2$. The sheaves $\bL_*$ and the tensors $\bpi_*$ and $\bd_*$ descend to $\shS(L)$. 

We will repeatedly make use of the following key fact about $\bL_\dR$ and $\bH_\dR$: 
\begin{proposition}
\label{prop: LH fil}
Let $s$ be any point on ${\wt{\shS}}_{\cK}(L)$. $\Fil^1 \bL_{\dR, s}$ is one-dimensional, and $\Fil^1 \bH_{\dR, s} = \ker(x)$ for any nonzero element $x \in \Fil^1 \bL_{\dR, s}$. 
\end{proposition}
\begin{proof}
If $\mathrm{char\,}k(s) = 0$, we can simply base change to $\bC$ and apply Hodge theory (see \cite[p.~8]{Yang}). If $\mathrm{char\,}k(s) = p$, we can check this by a lifting argument. One can also read off this fact from \cite[\S4.9]{CSpin}.
\end{proof}

We recall the definition of a CSpin-isogeny (\cite[Def.~3.2]{Yang}): 
\begin{definition}
Let $\kappa$ be a perfect field with algebraic closure $\bar{\kappa}$ and $s, s'$ be $\kappa$-points on ${\wt{\shS}}_{K_p}(L)$, we call a quasi-isogeny $\shA_s \to \shA_{s'}$ a \textit{CSpin-isogeny} if it commutes with the CSpin structures, i.e., it respects the $\bZ/2\bZ$-grading, $\Cl(L)$-action and sends $\bpi_{\ell, s \tensor \bar{\kappa}}$ to $\bpi_{\ell, s' \tensor \bar{\kappa}}$ for every $\ell \neq \mathrm{char\,}\kappa$ and in addition $\bpi_{\cris, s}$ to $\bpi_{\cris, s'}$ if $\mathrm{char\,} \kappa = p$.
\end{definition}
We remark that CSpin-isogenies are stable under liftings and reductions:

\begin{lemma}
\label{lem: CSpin-isog}
Let $\kappa$ be a perfect field of characteristic $p$ and $s, s'$ be two $k$-points on ${\wt{\shS}}_\cK(L)$. Let $K$ denote $W(\kappa)[1/p]$ and $F \subseteq \bar{K}$ be a finite extension of $K$ and $s_F, s'_F$ be $F$-valued points on ${\wt{\shS}}_\cK(L)$ which specialize to $s, s'$. Suppose $\psi_F : \shA_{s_F} \to \shA_{s'_F}$ is a quasi-isogeny which specializes to $\psi : \shA_s \to \shA_{s'}$. Then $\psi$ is a CSpin-isogeny if and only if $\psi_F$ is also a CSpin-isogeny. 
\end{lemma}
\begin{proof}
Clearly, $\psi$ respects the $\bZ/2 \bZ$-grading and the $\Cl(L)$-actions if and only if $\psi_F$ also respects these structures. Let $s_{\bar{K}}$ and $s'_{\bar{K}}$ denote the $\bar{K}$-valued geometric points over $s_F$ and $s'_F$. To check whether $\psi_F$ sends $\bpi_{\ell, s_{\bar{K}}}$ to $\bpi_{\ell, s'_{\bar{K}}}$ for every $\ell$, it suffices to check this for one $\ell$, as one can always take a base change to $\bC$ and use Betti realizations. Therefore, the only part of the statement which does not follow directly from the smooth and proper base change theorem is that if $\psi_F$ is a CSpin-isogeny, then $\psi$ sends $\bpi_{\cris, s}$ to $\bpi_{\cris, s'}$. This follows from \cite[Rmk~3.1]{Yang}. 
\end{proof} 
\begin{lemma}
\label{lem: liftisog}
Let $s_\bC, s_\bC'$ be two $\bC$-points on ${\wt{\shS}}_\cK(L)$. For every Hodge isometry $$g: \bL_{B, s'_\bC} \tensor \bQ \sto \bL_{B,s_\bC} \tensor \bQ$$ which sends $\bd_{2, s_\bC}$ to $\bd_{2, s'_\bC}$, there exists a CSpin-isogeny $\shA_{s_\bC} \sto \shA_{s'_\bC}$ which induces $g$ by conjugation. 
\end{lemma}
\begin{proof}
From the construction of the local system $\bH_{B}$ (see \cite[\S 3.3]{CSpin}) it is clear that there exists an isomorphism of free $\bZ_{(p)}$-modules $H \sto \bH_{B, s_\bC}$ which respects the CSpin-structures, i.e., it respects the $\bZ/2\bZ$-grading, $\Cl(L)$-action and sends $\pi$ to $\bpi_{B, s_\bC}$. The same is true for $s'_\bC$, so there exists an isomorphism of $\bZ_{(p)}$-modules $\psi' : \bH_{B, s_\bC} \sto \bH_{B, s'_\bC}$ which respects the CSpin structures. The map $g' : \bL_{B, s_\bC} \tensor \bQ \sto \bL_{B, s'_\bC} \tensor \bQ$ induced by $\psi'$ by conjugation sends $\bd_{2, s_\bC}$ to $\bd_{2, s'_\bC}$. Therefore, the composition $g^{-1} \circ g'$ lies in $\SO(\bL_{B, s_\bC} \tensor \bQ)$. Since the natural morphism $\CSpin(L_\bQ) \to \SO(L_\bQ)$ is surjective, we may lift $g^{-1} \circ g'$ to an automorphism of $\bH_{B, s_\bC}$ which preserves the CSpin structures and use it to adjust $\psi'$ to obtain a morphism $\psi$ which induces $g$ by conjugation. It follows from Prop.~\ref{prop: LH fil} that $f$ preserves the Hodge structures, so that $\wt{g}$ comes from a CSpin-isogeny. 
\end{proof}

\subsection{Hilbert Squares and Period Morphisms}
\label{sec: period} We will apply the period morphism construction to Hilbert squares of K3 surfaces, so we recollect some basic facts and set up some notations here. Let $k$ be any algebraically closed field of characteristic $0$ or $p > 2$, $X$ be any K3 surface over $k$ and $Y := X^{[2]}$ be the Hilbert scheme of $2$ points on $X$. The lemma below implies that $Y$ is a $\mathrm{K3}^{[2]}$-type variety in the sense of \cite[Def.~1]{Yang2}.

\begin{lemma}
    When $\mathrm{char\,} k = p > 2$ or $0$, $Y$ has the same Hodge numbers as those of a complex $\mathrm{K3}^{[2]}$-type variety, and the Hodge-de Rham spectral sequence of $Y$ degenerates at the $E_1$-page. 
\end{lemma}
\begin{proof}
Let $Y':=\mathrm{Bl}_\Delta(X \times X)$ be the blow-up of $X\times X$ along the diagonal $\Delta\subset X\times X$. Let $E\subset Y'$ be the exceptional divisor, which is isomorphic to the projectivization of the tangent bundle of $X$. There is an action of $\mathbf{Z}/2$ on $X\times X$ given by permuting the factors; this action lifts to an action on $Y'$ which is trivial on $E$, and there is a natural map $q:Y'\to Y$ which identifies $Y$ with the quotient $Y'/(\mathbf{Z}/2)$. The map $q$ is a double cover branched over the divisor $D=q(E)\subset Y$, which may be described explicitly as the locus of non--reduced subschemes. Using our assumption that 2 is invertible in $k$, we obtain a canonical direct sum decomposition
\[
    q_*\ms O_{Y'}=\ms O_Y\oplus \ms L
\]
where $\ms L$ is the cokernel of the pullback map $\ms O_Y\to q_*\ms O_{Y'}$. From this and the projection formula we deduce the equality
\[
    \H^j(Y',q^*\Omega^i_{Y})=\H^j(Y,\Omega^i_Y)\oplus\H^j(Y,\Omega^i_Y\otimes\ms L)
\]
All of these data may be defined in a flat family over a flat finite type $\mathbf{Z}$-scheme. By semi-continuity, the dimensions of both summands on the right hand side must be greater than or equal to their corresponding values over the complex numbers. Thus, it will suffice to verify that the groups $\H^j(Y',q^*\Omega^i_{Y})$ have the same dimensions as over the complex numbers. This can be done via a direct computation. In more detail, we compute using the identification $q^*\Omega^1_{Y}\cong\Omega^1_{Y'}(-E)$, which yields isomorphisms
\[
    q^*\Omega^i_Y\cong\Omega_{Y'}^i(-iE)
\]
The cohomology of these sheaves may be related to the Hodge cohomology of $X$ by pushing forward along the blow-up morphism $Y'\to X\times X$. The result then follows (eventually) from the fact that the Hodge numbers of $X$ do not depend on the characteristic of the ground field.



The degeneration of the Hodge-de Rham spectral sequence at the $E_1$-page follows from the fact that $\H^i(Y, \Ohm_Y^j) = 0$ for $i + j$ odd. 
\end{proof}

Let $\sfH^*(-)$ be a Weil cohomology with coefficient field $\mathcal{K}$.\footnote{Here we are using a different font for $\H$ to distinguish from the $\H^*(-)$ in Notation~\ref{ssec:notation}.} We will only make use of Betti, $\ell$-adic, crystalline, de Rham when appropriate. When there is a specified polarization, let $\sfP^*(-)$ denote the corresponding primitive cohomology. We will view $\NS(Y)$ as a $\bZ$-lattice inside $\sfH^2(Y)$ via $c_1$, and will not write $c_1$ explicitly. $\sfH^2(Y)$ is equipped with natural \textit{Beauville-Bogomolov forms} (BBF). When $\mathrm{char\,} k = 0$, these forms are well known. When $\mathrm{char\,}k = p > n + 1$, the \'etale and crystalline versions of these forms for $\mathrm{K3}^{[n]}$-type varieties were defined by the second author in \cite[\S 2.1]{Yang2}. Since $Y$ is a Hilbert square on a K3 surface $X$, as opposed to a general deformation of such a variety, the Beauville-Bogomolov form on $Y$ is easily described by the Poincar\'e pairing on $X$: Let $\delta$ be the class of the exceptional divisor. Then $\delta^2 = -2$ under the BBF. The incidence correspondence between $X$ and $Y$ embeds $\sfH^2(X)$ isometrically into $\sfH^2(Y)$ such that $\sfH^2(Y)$ admits a natural orthogonal decomposition $\sfH^2(X) \oplus \mathcal{K} \delta$. Similarly, $\NS(Y)$ decomposes as $\NS(X) \oplus \bZ \delta$. 

\begin{lemma}
\label{eqn: extend isometry} Let $\xi$ be a polarization on $X$ and $\zeta$ be a polarization on $Y$ of the form $m \xi - \delta$. Denote by $\mathrm{proj}_{\sfP^2(Y)} \delta$ the projection of $\delta$ to $\sfP^2(Y)$, and $\mathrm{Isom}(-, -)$ the set of isometries between two quadratic lattices. Now let $X'$ be another K3 surface over $k$, take $Y', \xi', \delta'$ similarly, and suppose $Y'$ is polarized by $\zeta' := m \xi - \delta'$. There are natural identifications
\begin{align}
    \Isom(\sfP^2(X), \sfP^2(X')) &= \{ f \in \Isom(\sfH^2(X), \sfH^2(X') : f(\xi) = \xi' \} \nonumber \\
    &= \{ f \in \Isom(\sfH^2(Y), \sfH^2(Y')) : f(\zeta) = \zeta', f(\delta) = \delta' \} \nonumber \\
    &= \{ f \in \Isom(\sfP^2(Y), \sfP^2(Y') : f(\mathrm{proj}_{\sfP^2(Y)} \delta) = \mathrm{proj}_{\sfP^2(Y')} \delta' \}. 
\end{align}
\end{lemma}

Assume now $p \ge 5$ to apply the results of \cite{Yang2}. Let $X$ be a K3 surface and $Y:= X^{[2]}$. Let $\zeta$ be any primitive polarization on $Y$ such that $p$ is prime to the top intersection number $\zeta^4$. Let $\Def(Y; \zeta)$ denote the deformation functor of the pair $(Y, \zeta)$, i.e., the functor which sends an Artin $W$-algera $A$ to the set of isomorphism classes of the flat deformations of $(Y, \zeta)$ over $A$. We have that $\Def(Y; \zeta)$ is representable by a formal scheme isomorphic to $\mathrm{Spf}(\cR)$ for $\cR := W[\![x_1, \cdots, x_{20}]\!]$. Let $(\sY, \bzeta)$ denote the universal family over $\Def(Y; \zeta)$. Note that $\bzeta$ algebraizes $\sY$ into a scheme over $\mathrm{Spec}(\cR)$. Again we use the symbol $\PH^2(-)$ for the primitive cohomologies of $(Y, \zeta)$. There are natural pairings on $\PH^2(Y, \what{\bZ}^p)$ and $\PH^2(Y/W)$ given by restricting the Beauville-Bogomolov forms (see \cite[\S2.1]{Yang2}).  

Let $F \subset \bar{K}$ be any finite extension of $K$ and $\wt{b}$ be any $\cO_{F}$-point on $\Def(X; \zeta)$. Choose an isomorphism $\iota : \bar{K} \iso \bC$. Let $L$ be the quadratic lattice $\PH^2(\sY_{\wt{b}}(\bC), \bZ_{(p)})$, equipped with the restriction of the negative Beauville-Bogomolov form. We remark that since $\H^2(\sY_{\wt{b}}(\bC), \bZ)$ is always isomorphic to the lattice $\Lambda^{[2]} := \Lambda \oplus \bZ(-2)$ and $p \nmid  c_1(\bzeta_{\wt{b}_\bC})^2 $, the isomorphism class of $L$ as a quadratic lattice over $\bZ_{(p)}$ is completely determined by the number $c_1(\bzeta_{\wt{b}_\bC})^2$ (\cite[I Lem.~4.2]{Milnor}). 

Let $b$ be the closed point of $\Def(X, \zeta)$. We pack the input we need from the the Kuga-Satake period morphism into the following proposition: 

\begin{proposition}
\label{prop: compMotive} 
Assume $p \ge 5$. There exists a local period morphism $\rho : \mathrm{Spec\,} \cR \to \shS_{\sfK_p}(L)$ which identifies $\mathrm{Spec\,}\cR$ with the complete local ring $\what{\cO}_s$ of $s := \rho(b)$ on $\shS(L)_W$ such that: 
\begin{enumerate}[label=\upshape{(\alph*)}]
    \item There are isometries $\alpha_{\dR} : \bP^2_\dR \sto \rho^* \bL_{\dR}(-1)$ of filtered vector bundles, and $\alpha_{\cris} : \bP^2_\cris \sto \rho^* \bL_{\cris}(-1)$ of F-crystals which are compatible via the crystalline-de Rham comparison isomorphisms;
    \item There is an isometry $\alpha_{\bA_f, b} : \PH^2_\et(\sY_{b, \bA_f}) \to \bL_{\bA_f, b}$ such that for any geometric $\wt{b}'$ of characteristic zero on $\mathrm{Spec\,} \cR$, the pair of isometries $(\alpha_{\bA_f, \wt{b}'}, \alpha_{\dR, \wt{b}'})$, where $\alpha_{\bA_f, \wt{b}'} : \PH^2_\et(\sY_{\wt{b}'}, \bA_f) \to \bL_{\bA_f, \wt{b}'}$ is induced by the smooth and proper base change theorem, is absolute Hodge. 
\end{enumerate}
Moreover, for any choice of trivialization $\epsilon_2 : \det(L \tensor \bQ_2) \sto \det(\PH^2_\et(Y, \bQ_2))$, $s$ can always be chosen such that $\det(\alpha_{2, b})$ sends $\epsilon_2$ to $\bd_{2, s}$. 
\end{proposition}
\begin{proof}
See \cite[\S3.3]{Yang2}, which is a direct generalization of the results in \cite[\S5]{Keerthi}.
\end{proof}

\begin{remark}
We remark that in order to construct the local period morphism $\rho$, we actually have to choose an appropriate $\bZ$-integral structure for the $\bZ_{(p)}$-lattice $L$. However, once it is constructed, we are allowed to forget about the $\bZ$-integral structure, as the integral models of the relevant Shimura varieties only depend on the $\bZ_{(p)}$-lattice $L$. 
\end{remark}

\subsection{Twisted Derived Torelli Theorem}
\label{sec: derived Torelli}
\begin{definition}
\label{def: liftable}
Let $X, X'$ be K3 surfaces over an algebraically closed field $k$ of characteristic $p > 0$. Let $f : \fh^2(X') \sto \fh^2(X)$ be an isogeny. We say that $f$ is \textit{liftable} if for some finite extension $F$ of $K$ with $V := \cO_F$ and projective schemes $X_V$ and $X'_V$ over $V$ which deform $X$ and $X'$ to $V$, $f$ lifts to an isogeny $f_F : \fh^2(X'_F) \sto \fh^2(X_F)$. If $X, X'$ are non-supersingular, we say that $f$ is \textit{perfectly liftable}, if $X_V$ and $X_V'$ can be chosen to be perfect liftings . 
\end{definition}

For the rest of \S\ref{sec: derived Torelli}, let $k$ be an algebraically closed field of $p \ge 5$. 
\begin{lemma}
\label{lem: liftable}
Let $(X_0, \xi_0), \cdots, (X_m, \xi_m)$ be finitely many non-supersingular polarized K3 surfaces over $k$ and let $f_i : \fh^2(X_i) \stackrel{\sim}{\to} \fh^2(X_{i + 1})$ be a perfectly liftable isogeny which sends $\xi_i$ to $\xi_{i + 1}$ for $i = 0, 1, \cdots, m - 1$. If $f := f_{m - 1} \circ \cdots \circ f_0 : (\fh^2(X_0), \xi_0) \sto (\fh^2(X_m), \xi_m)$ induces an integral isomorphism $\H^2_\cris(X_0/W) \sto \H^2_\cris(X_m/W)$, then $f$ is perfectly liftable to $K$ up to equivalence.  
\end{lemma}

\begin{proof}
Set $Y_i := X_i^{[2]}$ and let $\delta_i$ be the exceptional divisor on $Y_i$. For some number $N \gg 0$, $\zeta_i := p^N \xi_i - \delta_i$ is a polarization on $Y_i$ for each $i$. The number $\< \zeta_i, \zeta_i\>$ under the Beauville-Bogomolov form on $Y_i$ is an integer $M$ which is independent of $i$. Let $L$ denote a $\bZ_{(p)}$-lattice which is isomorphic to the orthogonal complement of an element $\lambda \in \Lambda^{[2]} \tensor \bZ_{(p)}$ with $\< \lambda, \lambda \> = M$. We choose trivializations $\epsilon_i : \det(L \tensor \bQ_2) \sto \det(\PH^2_\et(Y_i, \bQ_2))$ such that $f_i$ sends $\epsilon_i$ to $\epsilon_{i + 1}$. Let $\rho_i$ denote a local period morphism obtained by applying Prop.~\ref{prop: compMotive} to $(Y_i, \zeta_i)$ and $\epsilon_i$ and let $s_i$ denote the image of the base point under $\rho_i$. Let $\wt{s}_i$ be a lift of $s_i$ to $\wt{\shS}_{\cK_p}(L)$. 

We claim that there exists a CSpin-isogeny $\psi_i : \shA_{\wt{s}_i} \to \shA_{\wt{s}_{i + 1}}$ which induces the same isometries $\bL_{\ell, s_i} \sto \bL_{\ell, s_{i + 1}}$ and $\bL_{\cris, s_i} \sto \bL_{\cris, s_{i + 1}}$ as $f_i$ for each $i = 0, \cdots, m - 1$. Indeed, fix an $i$ and let $X_{i, V}, X_{i + 1, V}$ be perfect liftings of $X_i, X_{i + 1}$ over some finite extension $V$ of $W$ such that $f$ lifts to $f_F : X_{i, F} \sto X_{i + 1, F}$, where $F = V[1/p]$. Let $Y_{i, V}, Y_{i + 1, V}$ be the Hilbert squares of $X_{i, V}, X_{i + 1, V}$. Note that $Y_{i, V}$ and $Y_{i + 1, V}$ carry liftings of $\zeta_i$ and $\zeta_{i + 1}$, so via the local Torelli morphisms $\rho_i, \rho_{i + 1}$, $X_{i, V}$ and $X_{i + 1, V}$ induce $V$-points $s_{i, V}, s_{i+1, V}$ on $\shS_{\sfK}(L)$. Lift these points to $V$-points $\wt{s}_{i, V}, \wt{s}_{i+1, V}$ on $\wt{\shS}_{\cK}(L)$, which is \'etale over $\shS_{\sfK}(L)$. Now choose an isomorphism $\bar{F} \iso \bC$. The isogeny $f_{i, F}(\bC)$ induces a Hodge isometry $\PH^2(X_{i,F}(\bC), \bQ) \sto \PH^2(X_{i+1, F}(\bC), \bQ)$, which canonically extends to a Hodge isometry $\PH^2(Y_{i,F}(\bC), \bQ) \sto \PH^2(Y_{i+1, F}(\bC), \bQ)$ via Lem.~\ref{eqn: extend isometry}. By Prop.~\ref{prop: compMotive}, the latter can be identified with a Hodge isometry $\bL_{B, s_{i, F}(\bC)} \tensor \bQ \sto \bL_{B, s_{i + 1, F}(\bC)}\tensor \bQ$. Note that we have required that $f_i$ sends $\epsilon_i$ to $\epsilon_{i + 1}$. By \ref{lem: liftisog}, we obtain a CSpin-isogeny $\psi_{i, \bC} : \shA_{\wt{s}_{i, F}(\bC)} \sto \shA_{\wt{s}_{i + 1, F}(\bC)}$. By Lem.~\ref{lem: CSpin-isog}, $\psi_{i, \bC}$ specializes to a CSpin-isogeny $\psi_i$, which can be easily checked to have the desired properties. 

By \cite[Cor.~4.2]{LM}, we can find a lifting $X_{0, W}$ of $X_0$ which also lifts all line bundles on $X_0$. We transport the induced Hodge filtration on $\H^2_\cris(X_0/W)$ to $\H^2_\cris(X_m/W)$ using $f$, which induces a lift $X_{m, W}$ of $X_m$ over $W$. It is easy to check that $X_{m, W}$ also carries liftings of all line bundles on $X_m$ using \cite[Prop.~1.12]{Ogus}. Just as in the previous paragraph, after taking Hilbert squares of the liftings, we obtain via the local period morphisms $K$-valued points $s_{0, K}$, $s_{m, K}$, $\wt{s}_{0, K}, \wt{s}_{m, K}$ which lift $s_0, s_m, \wt{s}_0, \wt{s}_m$. It follows from Prop.~\ref{prop: LH fil} that the crystalline realization of $\psi := \psi_{m - 1} \circ \cdots \circ \psi_0$ preserves the Hodge filtrations of $\shA_{s_{0, K}}$ and $\shA_{s_{m, K}}$ via the Berthelot-Ogus comparison isomorphisms. By \cite[Thm~3.15]{BO} $\psi$ lifts to a CSpin-isogeny $\psi_K : \shA_{s_{0, K}} \sto \shA_{s_{m, K}}$. Choose an isomorphism $\bar{K} \cong \bC$. By running the arguments in the preceeding paragraph backwards, we obtain a rational Hodge isometry $\H^2(X_{0, K}(\bC), \bQ) \sto \H^2(X_{m, K}(\bC), \bQ)$, which by Huybrechts' theorem \cite[Thm~0.2]{Huy} is induced by an isogeny $f_\bC$. We get the desired isogeny $f$ by specializing $f_\bC$.
\end{proof}

\begin{proof}[Proof of Theorem \ref{thm: Torelli}.]
The forward direction is immediate (and does not need the restriction on $p$). For the converse, suppose that $f : \fh^2(X') \sto \fh^2(X)$ is polarizable and $\bZ$-integral. It is easy to see that if $X$ is supersingular, then so is $X'$. In this case, the result follows from Ogus' crystalline Torelli theorem \cite[Thm II]{Ogus2} (cf. \cite[Thm 6.5]{Yang}). Therefore, we reduce to the case when $X$ and $X'$ have finite height. We first remark that $f$ maps $\NS(X')$ isomorphically onto $\NS(X)$, so that by the structure of ample cones of K3 surfaces \cite[Prop.~1.10]{Ogus2}, $f(\xi')$ is ample for any ample $\xi$. 
By definition, there exists a sequence of K3 surfaces $X' = X_0, \cdots, X_m = X$ over $k$ and primitive derived isogenies $f_i : \fh^2(X_i) \sto \fh^2(X_{i + 1})$ such that $f = f_{m - 1} \circ \cdots \circ f_0$.

We now show that, there exists a sequence $\delta_i : \fh^2(X_i) \sto \fh^2(X_i)$ given by compositions of reflections in $(-2)$-curves up to a sign and a sequence of ample classs $\xi_i \in \NS(X_i)_\bQ$ such that $ (\delta_{i + 1} \circ f_i)(\xi_i) = \xi_{i + 1}$ for each $i$. We do this by slightly refining the argument of \cite[Lem.~6.2]{Yang}. Set $\delta_0$ to be the identity. Choose any ample class $\zeta_0 \in \NS(X_0)_\bQ$ and $\epsilon_0 > 0$, such that the open ball $B(\zeta_0, \epsilon_0)$ centered at $\zeta_0$ of radius $\epsilon_0$ in $\NS(X_0)_\bR$ lies inside the ample cone. By \cite[Lem.~7.9]{Ogus}, there exists some $\delta_1$, such that $\zeta_1' := \delta_1 \circ f_0 (\zeta_0)$ is big and nef. The image of $B(\zeta_0, \epsilon_0)$ in $\NS(X_1)_\bR$ under $\delta_1 \circ f_0$ is an open neighborhood of $\zeta_1'$ which necessarily intersects the ample cone of $X_1$. Therefore, we may now choose $\zeta_1$ together with $\epsilon_1 > 0$ such that $(\delta_1 \circ f_0)^{-1} B(\zeta_1, \epsilon_1) \subseteq B(\zeta_0, \epsilon_0)$. We iterate this process to obtain a sequence of open balls $B(\zeta_i, \epsilon_i) \subset \NS(X_i)_\bR$ which lie inside the ample cones, and a sequence of $\delta_i$'s such that $(\delta_{i + 1} \circ f_i)^{-1}(B(\zeta_{i + 1}, \epsilon_{i + 1})) \subseteq B(\zeta_{i}, \epsilon_i)$. Now we win by choosing an element $\xi_m \in B(\zeta_m, \epsilon_m)$ and iteratively set $\xi_i := (\delta_{i + 1} \circ f_i)^{-1}(\xi_{i + 1})$. By clearing denominators we may assume that each $\xi_i$ is integral.

Set $\xi = \xi_m$, $\xi' = \xi_0$, $h_i := \delta_{i + 1} \circ f_i$ for each $i < m$ and $f' := h_{m - 1} \circ \cdots \circ h_0 = f$. For each $i$, consider $T(X_i) := \NS(X_i)^\perp \subset \H^2(X_i)$. Clearly $f_i$ and $h_i$ induce the same maps on transcedental lattices $\mathrm{T}(X_i)_\bQ \sto \mathrm{T}(X_{i + 1})_\bQ$. Therefore, $f$ and $f'$ induce the same maps $\mathrm{T}(X')_\bQ \sto \mathrm{T}(X)_\bQ$ but their induced maps $\NS(X') \sto \NS(X)$ may differ by an automorphism of $\NS(X)$ which preserves the ample cone. By Thm~\ref{thm:lifting isogenies}, each $h_i$ is liftable, so that by Lem.~\ref{lem: liftable}, $f' : \fh^2(X') \sto \fh^2(X)$ admits a perfect lifting $f'_K : \fh^2(X'_K) \sto \fh^2(X_K)$. Therefore, $f'$, and hence $f$, lifts to a Hodge isometry $\H^2(X'_K(\bC), \bQ) \sto \H^2(X_K(\bC), \bQ)$ for a chosen isomorphism $\bar{K} \cong \bC$. Using the smooth and proper base change theorem for \'etale cohomology, we see that this rational Hodge isometry is $\bZ[1/p]$-integral. Now we show that see that it is $\bZ$-integral. Indeed, we first note that $f$ induces isomorphism $f_p : \H^2_\et(X'_{\bar{K}}, \bQ_p) \sto \H^2_\et(X_{\bar{K}}, \bQ_p)$ and $f_\cris : \H^2_\cris(X'/W)[1/p] \sto \H^2_\cris(X/W)[1/p]$. We have $f_p \tensor_{\bZ_p} B_\cris = f_\cris \tensor_W B_\cris$ under the $p$-adic comparison isomorphism (see (\ref{eq: p-adic comp}) in the appendix) as it is compatible with cycle class maps, Poincar\'e duality and trace maps (\cite[Cor.~11.6]{IIK}). Let $S$ be Breuil's $S$-ring. Then we have an identification $\H^2_\cris(X/W) \tensor_W S = \H^2_\cris(X/S)$ and a similar one for $X'$. Now, we are given that $f_\cris \tensor_W B_\cris$ sends the $S$-module $\H^2_\cris(X'/S)$ isomorphically onto $\H^2_\cris(X/S)$. By \cite[Thm 5.2]{CaisLiu} (see also Thm~\ref{thm: Cais-Liu} and Rmk~\ref{rem: Tcris} below), $f_p$ sends the $\bZ_p$-lattice $\H^2_\et(X'_{\bar{K}}, \bZ_p)$ isomorphically onto $\H^2_\et(X_{\bar{K}}, \bZ_p)$. Therefore, we have shown that $f$ in fact induces an integral Hodge isometry $\H^2(X'_K(\bC), \bZ) \sto \H^2(X_K(\bC), \bZ)$ which preserves the ample cones. We may now conclude using the global Torelli theorem and \cite[Thm 2]{MM}.
\end{proof}

\section{Isogenies and Hecke Orbits}
\label{sec: Hecke}
We briefly recall the definition of prime-to-$p$ Hecke orbit on the orthogonal Shimura varieties. Let $\Lambda$ be the K3 lattice $\mathbf{U}^{\oplus 3} \oplus \mathbf{E}_8^{\oplus 2}$, $\lambda \in \Lambda$ be a primitive element with $d := \lambda^2$ and $p > 2$ be a prime such that $p \nmid d$. We shall use the same notations for orthogonal and spinor Shimura varieties as in \S\ref{sec: spin} with $L = L_d$ and fix $\sfK_p = G(\bZ_p)$. The only difference is that this time $L_d$ has a $\bZ$-structure, so that the sheaf $\bL_{\bA^p_f}$ also has a $\what{\bZ}^p$-structure. Let $\sfK_0^p$ denote the image of $\CSpin(L_d \tensor \what{\bZ}^p)$ in $G(\bA^p_f)$. More concretely, $\sfK^p_0$ can be described as the maximal subgroup of $\SO(L_d \tensor \what{\bZ}^p)$ which acts trivially on the discriminant group $\mathrm{disc}(L_d \tensor \what{\bZ}^p) = \mathrm{disc}(L_d)$. A more helpful alternative description for us is that $\sfK_0^p$ can be viewed as the stabilizer of the element $\lambda \tensor 1$ of $\SO(\Lambda \tensor \what{\bZ}^p)$, which can naturally be viewed as a subgroup of $G(\bA^p_f)$. 

The limit $\Sh_{\sfK_p}(L_d)$ is equipped with a (right) $G(\bA_f^p)$-action. By the extension property of the canonical integral models, this action extends to $\shS_{\sfK_p}(L_d)$. Recall the complex uniformization of $\Sh_{\sfK_p}(L_d)$ 
$$ \Sh_{\sfK_p}(L_d)(\bC) = G(\bZ_{(p)}) \backslash \Ohm \times G(\bA^p_f), $$
where $\Ohm$ is the period domain parametrizing Hodge structures of K3 type on $L_d$ (\cite[\S3.1, 3.2]{CSpin}, \cite[Def.~3.1]{Yang}). Given a point $(\w, g) \in \Ohm \times G(\bA^p_f)$ and an element $g' \in G(\bA^p_f)$, $g'$ sends the class of $(\w, g)$ in $\Sh_{\sfK_p}(L_d)(\bC)$ to that of $(\w, gg')$. 
Let $k$ be an algebraically closed field of characteristic $0$ or $p$. Let $\Mod_{2d, \sfK_p}$ be the moduli stack over $\bZ_{(p)}$ of oriented quasi-polarized K3 surfaces of degree $2d$ with hyperspecial level structure at $p$ (see \cite[3.3.4]{Yang}, where it is denoted by $\wt{\Mod}_{2d, \sfK^{\mathrm{ad}}_p, \mathbb{Z}_{(p)}}$). By the modular interpretation of $\Mod_{2d, \sfK_p}$, $\Mod_{2d, \sfK_p}(k)$ is in natural bijection with the set of tuples $(X, \xi, \epsilon, \eta)$, where 
\begin{itemize}
    \item $(X, \xi)$ is a quasi-polarized K3 surface of degree $2d$ over $k$,
    \item $\epsilon$ is an isometry $\det(L_d \tensor \bQ_2) \sto \bP^2_\et(X, \bQ_2)$, which naturally extends to an isometry $\epsilon^p : \det(L_d \tensor \bA^p_f) \sto \bP^2_\et(X, \bA^p_f)$,\footnote{For details on how to obtain this extension, see \cite[\S3.3.3]{Yang} or \cite[Cor.~3.3.7]{Yang2}.} 
    \item $\eta$ is an isometry $\Lambda \tensor \what{\bZ}^p \sto \H^2_\et(X, \what{\bZ}^p)$ which sends $\lambda \tensor 1$ to $c_1(\xi)$ and is compatible with the isometry $\epsilon^p$. 
\end{itemize}
Using these explicit descriptions, it is easy to write down the map $\Mod_{2d, \sfK_p}(\bC) \to \Sh_{\sfK_p}(L_d)(\bC)$ explicitly: Let $(X, \xi, \epsilon, \eta)$ be the tuple which corresponds to a point $s \in \Mod_{2d, \sfK_p}(\bC)$. Choose an isomorphism $\alpha : (\Lambda \tensor \bZ_{(p)}) \sto (\H^2(X, \bZ_{(p)}), c_1(\xi))$ which is compatible with $\epsilon$. Then $s$ is sent to the class of $(\w, \eta^{-1} \circ (\alpha \tensor \bA^p_f))$, where $\w$ is the Hodge structure on $L_d$ endowed by $\alpha$. This map is clearly well defined. The integral extension $\Mod_{2d} \to \shS_{\sfK_0}(L_d)$ is constructed and studied in \cite{Keerthi}. The reader can also look at \cite[\S3.3]{Yang} for a quick summary of the properties. 

\begin{theorem}
\label{thm: surj on HO}
Assume $\mathrm{char\,} k = p > 2$. If any point $x \in \shS_{\sfK_p}(L_d)(k)$ lies in the image of $\Mod_{2d, \sfK_p}(k)$, then so does $x \cdot g$ for any $g \in G(\bA^p_f)$.
\end{theorem}
\begin{proof}
Let $s \in \Mod_{2d, \sfK_p}(k)$ be a point such that $x = \rho_\sfK(s)$. Let $(X, \xi, \eta, \epsilon)$ be the tuple which corresponds to $s$. We view $G(\bA^p_f)$ as the subgroup of $\SO(\Lambda \tensor \bA^p_f)$ which fixes $\lambda \tensor 1$. 

By Thm~\ref{thm: existence} and Rmk~\ref{rmk: relax assumption}, there exists a K3 surface $X'$ together with a derived isogeny $f : \fh^2(X') \to \fh^2(X)$ such that $f_*(\H^2(X')) = \H^2_\cris(X/W) \times \mathrm{im}(g) \subset \H^2(X)_\bQ$. Moreover, $f$ is a composition of primitive derived isogenies which come from twisted derived equivalences involving Brauer classes of prime-to-$p$ order. Since $f_*(\NS(X')) = f_*(\H^2(X')) \cap \NS(X)_\bQ$, $\xi \in f_*(\NS(X'))$ so that $\NS(X')$ contains a primitive vector of degree $2d$. By \cite[Lem.~7.3]{Ogus}, we can find a derived auto-isogeny $\delta$ on $X$ which is given by reflections in $(-2)$-curves up to a sign such that $\delta \circ f$ sends $\xi$ to a quasi-polarization $\xi'$. Now we use $\delta \circ f$ to transport $(\epsilon, \eta)$ to similar structures $(\epsilon', \eta')$ on $(X', \xi')$ so that we obtain a point $s' \in \Mod_{2d, \sfK^p}(k)$. We claim that $\rho(s') = x \cdot g$. Although $\shS_{\sfK_p}(L_d)$ lacks a direct modular interpretation, we can do this by a lifting argument.

We claim that there exist liftings $(X_W, \xi_W), (X'_W, \xi'_W)$ of $(X, \xi)$ and $(X', \xi')$ together with an isogeny $(\fh^2(X'_K), \xi'_K) \to (\fh^2(X_K), \xi_K)$ whose \'etale realization agrees with $\delta \circ f$ via the smooth and proper base change theorem. If $X, X'$ are of finite height, by Thm~\ref{thm:lifting isogenies}, $\delta \circ f$ can be lifted to an isogeny on the nose. In the supersingular case, we first choose a lifting $(X_W, \xi_W)$. Then $X_W$ induces a Hodge filtration on $\H^2_\cris(X/W)$, which can be transported to a filtration on $\H^2_\cris(X'/W)$ lifting the one on $\H^2_\dR(X'/k)$. By the local Torelli theorem, this defines a lifting $X_W'$ of $X'$. One easily checks by \cite[Prop.~1.12]{Ogus} that $\xi'$ lifts to $X_W'$. Now we apply \cite[Lem.~4.3.5]{Yang2} and Thm~\ref{thm: existence, char 0}.

Liftings as above induce $W$-points $s_W$ and $s_W'$ on $\Mod_{2d, \sfK_p}$ which lift $s$ and $s'$.
Let $x_W := \rho(s_W), x'_W := \rho(s'_W)$. Using the $G(\bA^p_f)$-action, the lifting $x_W$ of $x$ induces a lifting $x''_W$ of $x'' := x \cdot g$. Using the complex uniformization one quickly checks that $x''_W \tensor \bC = x_W' \tensor \bC$ for any embedding $K \subset \bC$. Since $\shS_{\sfK_p}(L_d)$ is a limit of separated schemes, we conclude that $x' = x''$ as desired.
\end{proof}


Choose a small enough compact open $\sfK^p \subseteq \sfK_0^p$ such that for $\sfK := \sfK_p \sfK^p$, $\shS_\sfK(L_d)$ is a scheme and denote the period morphism $\Mod_{2d, \sfK} \to \shS_\sfK(L_d)$ by $\rho_\sfK$. For any $k$-point $x \in \shS_{\sfK}(L_d)$, the image of the $G(\bA^p_f)$-orbit of a lift $\wt{x} \in \shS_{\sfK_p}(L_d)(k)$ under the natural projection $\shS_{\sfK_p}(L_d) \to \shS_{\sfK}(L_d)$ is what we call the prime-to-$p$ Hecke orbit of $x$.


Let $\sX$ denote the universal family over $\Mod_{2d, \sfK}$. The the mod $p$ fiber $\Mod_{2d, \sfK, \bF_p}$ (resp. $\shS_\sfK(L_d)_{\bF_p}$) of moduli space $\Mod_{2d, \sfK}$ admits a stratification $\Mod_{2d, \sfK, \bF_p} = \Mod^{1} \supseteq \Mod^{2} \supseteq \cdots \supseteq \Mod^{20}$ (resp. $\shS_{\sfK}(L_d)_{\bF_p} = \shS^{1} \supseteq \shS^{2} \supseteq \cdots \supseteq \shS^{20}$) such that for $1 \le i \le 10$, a geometric point $s$ lies in $\Mod^i$ (resp. $\shS^i$) if and only if $\sX_s$ (resp. $\bL_{\cris, s}(-1)$) has height $\ge i$, and for $11 \le i \le 20$, a geometric point $s$ lies in $\Mod^i$ (resp. $\shS^i$) if and only if $\sX_s$ (resp. $\bL_{\cris, s}(-1)$) is supersingular and has Artin invariant $\le 21 - i$. Set $\mathring{\Mod}^i := \Mod^i - \Mod^{i - 1}$ and $\mathring{\shS}^i := \shS^i - \shS^{i - 1}$. Heights and Artin invariants are rather classical invariants. For a more modern interpretation in terms of Newton and Ekedahl-Oort (E-O) strata for $\shS_\sfK(L_d)_{\bF_p}$, see for example \cite[\S8.4]{Xu}. It follows from \cite[Cor.~5.14]{Keerthi} that the period morphism respects these stratifications in the sense that $\Mod^i = \shS^i \times_{\shS_\sfK(L_d)} \Mod_{2d, \sfK}$. We remark that the Zariski closure of the locally closed subscheme $\mathring{\shS}^i$ is $\shS^i$. By \cite[Cor.~7.2.2, Cor.~7.3.4]{XuZhang}, if $1 \le i \le 10$, $\shS^i$ is a central leaf. The locus $\shS^{20}$ is the superspecial locus (the unique closed E-O stratum), and is also a central leaf (see \cite[Rmk 3.2.2, Ex. 6.2.4]{XuZhang}).

In our case, the Hecke orbit conjecture predicts the following:

\begin{conjecture}
\label{conj: HO for i}
For $1 \le i \le 10$ or $i = 20$, the prime-to-$p$ Hecke orbit of every $s \in \shS^i(\bar{\bF}_p)$ is Zariski dense in $\shS^i$. 
\end{conjecture}
We remark that once the above conjecture is known for $\bar{\bF}_p$, it is automatically true for any algebraically closed field over $\bF_p$ by a specialization argument. Conj.~\ref{conj: HO for i} has been proved by Maulik--Shankar--Tang \cite[Thm~1.4]{maulik2020picard} when $i = 1$ and $p \ge 5$. We prove another special case below (Thm \ref{thm: Hecke orbit for superspecial}).

We use $N_\sigma$ to denote the supersingular lattice of Artin invariant $\sigma$. We restrict to considering $p > 2$ case, when these lattices are characterized by \cite[\S17, Prop.~2.20]{K3book}. The original reference \cite{RS2} also treated the $p = 2$ case. 

\begin{lemma}
\label{lem: superspecial trick}
For each $d > 0$ and $i = 0, 1$, there exist a primitive element $\xi \in N_1$ with $\xi^2 = 2d$, an $\alpha_i \in \O(N_1)$ such that $\alpha_i$ fixes $\xi$, interchanges the two isotropic lines in $(N_1^\vee / N_1) \tensor \bF_{p^2}$, and $\det(\alpha_i) = (-1)^i$. 
\end{lemma}
\begin{proof}
The supersingular K3 surface with Artin invariant $1$, which is unique up to isomorphism, is given by the desingularization of $A / A[2]$, where $A = E \times E$ for a supersingular elliptic curve $E$ (\cite[Cor.~7.14]{Ogus}). Since $E$ admits a model over $\bF_p$, so does $X$. Let $\varphi$ be a topological generator of $\Gal_{\bF_p}$. We fix an isomorphism between $N_1$ and $\NS(X_{\bar{\bF}_p})$, so that $N$ is equipped with a $\Gal_{\bF_p}$-action such that $\NS(X)$ is identified with the $\varphi$-invariants $N^\varphi$. 

Let $\NS(A)(2)$ denote the lattice $\NS(A)$ but with the quadratic form multiplied by a factor of $2$. As a result of the Kummer construction, there exist 16 $(-2)$-curves $\delta_1, \cdots, \delta_{16}$ on $X$ and an isometric embedding 
$$ \NS(A)(2) \oplus (\bigoplus_{i = 1}^{16} \bZ \delta_i) \into \NS(X). $$

Let $\mu \in \NS(A)(2)$ be a primitive element with $\mu^2 > 0$. For some coprime numbers $a, b$, $(a \mu + b \delta_1)^2 = 2d$. The generator $\varphi$ fixes $\xi := a \mu + b \delta_1$ and interchanges the isotropic lines in $(N_1^\vee / N_1) \tensor \bF_{p^2}$ (cf. the paragraph below \cite[Ex.~4.20]{Liedtke2}). 

Let $s_{\delta_2}$ be the reflection in $\delta_2$. Note that $s_{\delta_2}$ fixes $\mu$ and $\delta_1$, and hence $\xi$. Moreover, it is not hard to check that $s_{\delta_2}$ acts trivially on $N^\vee/N$. Therefore, we can simply set $\alpha_0, \alpha_1$ to be $\varphi$ and $s_{\delta_2} \circ \varphi$ up to permutation.
\end{proof}

\begin{lemma}
$\mathring{\Mod}^i \neq \emptyset$ for all $i$.  
\end{lemma}
\begin{proof}
Each $\Mod^{i + 1} \subseteq \Mod^i$ is locally cut out by a single equation. $\Mod_{2d, \sfK, \bF_p}$ is smooth of dimension $19$, and we know that $\Mod^{20}$ is zero-dimensional (cf. \cite[\S7]{Artin}). Therefore, it suffices to show that $\Mod^{20} \neq \emptyset$, i.e., there exists a quasi-polarization of degree $2d$ on the superspecial K3 surface, which is unique up to isomorphism. This follows from the preceeding lemma and \cite[Lem.~7.9]{Ogus}.
\end{proof}

Let $\cK \subset \wt{G}(\bA^p_f)$ be the preimage of $\sfK$. Before preceeding we recall that for any geometric point $t \in \wt{\shS}_{\cK}(L_d)$, there is a distinguished subspace $\LEnd(\shA_t)$ of $\End(\shA_t)$ which consists of the elements whose whose cohomological realizations lie in $\bL_{\bA^p_f, t}$ and $\bL_{\cris, t}$ (\cite[Def. 3.10]{Yang}, cf. \cite[Def. 5.11]{CSpin}). When $t$ is on the supersingular locus, the natural maps $\LEnd(\shA_t) \tensor \what{\bZ}^p \to \bL_{\ell, t}$ and $\LEnd(\shA_t) \tensor \bZ_p \to \bL_{\cris, t}^{F = 1}$ are isomorphisms (\cite[Prop.~3.2.3]{Yang2}). 

\begin{lemma}
\label{lem: glue lattice}
Let $k$ be an algebraically closed field with $\mathrm{char\,} k = p$. Let $x$ be a $k$-point on $\shS^i$ for some $i \ge 11$, $t$ be a $k$-point on $\wt{\shS}_{\cK}(L_d)$ which lifts $x$ and set $P := \LEnd(\shA_t)$. Then there exists a primitive element $\nu \in N_\sigma$ with $\sigma := 21 - i$ and $\nu^2 = 2d$ such that $P \cong \nu^\perp$. 
\end{lemma}
\begin{proof}
Let $\bZ \nu$ be a quadratic lattice of rank $1$ generated by $\nu$ with $\nu^2 = 2d$. By the theory of gluing lattices (see \cite[\S2]{Curt} for a quick summary), primitive extensions of $P \oplus \bZ \nu$ corresponds to the data $(G_1, G_2, \phi)$, where $G_1$, $G_2$ are subgroups of $\disc(P)$ and $\disc(\bZ \nu)$ and $\phi$ is an isometry $G_1 \iso G_2$. Therefore, constructing $N$ amounts to choosing appropriate $(G_1, G_2, \phi)$. 

In our case, we take $G_1$ to be the prime-to-$p$ part of $\disc(P)$, i.e., $\disc(P \tensor \what{\bZ}^p)$, and $G_2 = \disc(\bZ \nu)$ which is isomorphic to $\bZ / (2d) \bZ$ as an abelian group. Then we construct $\phi$ by a lifting argument: Let $x_W$ be a $W$-point on $\shS(L_d)$ which lifts $x$ and let $x_\bC$ be $x_W$ for some embedding $W \into \bC$. The period morphism $\rho_\sfK$ is known to be surjective on $\bC$-points, so there exists a quasi-polarized K3 surface $(X_\bC, \xi_\bC)$ such that the $\bZ$-Hodge structure $\bL_{B, x_\bC}$ is naturally identified with $\P^2(X_\bC, \bZ)$. We have that the natural map $P \tensor \what{\bZ}^p \to \bL_{\what{\bZ}^p, x}$ is an isomorphism (\cite[Prop.~3.2.3]{Yang2}) and $\bL_{\what{\bZ}^p, x} \cong \bL_{B, x_\bC} \tensor \what{\bZ}^p$ by smooth and proper base change and the Artin comparison isomorphisms. Therefore, there is an isomorphism $\beta_1 : G_1 \sto \disc(\P^2(X_\bC, \bZ) \tensor \what{\bZ}^p) = \disc(\P^2(X_\bC, \bZ))$. On the other hand, let $\beta_2 : G_2 \sto \disc(\bZ \xi_\bC)$ be the isomorphism given by sending $\nu$ to $\xi_\bC \in \H^2(X_\bC, \bZ)$. 

We may transport the gluing data given by the primitive embedding $\P^2(X_\bC, \bZ) \oplus \bZ \xi \subset \H^2(X_\bC, \bZ)$ to a gluing data $\phi$ for $G_1, G_2$ via $\beta_1, \beta_2$. Let $N$ be the lattice given by $(G_1, G_2, \phi)$. We check that it is a supersingular K3 lattice. Clearly, by our construction, $N \tensor \what{\bZ}^p \cong \Lambda \tensor \what{\bZ}^p$. As $P$ is negative definite, $N$ has signature $(1+, 21-)$. Finally, $\disc(N \tensor \bZ_p) = \disc(P \tensor \bZ_p) \cong (\bZ/ p \bZ)^{2\sigma}$ as an abelian group. Therefore, $N \cong N_\sigma$.
\end{proof}

We now prove another special case of Conj. \ref{conj: HO for i}: 
\begin{theorem}
\label{thm: Hecke orbit for superspecial}
Conj.~\ref{conj: HO for i} holds for $i = 20$. 
\end{theorem}
\begin{proof}
Take two $\bar{\bF}_p$-points $x, x' \in \shS^{20}$. Choose lifts $t, t'$ for $x, x'$ in $\wt{\shS}_{K}(L_d)$. We only need to show that there exists a CSpin-isogeny $\shA_t \to \shA_{t'}$ which is prime-to-$p$. Indeed, this follows from an explicit description of the isogenies classes in $\wt{\shS}_{\cK_p}(L_d)(\bar{\bF}_p)$ and their images on $\shS_{\sfK_p}(L_d)(\bar{\bF}_p)$ (\cite[\S3.2.3]{Yang}). Let $P$ and $P'$ denote $\LEnd(\shA_t)$ and $\LEnd(\shA_{t'})$ respectively. 

We first show that every isometry $P_\bQ \sto P'_\bQ$ whose induced isomorphism $\bL_{2, t} \tensor \bQ \sto \bL_{2, t'} \tensor \bQ$ sends $\delta_{2, t}$ to $\delta_{2, t'}$ is induced by a CSpin-isogeny $\psi : \shA_t \to \shA_{t'}$ by conjugation. Indeed, by \cite[Prop.~3.2.4]{Yang2}, there exists some CSpin-isogeny $\psi' : \shA_t \to \shA_{t'}$, which induces some isomorphism $P_\bQ \sto P'_\bQ$ whose induced isomorphism $\bL_{2, t} \tensor \bQ \sto \bL_{2, t'} \tensor \bQ$ sends $\delta_{2, t}$ to $\delta_{2, t'}$. The group of CSpin-isogenies from $\shA_{t}$ to itself is identified with $\CSpin(P_\bQ)$, which surjects to $\SO(P_\bQ)$. By composing $\psi'$ with some CSpin-isogeny $\shA_t \to \shA_t$, we get the desired $\psi$. 

We only need to show that there exists a CSpin-isogeny $\shA_t \to \shA_{t'}$ which is prime-to-$p$. By a Cartan decomposition trick \cite[Lem.~3.2.6]{Yang2}, we only need to show the following claim: \\\\
\textit{Claim: There exists an isometry $P \tensor \bZ_{(p)} \sto P' \tensor \bZ_{(p)}$ which sends $\delta_{2,x}$ to $\delta_{2, x'}$ and extends to an isomorphism $\bL_{\cris, x} \sto \bL_{\cris, x'}$.}

By Lem.~\ref{lem: glue lattice}, for some primitive vectors $\xi, \xi'$ in $N_1$ with $\xi^2 = (\xi')^2 = 2d$, we have $P \cong \xi^\perp$ and $P' \cong (\xi')^\perp$. Since some reflection of $N \tensor \bZ_{(p)}$ takes $\xi$ to $\xi'$ (\cite[I, Lem.~4.2]{Milnor}), $P \tensor \bZ_{(p)} \cong P' \tensor \bZ_{(p)}$ as quadratic lattice over $\bZ_{(p)}$. Now $P \tensor \bZ_p$ and $P' \tensor \bZ_p$ are the Tate modules of the supersingular K3 crystals $\bL_{\cris, t}(-1)$ and $\bL_{\cris, t'}(-1)$ respectively. By Ogus' theory of characteristic subspaces \cite[Thm 3.20]{Ogus}, $\bL_{\cris, t}$ (resp. $\bL_{\cris, t'}$) determines an isotropic line of $(P^\vee / P) \tensor \bar{\bF}_p$ (resp. $((P')^\vee/P') \tensor \bar{\bF}_p $) and isomorphism $P \tensor \bZ_p \to P' \tensor \bZ_p$ extends to an isomorphism $\bL_{\cris, t} \sto \bL_{\cris, t'}$ if and only these isotropic lines are respected. Now the claim follows from Lem.~\ref{lem: superspecial trick}. 
\end{proof}

\begin{remark}
As the reader can readily tell, the heart of the above theorem is the claim. Here we have proved the claim in a rather ad hoc way. We go through Lem.~\ref{lem: glue lattice} because there does not seem to be a good classification theory for quadratic lattices over $\bZ_{(p)}$. Moreover, $P$ and $P'$ are negative definite, so one cannot apply, say, Nikulin's theory to generate automorphisms, which only handles indefinite lattices. Luckily, in our special case, there is a geometric way of constructing the automorphisms we need. 
\end{remark}

\begin{lemma}
    Let $k$ be an algebraically closed field with $\mathrm{char\,} k = p > 2$. Let $R$ be a DVR over $k$ with fraction field $\kappa$ and let $X_{\kappa}$ be a supersingular K3 surface over $\kappa$ such that $X_{\bar{\kappa}}$ has Artin invariant $\sigma_0$. There exists a DVR $S$ over $k$ with fraction field $L$, a finite separable map $R\to S$, and an $N_{\sigma_0}$-marked supersingular K3 surface $X_S$ over $S$ such that $(X_S)_L\cong (X_{\kappa})_L$.
\end{lemma}
\begin{proof}
By a result of Rudakov and Shafarevich (see \cite[Thm 50]{RS}, and \cite[Thm 5.2.1]{BL} for $p=3$)  there exists a DVR S, a finite separable map $R\to S$, and a supersingular K3 surface $X_S$ over $S$ such that $(X_S)_L\cong (X_{\kappa})_L$. The Picard scheme $\uPic_{X_{L}}$ is formally \'{e}tale over $\Spec L$. As $\Pic(X_{\overline{L}})$ is finitely generated, after taking a further finite separable extension we may ensure that the restriction map $\Pic(X_L)\cong\Pic(X_{\overline{L}})$ is an isomorphism. Thus, $X_L$ admits an $N_{\sigma_0}$-marking. As $S$ is a DVR, we have $\Pic(X_L)=\Pic(X)$, so the generic marking extends uniquely to an $N_{\sigma_0}$-marking of $X_S$.
\end{proof}

\begin{theorem}
\label{thm: HO implies surj precise}
If Conj.~\ref{conj: HO for i} holds for $i$, or $i \ge 11$, then $\shS^i \subset \mathrm{im}(\rho_\sfK)$.
\end{theorem}
\begin{proof}
If Conj.~\ref{conj: HO for i} holds for $i$ then the conclusion is a direct consequence of Thm~\ref{thm: surj on HO} and the fact that $\mathrm{im}(\rho_\sfK)$ is open. Now assume $i \ge 11$ and take $k = \bar{\bF}_p$. Note that by Thm~\ref{thm: Hecke orbit for superspecial}, $\shS^{20} \subset \mathrm{im}(\rho_\sfK)$. Since the Zariski closure of $\mathring{\shS}^i$ is $\shS^i$, the $\mathrm{im}(\rho_\sfK) \cap \shS^i$ is open and dense in $\shS^i$. Take a closed point $x \in \mathring{\shS}^i_k$. Let $\cR$ be the ring $k[\![t]\!]$ and $\cF$ be its fraction field. Choose a $\cR$-valued point $\wt{x}$ which extends $x$ such that $\wt{x}_\cF$ lies in $\mathrm{im}(\rho_\sfK) \cap \mathrm{\shS}^i$. Such a $\wt{x}$ can always be found: we can always choose a smooth curve which passes through $x$ and whose generic point lies in $\mathrm{im}(\rho_\sfK) \cap \mathrm{\shS}^i$. Then we simply take the completion of this curve at $x$. Let $X_\cF$ be a supersingular K3 surface over the generic point of $\wt{x}_\cF$. Note that the geometric fiber of $X_\cF$ has Artin invariant $\sigma := 21 - i$. By the preceeding lemma, there exists a DVR $\cR'$ over $\cR$, whose fraction field $\cF'$ is a finite extension of $\cF$, such that there is an $N_\sigma$-marked supersingular K3 surface $\cX$ over $\cR'$. 

We argue that the special fiber $\cX_k$ of $\cX$ has Artin invariant $\sigma$. There are two families of supersingular K3 crystals over $\cR'$ (see \cite[\S5]{Ogus} for the definition):  One is obtained by pulling back $\bL_{\cris, \wt{x}}(-1)$ along $\cR \to \cR'$. The other is given by $\H^2_\cris(\cX')$. By construction, these two families agree on the generic fiber. By Prop.~4.6 and Thm~5.3 of \cite{Ogus}, there exists a universal family of supersingular K3 crystals over a smooth projective space $\cM$ such that these two families are both obtained by pullback the universal family along morphisms $\cR \to \cM$. Since $\cM$ is in particular separated, these two morphisms have to agree. Therefore, $\H^2_\cris(\cX'/\cR')$ is precisely the pullback of $\bL_{\cris, \wt{x}}(-1)$. Now we conclude by the hypothesis that $x \in \shS^i_k$. 

Now we know that $\cX_{\bar{\cF}} := \cX \tensor \overline{\cF}$ and $\cX_k$ have the same Artin invariant. This guarantees that the specialization map $\Pic(\cX_{\cF'}) \to \Pic(\cX_k)$ must be an isomorphism, and hence must send the ample cone isomorphically onto the ample cone. Since the big and nef cone is the closure of the ample cone, the quasi-polarization on $\cX_{\cF'}$ extends to a quasi-polarization on $\cX_k$. This shows that $x \in \mathrm{im}(\rho_\sfK)$. 
\end{proof}

Finally we discuss some implications of the surjectivity of the period morphism to the good reduction theory of K3 surfaces. As Conj. \ref{conj: HO for i} is known for $i=1$ and $p\geq 5$ (by \cite[Thm~1.4]{maulik2020picard}), the following result in particular implies the unconditional Thm \ref{thm: intro GR}.

\begin{theorem}
\label{thm: GR}
Let $k$ be a perfect field of characteristic $p > 2$. Let $F$ be a finite extension of $K = W[1/p]$. Let $X_F$ be a K3 surface over $F$ equipped with a quasi-polarization $\xi$ of degree $2d$ with $p \nmid d$. Suppose that the $\Gal_F$-action on $\bH^2_\et(X_{\bar{F}}, \bQ_\ell)$ is potentially unramified for some $\ell \neq p$. Then we have: 
\begin{enumerate}[label=\upshape{(\alph*)}]
    \item $\H^2_\et(X_{\bar{F}}, \bA_f^p)$ and $\H^2_\et(X_{\bar{F}}, \bQ_p)$ are potentially unramified and crystalline respectively. 
    \item If $\H^2_\et(X_{\bar{F}}, \bQ_p)$ is crystalline, then $\ID_\cris(\H^2_\et(X_{\bar{F}}, \bQ_p))$ is a K3 crystal. 
    \item Suppose that the hypothesis of $(b)$ is satisfied and $\ID_\cris(\H^2_\et(X_{\bar{F}}, \bQ_p))$ is a K3 crystal of height $i$. If Conj.~\ref{conj: HO for i} holds for $i$ or if $i = \infty$, then $X_F$ has potential good reduction.
\end{enumerate}
\end{theorem}
We recall that $X_F$ as above is said to have potential good reduction if, up to replacing $F$ by a finite extension, there exists a smooth proper algebraic space $\cX$ over $\cO_F$, whose special fiber is a K3 surface over $k$ and whose generic fiber is $X_F$ (cf. \cite[Def.~2.1]{LMGR}). 

\begin{proof}
(a)(b) Up to replacing $F$ by a finite extension, we may equip $(X, \xi)$ with a $\sfK$-level structure and an orientation so that it is given by a $F$-point $s$ on $\Mod_{2d, \sfK}$, and find a lift $t \in \wt{\shS}_{\cK}(L_d)(F)$ of $\rho_\sfK(s)$. Consider the abelian variety $\shA_t$. One easily adapts Deligne's argument in \cite[\S6.6]{Deligne2} to see that, up to replacing $F$ by a further extension, $\shA_t$ admits good reduction. By the extension property of the integral models, we can extend $t$ to an $\cO_F$-valued point $\tau$ on $\shS_\sfK(L_d)$. This implies both (a) and (b). 

(c) We have $\tau \tensor k \in \shS^i$. If the hypothesis is satisfied, then $\shS^i \subset \mathrm{im}(\rho_\sfK)$. Now we conclude by the \'etaleness of $\rho_\sfK$. Indeed, the global Torelli theorem implies that if two $\bC$-points of $\Mod_{2d, \sfK}$ are mapped to the same points under $\rho_\sfK$, then the K3 surfaces they correspond to are (non-canonically) isomorphic. If there is a quasi-polarized K3 surface over $k$ whose moduli point is sent to $\tau \tensor k$, then the \'etaleness of $\rho_\sfK$ tells us that there exists an $F$-point $s'$ of $\Mod_{2d, \sfK}$ such that $\rho_\sfK(s) = \rho_\sfK(s')$. Up to replacing $F$ by a finite extension, the K3 surfaces defined by $s$ and $s'$ are isomorphic. 
\end{proof}

\appendix 

\section{Some Results from Integral $p$-adic Hodge Theory}

We review some basic results in $p$-adic Hodge theory. Let $k$ be a perfect field of characteristic $p > 0$. We write $W$ for $W(k)$ and $K_0$ for $W[1/p]$. Let $K$ be a totally ramified extension of $K_0$ and let $\pi$ be a uniformizer of its ring of integers $\cO_K$.\footnote{This choice of notation is chosen to be in line with most references in $p$-adic Hodge theory. In the main text of the paper, the letters $K$ and $F$ take the role of $K_0$ and $K$ respectively. We apologize for this inconsistency of notation.} Let $G_K$ denote the absolute Galois group $\Gal_K$. Set $R := \cO_K/(p)$. 

Let $f : \cX \to \Spec \cO_K$ be a smooth and proper scheme (more generally, the following discussion applies also when $\cX$ is only a formal scheme and $\cX_K$ denotes the rigid analytic generic fiber). 
The subject of $p$-adic Hodge theory is concerned with how to recover the following tuples of data from one another under suitable assumptions: 
\begin{itemize}
    \item[(A)] The $\bZ_p$-module $\H^i_\et(\cX_{\bar{K}}, \bZ_p)$ equipped with a $G_K$-action.
    \item[(B)] The $F$-crystal $\bR^i f_{R, \cris *} \cO_{X_R}$ over $\Cris(R/W)$ together with the filtered $\cO_K$-module $\H^i_\dR(\cX / \cO_K)$. 
    \item[(B')] The $F$-crystal $\H^i_\cris(\cX_k / W)$ together with the filtered $\cO_K$-module $\H^i_\dR(\cX/ \cO_K)$.
\end{itemize}

\begin{remark}
Let $e$ be the ramification degree of $\cO_K$ over $W$. When $e \le p - 1$, $R \cong k[\varepsilon]/\varepsilon^e$ has a PD structure, so that the category of crystals of quasi-coherent sheaves over $\Cris(R/W)$ is equivalent to that over $\Cris(k / W)$ (\cite[Cor. 6.7]{BOBook}). Therefore, under mild torsion-freeness assumptions on various cohomology modules of $\cX$, (B) and (B') are equivalent data. Moreover, as $\cO_K$ is a PD thickening of $W$, the crystalline de-Rham comparison theorem gives us a canonical isomorphism 
$$ \H^i_\cris(\cX_k / W) \tensor_W \cO_K \cong \H^i_\dR(\cX / \cO_K). $$
When $e > p - 1$, then (B) contains strictly more information than (B'). The above isomorphism no longer holds integrally in general. However, there is still a canonical isomorphism after inverting $p$: 
$$ \H^i_\cris(\cX_k / W) \tensor_W K \cong \H^i_\dR(\cX / \cO_K) \tensor_{\cO_K} K. $$
This isomorphism is often called the Berthelot-Ogus isomorphism as it was first introduced in \cite{BO}. Below we will often make use of this isomorphism implicitly. Note that in the above isomorphisms, the LHS only depends on the special fiber $\cX_k$, whereas the RHS is equipped with the additional data of a Hodge filtration, which in general depends on the lifting $\cX$ of $\cX_k$. 
\end{remark}

Here is an overview of the relationship between the above tuples: The classical (rational) $p$-adic comparison isomorphisms tell us how to recover (A) and (B') from one another after inverting $p$. Integral $p$-adic Hodge theory (e.g., Bhatt-Morrow-Scholze's seminal paper \cite{BMS}) tells us how to recover (B) from (A). For our purposes, we are mainly concerned with how to recover (A) from (B). Roughly speaking, the way to do this is to evaluate the $F$-crystal $\bR^i f_{R, \cris *} \cO_{X_R}$ on a certain PD-thickening $S$ of $R$ ($S$ is often called Breuil's $S$-ring), so that we obtain an $S$-module. This $S$-module is equipped with a Frobenius action from the F-crystal structure on $\bR^i f_{R, \cris *} \cO_{X_R}$, and is moreover equipped with a filtration which absorbs the data of the Hodge filtration on $\H^i_\dR(\cX/\cO_K)$. The main result of Cais-Liu's work \cite{CaisLiu} tells us that by applying a certain functor (denoted by $T_\cris$ below) to this $S$-module, we recover (A). Of course, \cite{BMS} already treats the relationship between (A) and (B), but the conclusions there are packaged in a more abstract way. 

\subsection{After inverting $p$} Let $\MF^{\varphi, N}_K$ denote the category of filtered $(\varphi, N)$-modules. An object of this category is a $K_0$-vector space $D$ which is equipped with 
\begin{itemize}
    \item a Frobenius semilinear injection $\varphi : D \to D$;
    \item a $K_0$-linear map $N : D \to D$ such that $N \varphi = p N \varphi$;
    \item a descending filtration on $D_K$ such that $\Fil^i D_K = D_K$ for $i \ll 0$ and $\Fil^i D_K = 0$ for $i \gg 0$. 
\end{itemize}
Let $\MF^{\varphi}_K$ denote the subcategory with $N = 0$. The motivation to consider this category is that the data in (A) is naturally an object in $\MF^\varphi_K$ after inverting $p$ as there is a canonical Berthelot-Ogus isomorphism $\H^i_\cris(\cX_k / W) \tensor_W K \cong \H^i_\dR(\cX / \cO_K) \tensor_{\cO_K} K$. We will use this isomorphism repeatedly without explicitly mentioning it. We remark that in most references the operator $N$ as in $\MF^{\varphi, N}_K$ to treat varieties with semi-stable reductions. Since we are assuming good reduction, we may restrict to considering the category $\MF^\varphi_K$. 

Let $\Rep_{{G_K}}$ denote the category of $G_K$-representations over $\bQ_p$ and let $\Rep^{\cris}_{{G_K}}$ denote the subcategory of crystalline representations. Given an object $Q \in \Rep^\cris_{G_K}$, one may define an object in $\MF^\varphi_K$ using the (covariant) Fontaine's functors $D_\cris$ and $D_\dR$, which are defined by $D_\cris(Q) = (Q \tensor_{\bQ_p} B_\cris)^{G_K}$ and $D_\dR(Q) = (Q \tensor_{\bQ_p} B_\dR)^{G_K}$. The pair $(D_\cris(Q), D_\dR(Q))$ are equipped with a Frobenius action and filtrations respectively, and hence define an object in $\MF^\varphi_K$. We abusively denote the resulting functor $\Rep_{G_K} \to \MF^\varphi_K$ also by $D_\cris$. We define a functor from the essential image of $D_\cris$ to $\Rep_{G_K}^\cris$ by $V_\cris = \Fil^0(D \tensor_{K_0} B_\cris)^{\varphi = 1}$. There is an equality of sub-$\bQ_p$-modules 
$$ Q = V_\cris(D_\cris(Q)) $$
of $(Q \tensor_{\bQ_p} B_\cris) \tensor_{K_0} B_\cris$, which specifies a natural transformation $V_\cris \circ D_\cris \Rightarrow \id$ on $\Rep^\cris_{G_K}$.\footnote{Note that the natural transformations between two functors between $1$-categories (or locally small categories in the usual sense) do form a \textit{set} (as opposed to a groupoid), so it makes sense to specify an element in this set.} The reader may look at \cite[Part I Sec. 8, 9]{Conrad} for more details about these objects.  

By \cite[Prop.~5.1, Thm~14.6]{BMS}, there is a $p$-adic comparison isomorphism
\begin{align}
    \label{eq: p-adic comp}
    \H^i_\cris(\cX_k/W) \tensor_{W} B_\cris \iso \H^i_\et(\cX_{\bar{K}}, \bZ_p) \tensor_{\bZ_p} B_\cris
\end{align}
which respects the ${\Gal_F}$-actions and filtrations. Therefore, we obtain an isomorphism of objects in $\MF^\varphi_K$
\begin{align}
\label{eq: Dcris p-comp}
    D_\cris(\H^i_\et(\cX_{\bar{F}}, \bQ_p)) \iso (\H^i_\cris(\cX_k/W)[1/p], \H^i_\dR(\cX_K / K)). 
\end{align}
There are multiple rational $p$-adic comparison isomorphisms of the form (\ref{eq: p-adic comp}) (e.g., those constructed earlier by Faltings \cite{Faltings}, Tsuji \cite{Tsuji}, etc.). We choose to use the one from \cite{BMS} because this is the one used in \cite{CaisLiu}, to be cited below. Once we fix this choice of rational $p$-adic comparison isomorphism, then the isomorphism (\ref{eq: Dcris p-comp}) is also fixed. 

\subsection{Recovering Integral Lattices} We now explain how to recover the natural integral lattices in the objects of (\ref{eq: Dcris p-comp}) from one another. Let $\fS := W[\![u ]\!]$ and let $\theta : \fS \to \cO_K$ be the map sending $u$ to $\pi$. Let $\Rep^{\cris \circ}_{{\Gal_K}}$ denote the category of $G_K$-stable $\bZ_p$-lattices in objects of $\Rep^{\cris}_{G_K}$. Let $\fM(-)$ be the functor as in \cite[Thm 1.2.1]{int} which sends an object in $\Rep^{\cris \circ}_{G_K}$ to a Breuil-Kisin module in the sense of \cite[Thm 4.4]{BMS}, so that there exist canonical isomorphisms 
\begin{align}
    \label{eq: Kisin modules}
    \varphi^* \fM(T) \tensor_\fS K_0 \iso D_\cris(T[1/p]) \hspace{.5cm}\text{ and }\hspace{.5cm} \varphi^* \fM(T) \tensor_{\fS, \theta} K \iso D_\dR(T[1/p])
\end{align}
which preserve Frobenius actions and filtrations respectively. Then we have the following result (\cite[Thm~14.6]{BMS}).
\begin{theorem}
\label{thm: BMS integral}
Assume that $\H^i_\cris(\cX_k/W)$ and $\H^{i + 1}(\cX_k/W)$ are torsion-free. Then for $T = \H^i_\et(\cX_{\bar{K}}, \bZ_p)$ the isomorphisms (\ref{eq: Kisin modules}) map $\fM(T) \tensor_\fS W$ and $\fM(T) \tensor_{\fS, \theta} \cO_K$ isomorphically onto $\H^i_\cris(\cX_k/W)$ and $\H^i_\dR(\cX/\cO_K)$ respectively, when composed with the isomorphisms in (\ref{eq: p-adic comp}). 
\end{theorem}
We refer the reader also to \cite[Thm~3.2]{IIK} for an exposition which is closer to ours in notation. The above theorem tells us how to recover (B') from (A). Under the additional assumption $i < p - 1$, \cite[Thm~5.4]{CaisLiu} tells us how to recover (A) from (B). Before doing so we need to introduce the intermediate category of Breuil's $S$-modules, which packages the data of (B) in a different way. \\\\
\textbf{Breuil's $S$-modules} Let $S$ denote the $p$-adic completion of the PD evelope of $(\fS, \ker \theta)$. Let $S_\pi$ denote the ring $W[\![u - \pi]\!]$. Then there is an embedding $\iota: S \into S_\pi$ which sends $u$ to $u - \pi$. Let $f_\pi : S_\pi \to \cO_K$ (resp. $f_0 : S \to W$) be the projection which sends $u - \pi$ to $0$ (resp. $u$ to $0$). Then there is a commutative diagram of $W$-algebras 
\begin{center}
    \begin{tikzcd}
    S \arrow{r}{\iota} \arrow{d}[swap]{f_0} & S_\pi \arrow[]{d}{f_\pi} \\
    W \arrow[]{r}{} & \cO_K
    \end{tikzcd}. 
\end{center}

In \cite{Breuiltrans}, the above ring $S$ is denoted by $S^0_{\mathrm{min}}$. The letter $S$ in \textit{loc. cit.} denotes a certain extension of $S^0_{\mathrm{min}, K_0}$. For our purposes, one may simply take $S = S^0_{\mathrm{min}, K_0}$ when reading \cite{Breuiltrans}. The letter $S$ in our notation is in line with \cite{CaisLiu} and \cite{Liu}.  

Let $\mathcal{MF}_{S_{K_0}}^{\varphi, N}$ denote the category of filtered $(\varphi, N)$-modules over $S_{K_0}$.\footnote{This is just the category denoted by $\mathcal{MF}(\varphi, N)$ in \cite[\S 2.2]{Liu}, except that we have not restricted to positive objects, so that we replace the condition $\Fil^0 \shD = \shD$ by $\Fil^j \shD = \shD$ for $j \ll 0$.} There is an equivalence of categories 
\begin{equation}
\label{eqn: equiv to Breuil module}
    \eta : \MF^{\varphi, N}_K \to \mathcal{MF}^{\varphi, N}_{S_{K_0}}    
\end{equation}
which sends $(D, \Fil^\bullet D_K, \varphi, N)$ to an object $(\shD, \Fil^\bullet \shD, \varphi_\shD, N_\shD)$ with $\shD = D \tensor_W S$ (\cite[p.~1215]{CaisLiu}, see also \cite[Thm~6.1.1]{Breuiltrans}). The quasi-inverse $\eta^{-1}$ is defined by $(\shD \tensor_{f_0} W, \shD \tensor_{f_\pi \circ \iota} \cO_K)$, for which the Frobenius action and filtration are inherited from those on $\shD$. There is a canonical natural transformation $\eta^{-1} \circ \eta \Rightarrow \id$ on $\MF^\varphi_{K}$ which underlies the tautological identification of modules $$(D, D_K) = (D \tensor_W S \tensor_{f_0} W, D_K \tensor_W S \tensor_{f_\pi \circ \iota} \cO_K).$$

A strongly divisible $S$-lattice (of height $r$) in an object $\shD \in \mathcal{MF}^{\varphi, N}_{S_{K}}$ with $\Fil^0 \shD = \shD$ is an $S$-lattice such that $\sM [1/p] = \shD$, $N_{\shD}(\sM) \subseteq \sM$, and $\varphi_{\shD}(\Fil^r \sM) \subseteq p^r \sM$ where $\Fil^r \sM := \sM \cap \Fil^r \shD$. Let $\mathcal{MF}^{\varphi, N}_{S}$ denote the category of strongly divisible $S$-lattices in objects of $\mathcal{MF}^{\varphi, N}_{S_{K_0}}$.  

\begin{theorem}\emph{(Liu)}
\label{thm: Liu}
Suppose that $Q \in \Rep^{\cris}_{G_K}$ has Hodge-Tate weights in $\{0, 1, \ldots, p - 2\}$. Let $\shD$ denote $\eta(Q)$. The covariant functor $T_\cris : \sM \mapsto \Fil^0(\sM \tensor_S A_\cris)^{\varphi = 1}$ defines a bijection between the set of strongly divisible $S$-lattices in $\shD$ and that of $G_K$-stable $\bZ_p$-lattices in $V_\cris(D_\cris(Q)) = Q$. 
\end{theorem}
\begin{proof}
Thm~2.3.5 of \cite{Liu} tells us that the above theorem holds for Breuil's functor $T_\sst$. The contravariant version of this functor is reviewed in \S2.2 of \textit{loc. cit.} If we use supscript (resp. subscript) ``$*$'' to indicate contravariance (resp. covariance), then $T^*_\sst(-) = T_{* \sst}((-)^\vee)$. Prop.~3.5.1 of \textit{loc. cit.} tells us that $T_\sst(\sM) = T_\cris(\sM)$ as $Q$ is crystalline. 
\end{proof}

\begin{remark}
\label{rem: Tcris}
Let $\mathcal{C}$ denote the full subcategory of $\mathcal{MF}^{\varphi, N}_S$ whose image in $\mathcal{MF}^{\varphi, N}_{S_{K_0}}$ lies in the essential image of $\MF^\varphi_K$ (as a subcategory of $\MF^{\varphi, N}_K$) under $\eta$. To sum up, we now have a commutative diagram of categories
\begin{center}
    \begin{tikzcd}
    \Rep^\cris_{G_K} \arrow[bend left = 30]{r}{D_\cris} & \MF^\varphi_K \arrow[bend left = 30]{l}{V_\cris} \arrow[hook]{r}{} & \MF^{\varphi, N}_K  \arrow[bend left = 30]{r}{\eta} & \mathcal{MF}^{\varphi, N}_{S_{K_0}} \arrow[bend left = 30]{l}{\eta^{-1}} \\
    \Rep^{\cris \circ}_{G_K} \arrow{u}{} & & \mathcal{C} \arrow{ll}{T_\cris} \arrow[hook]{r}{} & \mathcal{MF}^{\varphi, N}_{S} \arrow{u}{}
    \end{tikzcd}
\end{center}
in which the vertical arrows are given by inverting $p$. Moreover, the natural transformations $V_\cris \circ D_\cris \Rightarrow \id$ and $\eta^{-1} \circ \eta \Rightarrow \id$ are tautological. By the above theorem, $T_\cris$ is an equivalence of categories. We remark that since $A_\cris$ is a sub-$W$-algebra of $B_\cris$ and the inclusion $A_\cris \subseteq B_\cris$ respects the filtration and Frobenius structures, $T_\cris(\sM)$ is a priori a sub-$\bZ_p$-module of $V_\cris(D_\cris(Q))$. The reason that we emphasize the natural transformations used is to de-categorify the language, so that $T_\cris$, which is often stated as an equivalence of categories, is concretely an equality of sets.
\end{remark}

\begin{theorem}
\label{thm: Cais-Liu}
\emph{(Cais-Liu)} Assume that $H^i_\cris(\cX_k/W)$ and $H^{i + 1}_\cris(\cX_k/W)$ are torsion-free and $i \le p - 2$. Set $\sM := \H^i_\cris(\cX_R/S)$. Let $\fp : \sM \to \H^i_\cris(\cX_k/W)$ be the canonical projection induced by $f_0$. Let $\shD \in \mathcal{MF}^{\varphi, N}_{S_{K_0}}$ be given by the object $(\H^i_\cris(\cX_k)_K, \Fil^\bullet \H^i_\dR(\cX_K/K))$ in $\MF^{\varphi}_K$ via $\eta$. Then we have:  
\begin{enumerate}[label=\upshape{(\alph*)}]
    \item There is a canonical section $s$ to $\fp[1/p]$ such that $s$ is $\varphi$-equivariant and $s \tensor_W S$ induces an isomorphism $\sM[1/p] \iso \shD$. 
    \item Under the isomorphism in $(a)$, $\sM$ defines a strongly divisible $S$-lattice in $\shD$ and $T_\cris(\sM) = \H^i_\et(\cX_{\bar{K}}, \bZ_p)$.
\end{enumerate}
\end{theorem}
\begin{proof}
Part (a) is a variant of the Berthelot-Ogus isomorphism (\cite[Prop.~5.1]{CaisLiu}). Part (b) follows from \cite[Thm 5.4(2)]{CaisLiu} and its proof, which does proceed by reducing to proving the equality of two lattices. 

Let $T$ be an object of $\Rep^{\cris \circ}_{G_{K}}$ and let $\fM(T)$ be the Breuil-Kisin module associated to $T$. Let $M(-)$ be the functor defined by $\varphi^*(\fM(-))$. Then $\underline{\sM} (M(T))  := M(T) \tensor_\fS S$ can be equipped with additional structures so that it becomes an object in $\mathcal{MF}^{\varphi, N}_S$. The base-change-to-$S$ functor $\underline{\sM}$ used here is defined in (3.6) of \textit{loc. cit.} There is a natural isomorphism $\underline{\sM}(M(T))[1/p] \iso \eta(D_\cris(T[1/p]))$ which lifts the isomorphism $M(T) \tensor_\fS K_0 \iso D_\cris(T[1/p])$ in (\ref{eq: Kisin modules}). Moreover, $T_\cris$ sends the strongly divisible $S$-lattice $\underline{\sM}(M(T))$ to $T$. The reader may also check out the proof of \cite[Lem.~A.3]{Maulik} for entirely similar considerations.

Now let $T$ be $\H^i_\et(\cX_{\bar{K}}, \bZ_p)$. Since $T_\cris$ estabilishes a bijection between strongly divisible $S$-lattices in $\shD$ and $G_K$-stable $\bZ_p$-lattices in $T[1/p]$, one reduces to showing an equality of $S$-lattices $\sM = \underline{\sM}(M(T))$ under the isomorphisms
$$ \underline{\sM}(M(T))[1/p] \cong \shD \cong \sM[1/p]. $$
This is the main step in the proof of \cite[Thm 5.4(2)]{CaisLiu} (see the second paragraph on page 1226).
\end{proof}

\begin{remark}
In the above setting, let $f : \cX_R \to \mathrm{Spf}(R)$ be the structure morphism and let $\bH^i_\cris(\cX_R)$ denote the F-crystal $\bR^i f_{\cris *} \cO_{\cX_R}$. Then $\H^i_\cris(\cX_R/S)$ (resp. $\H^i_\cris(\cX_R/\cO_K)$) can be viewed as a the $S$-module given by evaluating $\bH^i_\cris(\cX_R)$ on the object $S$ (resp. $\cO_K$) of $\Cris(R/W)$. The morphism $\theta : S \to \cO_K$ defines a canonical isomorphism $\theta^* \bH^i_\cris(\cX_R)_S \iso \bH^i_\cris(\cX_R)_{\cO_K}$. The lifting $\cX$ of $\cX_R$ to $\cO_K$ endows $\bH^i_\cris(\cX_R)_{\cO_K}$ with a Hodge filtration via the crystalline de-Rham comparison $\bH^i_\cris(\cX_R)_{\cO_K} \cong \H^i_\dR(\cX/\cO_K)$. The $S$-module $\bH^i_\cris(\cX_R)_S$, being an object of $\mathcal{MF}^{\varphi, N}_S$, is also equipped with a natural filtration, which maps isomorphically onto the Hodge filtration on $\bH^i_\dR(\cX/\cO_K)$. However, note that the filtration on $\bH^i_\cris(\cX_R)_S$ is defined in a more formal way, with the Hodge filtration on $\bH^i_\dR(\cX/\cO_K)$ being the key input. Namely, one first constructs $\sD$ out of $(\H^i_\cris(\cX_k/W), \H^i_\dR(\cX/\cO_K))$, and then defines a filtration on $\sM$ by intersecting with $\Fil^\bullet \sD$ under the isomorphism in part (a) of the above theorem. One naturally wonders whether this filtration has a more direct cohomological construction. This question is addressed in \cite[\S6.1]{CaisLiu}. However, we won't make use of this cohomological interpretation. 
\end{remark} 

\begin{remark}
If $\cX$ is a smooth proper scheme over $\cO_K$, or more generally a smooth proper algebraic space over $\cO_K$ whose special and generic fibers are schemes, then the above results hold for $\cX_K$ interpreted as the generic fiber in the usual sense. The point is that the analytification of the generic fiber is functorially isomorphic to the rigid analytic generic fiber of the formal completion of $\cX$ at the special fiber. The reader may look at \cite[\S11.2]{IIK} for details. 
\end{remark}

\noindent \textbf{Applications to $p$-divisible groups} Let $\sG$ be a $p$-divisible group over $\cO_K$ and assume $p \ge 3$. Let $T_p(-)$ denote the Tate module functor, $\ID(-)$ denote the contravariant Dieudonn\'e module functor and $\sG^*$ denote the Cartier dual of $\sG$. There is a $p$-adic comparison isomorphism 
\begin{align}
    \label{eq: p-div p-comp}
    \ID(\sG_k) \tensor_W B_\cris \iso T_p \sG^* (-1) \tensor_{\bZ_p} B_\cris.
\end{align}
which induces an isomorphism $D_\cris(T_p \sG^* (-1) \tensor_{\bZ_p} \bQ_p) \iso \ID(\sG_k)[1/p]$. $T_\cris(\ID(\sG_R)_S)$ recovers the $\bZ_p$-lattice $T_p \sG^*(-1)$ inside $T_p \sG^* (-1) \tensor_{\bZ_p} \bQ_p$ (\cite[Lem.~2.2.4]{KisinF-crystal}). Note that $T_p \sG^*(-1)$ is canonically isomorphic to $(T_p \sG)^\vee$.

\printbibliography


\end{document}